\documentclass[11pt,a4paper,leqno]{amsart}

\usepackage{amsmath,amsthm,amssymb,amsfonts}
\usepackage{txfonts}
\usepackage{enumitem}
\usepackage{mathrsfs}
\usepackage{hyperref}

\hypersetup{
  colorlinks=true,
  linkcolor=blue,
  citecolor=blue,
  urlcolor=blue,
  pdftitle={Boundary Harnack inequalities for Kolmogorov equations in asymptotically cylindrical Lipschitz domains},
  pdfauthor={Kaj Nystrom}
}

\newcommand{\dd}{\,\mathrm{d}}
\DeclareMathOperator{\osc}{osc}

\setlist[itemize]{leftmargin=2em}

\calclayout

\allowdisplaybreaks
\numberwithin{equation}{section}

\newtheorem{theorem}{Theorem}[section]
\newtheorem{proposition}{Proposition}[section]
\newtheorem{lemma}{Lemma}[section]
\newtheorem{corollary}{Corollary}[section]
\theoremstyle{definition}
\newtheorem{definition}{Definition}
\newtheorem{remark}{Remark}[section]

\begin{document}

\title[Kolmogorov equations in asymptotically cylindrical Lipschitz domains]
{Boundary Harnack inequalities for Kolmogorov equations in asymptotically
cylindrical Lipschitz domains}
\author{Kaj Nystr\"om}
\address{Department of Mathematics, Uppsala University, Box 480\\
SE-751 06 Uppsala, Sweden}
\email{kaj.nystrom@math.uu.se}
\subjclass[2020]{Primary 35K65, 35H20; Secondary 35K70, 35B65, 35R03}
\keywords{Kolmogorov equation, boundary Harnack inequality, intrinsic
Lipschitz domain, hypoelliptic equation, compactness, rigidity.}

\begin{abstract}
We establish boundary Harnack inequalities for non-negative solutions of the
Kolmogorov equation
\[
 \mathcal Ku=\Delta_xu+x\cdot\nabla_yu-\partial_tu=0
\]
in non-characteristic intrinsic Lipschitz graph domains whose defining
functions may depend on the higher-order variable \(y_m\). Earlier theory
required independence of \(y_m\). We replace this symmetry by quantitative
asymptotic cylindricality, measured by a scale-invariant
\(C^{0,1/3}_{y_m}\) defect. The perturbative comparison theorem applies
whenever this defect is sufficiently small, and hence at all sufficiently
small scales for \(C^{0,\alpha}_{y_m}\) graphs with \(\alpha>1/3\), as well
as under a vanishing critical cylindrical modulus. We obtain local
comparability and intrinsic H\"older continuity of quotients of positive
solutions, together with a H\"older estimate for their logarithmic quotient.

The proof combines local boundary estimates, valid without
\(y_m\)-independence, with a global rigidity argument for cylindrical blow-up
limits. The central issue is uniform balance of forward and backward
reference values. If this balance fails at collapsing scales, maximal-scale
selection and a second rescaling produce a non-negative polynomial-growth
Dirichlet solution in an unbounded cylindrical domain that vanishes at the
backward reference point but not at the forward one. We exclude this
configuration using boundary-energy and mean-value estimates,
finite-dimensionality of spaces of polynomial-growth solutions, and
translation invariance in \(y_m\). The resulting finite-dimensional
translation representation yields real-analytic dependence on \(y_m\).
Propagation of zeros, analytic continuation along invariant fibres, and a
Tikhonov uniqueness argument then force the limiting solution to vanish
identically. The exponent \(1/3\) is critical because \(y_m\) has homogeneous
degree three.
\end{abstract}

\maketitle

\section{Introduction}
\label{sec:introduction}

The Kolmogorov operator
\begin{equation*}
 \mathcal K=\Delta_x+x\cdot\nabla_y-\partial_t,
 \qquad (x,y,t)\in\mathbb R^m\times\mathbb R^m\times\mathbb R,
\end{equation*}
is degenerate in the \(y\)-variables but hypoelliptic through the interaction of
the diffusion fields \(\partial_{x_i}\) with the drift
\(x\cdot\nabla_y-\partial_t\).  It is the basic kinetic diffusion introduced by
Kolmogorov \cite{Kolmogorov1934} and a canonical example in H\"ormander's theory
\cite{Hormander1967}.  Its geometry is anisotropic: the natural dilations assign
degrees one, three, and two to \(x\), \(y\), and \(t\), respectively.  Boundary
regularity must therefore be formulated in terms of the associated group law,
quasi-distance, and intrinsic Lipschitz graphs.

Boundary Harnack inequalities compare two positive solutions which vanish on
a common boundary portion. In elliptic Lipschitz and nontangentially
accessible (NTA) domains, boundary H\"older estimates, Carleson estimates,
and boundary comparison principles are fundamental tools in potential theory
and free-boundary analysis
\cite{CaffarelliFabesMortolaSalsa1981,JerisonKenig1982}. Their uniformly parabolic counterparts were developed in
\cite{Garofalo1984,FabesGarofaloSalsa1984,FabesGarofaloSalsa1986}.
These tools are central to the regularity theory for parabolic free-boundary
problems; see, for example,
\cite{AthanasopoulosCaffarelliSalsa1996a,
AthanasopoulosCaffarelliSalsa1996b,
AthanasopoulosCaffarelliSalsa1998,
CaffarelliPetrosyanShahgholian2004}.
They also play an important role in modern quantitative boundary theory for
the heat equation; see, for example,
\cite{BortzHofmannMartellNystrom2025} for the equivalence between
\(\mathrm L^p\) Dirichlet solvability and parabolic uniform rectifiability
on parabolic Lipschitz graph domains.

A specifically parabolic feature is the time orientation: forward and
backward estimates are naturally tied to distinct reference points.  For Kolmogorov
operators there is an additional obstruction.  Admissible Harnack paths control
the endpoint in the diffusive variables, while the drift transports the
higher-order variables.  If the graph is taken in the \(x_m\)-direction, the
corresponding variable \(y_m\) cannot in general be prescribed independently at
the other endpoint of the path.

The boundary theory developed by Cinti, the author, and Polidoro in
\cite{CintiNystromPolidoro2010,CintiNystromPolidoro2012,
CintiNystromPolidoro2013} provides propagation, boundary decay, and a
one-sided Carleson estimate in intrinsic Lipschitz graph domains. The author
and Polidoro subsequently proved two-sided comparison and H\"older continuity
of quotients when the defining function is independent of \(y_m\)
\cite{NystromPolidoro2016}. Litsg{\aa}rd and the author later developed a
potential theory for divergence-form Kolmogorov operators with rough
coefficients \cite{LitsgardNystrom2022}. Of particular importance here, their
boundary H\"older estimate, one-sided Carleson estimate, and localized weak
comparison principle remain valid for graphs with full \(y_m\)-dependence:
the cylindricality assumption is not used up to and including the localized
weak comparison lemma. Their complete two-sided comparison theorem does,
however, require independence of \(y_m\), precisely because the argument must
compare forward and backward reference values uniformly across scales.

The present proof separates these two parts of the theory. From
\cite{CintiNystromPolidoro2010,CintiNystromPolidoro2012,
CintiNystromPolidoro2013,NystromPolidoro2016}, we use the intrinsic graph
geometry, propagation and Harnack-chain constructions, boundary decay, the
one-sided Carleson estimate, and two-sided comparison in exactly cylindrical
domains. From \cite{LitsgardNystrom2022}, we use the local energy framework,
boundary H\"older estimates, and, crucially, the localized weak comparison
principle in graphs with full \(y_m\)-dependence. These previously available
local ingredients are recalled below in the precise form required. The new
work begins with the scale-uniform reference-point balance and culminates in
a global rigidity theorem for polynomial-growth solutions in unbounded
cylindrical domains.

The principal result of the paper is a two-sided boundary comparison theorem
for graph domains with genuine \(y_m\)-dependence, in which exact
\(y_m\)-translation invariance is replaced by asymptotic cylindricality under
intrinsic blow-up. We consider
\[
 \Omega_\psi
 =
 \{(x',x_m,y',y_m,t):x_m>\psi(x',y',y_m,t)\},
\]
where \(\psi\) is intrinsic Lipschitz with constant \(M\). A model sufficient
condition is
\begin{equation*}
 |\psi(x',y',y_m,t)-\psi(x',y',\widetilde y_m,t)|
 \leq
 L|y_m-\widetilde y_m|^\alpha,
 \qquad
 \alpha\in(1/3,1].
\end{equation*}
Under an intrinsic blow-up at scale \(r\), a normalized displacement
\(\widehat y_m-\widetilde{\widehat y}_m\) corresponds to the physical
displacement  $r^3(\widehat y_m-\widetilde{\widehat y}_m)$, whereas the graph height is divided by \(r\). Consequently, the
\(y_m\)-oscillation of the rescaled graph is bounded by
\[
 r^{-1}L
 \bigl(r^3|\widehat y_m-\widetilde{\widehat y}_m|\bigr)^\alpha
 =
 Lr^{3\alpha-1}
 |\widehat y_m-\widetilde{\widehat y}_m|^\alpha.
\]
Thus the cylindrical defect tends to zero on every fixed normalized box as
\(r\downarrow0\) whenever \(\alpha>1/3\). The analytic result is formulated
intrinsically in terms of the scale-invariant
\(C^{0,1/3}_{y_m}\) defect: the perturbative comparison theorem applies
whenever this defect is sufficiently small. Hence the original domain need
not be cylindrical at any positive scale, even though every tangent domain is
cylindrical.

Our first main result, Theorem~\ref{thm:perturbative-BHI}, is a scale-one
perturbative boundary comparison theorem. If the localized, scale-invariant
\(C^{0,1/3}_{y_m}\) defect of the graph is sufficiently small and the forward
and backward reference values of each solution are comparable at the outer
scale, then the quotient of two positive solutions is comparable and
intrinsically H\"older continuous in a smaller box. The constants depend only
on the dimension, the intrinsic Lipschitz constant, and the outer reference
imbalance. Scaling yields Theorem~\ref{thm:main-BHI}: under
\(C^{0,\alpha}_{y_m}\) regularity with \(\alpha>1/3\), the conclusion holds at every scale
\[
 r\leq r_*\asymp L^{-1/(3\alpha-1)}.
\]
Corollary~\ref{cor:logarithmic-BHI} gives the corresponding estimate for the
logarithmic quotient. Theorem~\ref{thm:general-modulus} gives the principal
geometric formulation, requiring only that the critical cylindrical modulus
vanish at small scales.

The central difficulty is to establish a uniform balance between the forward
and backward reference values. If such balance fails, compactness first
excludes bad scales bounded away from zero, because their limits are governed
by the existing cylindrical comparison theorem. The remaining bad scales
must therefore collapse to zero. We replace them by maximal bad scales and
normalize the corresponding solutions at their forward reference points. The
one-sided Carleson estimate then gives polynomial growth on an expanding
family of intrinsic boxes. After passing to a subsequence, we obtain an
unbounded graph domain that is exactly cylindrical in \(y_m\) and a
non-negative \(\mathcal K\)-harmonic function \(W\), vanishing on the
boundary, such that
\[
 W(A^-_{1,\Lambda})=0,
 \qquad
 W(A^+_{1,\Lambda})=1.
\]
The global rigidity theorem proves that no such configuration can exist.

The proof of this rigidity theorem combines causal propagation,
finite-dimensionality, translation symmetry, and forward uniqueness. Bony's
strong minimum principle \cite{Bony1969} propagates an interior zero only
through the causally accessible set and therefore does not, by itself, force
\(W\) to vanish identically. An admissible-curve construction adapted from
\cite{NystromPolidoro2016}, using the exact \(y_m\)-invariance of the limiting
domain, shows that the accessible set contains a half-ray in every
\(y_m\)-fibre at each earlier time. Here a fibre is obtained by fixing
\((x',x_m,y',t)\) and varying only \(y_m\). Bony's principle therefore gives
vanishing on one half-ray of every such fibre.

To remove this remaining directional obstruction, we prove that the space of
Dirichlet \(\mathcal K\)-harmonic functions of any fixed polynomial growth
order is finite-dimensional. The proof is an intrinsic Kolmogorov adaptation
of the dimension-counting method for polynomial-growth harmonic functions
\cite{ColdingMinicozzi1997,Kleiner2010}. Translation in \(y_m\) consequently
acts on this space through a finite-dimensional one-parameter matrix group.
Every \(y_m\)-fibre is therefore an exponential polynomial and, in particular,
real analytic. Vanishing on a half-ray then implies vanishing on the entire
fibre. Finally, an explicit Tikhonov-type barrier gives uniqueness forward
from the resulting zero Cauchy section. This contradiction establishes the
uniform reference balance.

Once reference balance is available, the localized weak comparison principle
from \cite{LitsgardNystrom2022} yields a one-step boundary oscillation
contraction for the quotient. An interior contraction repairs the asymmetry
created by the time direction of the interior Harnack inequality, again using
reference balance to compare the denominator at past and future points. A
nearest-boundary argument then patches the interior and boundary estimates and
gives the asserted intrinsic H\"older continuity.

\subsection{Related work}

The boundary theory for kinetic Fokker-Planck equations, with the homogeneous
Kolmogorov equation \(\mathcal Ku=0\) as its basic prototype, is developing
rapidly; see, among recent contributions,
\cite{Silvestre2022,Zhu2024,RosOtonWeidner2025,KimWeidnerGrazing2026}. Very recently, Kim and Weidner \cite{KimWeidner2026} proved Hopf-type
boundary lower estimates and higher-order boundary Harnack estimates for
kinetic Fokker-Planck equations with absorbing boundary conditions. Their
geometric setting is, however, fundamentally different from the one considered
here. Their variables \((v,x,t)\) correspond to our \((x,y,t)\), respectively.
In the notation of the present paper, they study space-time cylinders of the form
\[
 (x,y,t)\in\mathbb R^m\times\Omega\times(0,T],
\]
where \(\Omega\subset\mathbb R^m\) is a smooth domain in the degenerate
\(y\)-variables, and impose zero data only on the incoming part of the kinetic
boundary. The boundary is therefore invisible to the diffusion in \(x\), and
its interaction with the equation is governed by transport; the incoming,
outgoing, and grazing portions have genuinely different behavior. In
particular, Kim and Weidner establish optimal higher-order regularity of
quotients near the grazing set, obtaining \(C^{3/2}\) regularity in the
presence of source terms and \(C^{1,1}\) regularity in the homogeneous case,
while boundary comparison may fail in the incoming boundary layer. Their
proof is based on an explicit grazing profile, propagation of positivity,
a kinetic Hopf lemma, and fine boundary expansions combined with Schauder and
Liouville estimates. By contrast, the boundary considered here is a
non-characteristic graph in the diffusive \(x_m\)-direction, zero data are
prescribed on the entire relevant graph portion, and the principal difficulty
is the loss of exact cylindricality in \(y_m\). Accordingly, our argument is
based on compactness, uniform reference-point balance, global cylindrical
rigidity, and oscillation reduction. The two theories are therefore
complementary and illustrate the substantially different forms taken by
boundary Harnack principles in distinct kinetic boundary geometries.

\subsection{Organization of the paper}

The paper is organized as follows.
Section~\ref{sec:geometry} introduces the intrinsic geometry, graph domains,
reference points, and cylindrical modulus. Section~\ref{sec:main-results}
states the main results. Section~\ref{sec:compactness} develops the
compactness framework and the expanding-cylinder growth estimate required at
collapsing bad scales. Section~\ref{sec:global-rigidity} proves the global
cylindrical rigidity theorem. Section~\ref{sec:oscillation-reduction} uses
this rigidity to establish uniform reference-point balance and completes the
proof through localized weak comparison and oscillation reduction.
Section~\ref{sec:conclusion} summarizes the role of the critical exponent and
records directions in which the rigidity method may be useful. We also briefly explain how the boundary comparison theorem, together with
scale-decaying estimates for the commutator-generated inhomogeneous equations,
yields \(C_{\mathcal K}^{1,\beta}\)-regularity of the regular part of the free
boundary in the Kolmogorov obstacle problem. This completes the program
initiated in \cite{FrentzNystromPascucciPolidoro2010} and continued in
\cite{Bowman2025}. A complete proof is given in the author's preprint
\cite{NystromHigherRegularity2026}.

\section{Geometric setting}
\label{sec:geometry}

In this section we introduce the geometric and analytic framework used
throughout the paper. The main results are stated separately in
Section~\ref{sec:main-results}. Throughout, \(m\geq1\), \(N=2m\), $X=(x,y)\in\mathbb R^m\times\mathbb R^m$, and
\begin{equation}
\label{eq:K-operator}
 \mathcal K
 =
 \sum_{i=1}^m\partial_{x_i x_i}
 +
 \sum_{i=1}^m x_i\partial_{y_i}
 -
 \partial_t.
\end{equation}

The purpose of the paper is to establish a boundary comparison principle in
intrinsic Lipschitz graph domains whose defining functions may depend on the
higher-order variable \(y_m\) paired with the graph direction \(x_m\). The
scale-one perturbative theorem requires the localized, scale-invariant
\(C^{0,1/3}_{y_m}\) defect to be sufficiently small. In the asymptotically
cylindrical setting, this smallness follows because the \(y_m\)-dependence
vanishes under intrinsic blow-up. The model sufficient condition is
H\"older continuity of order \(\alpha>1/3\) in \(y_m\).

Unless otherwise stated, a solution of \(\mathcal Ku=0\) in an open set
\(G\subset\mathbb R^{2m+1}\) is understood to be a weak solution in the
natural local energy class. More precisely, we require $u\in \mathrm L^2_{\mathrm{loc}}(G)$, $\nabla_xu\in \mathrm L^2_{\mathrm{loc}}(G)$, and
\begin{equation}
\label{eq:weak-energy-formulation}
\int_G
\left[
\nabla_xu\cdot\nabla_x\varphi
+
u\bigl(x\cdot\nabla_y\varphi-\partial_t\varphi\bigr)
\right]
\,\dd X\,\dd t
=0
\end{equation}
for every \(\varphi\in C_c^\infty(G)\). We refer to such a function
interchangeably as a \emph{weak solution} or an \emph{energy solution}.
The terminology reflects the fact that the energy controls derivatives only
in the non-degenerate \(x\)-variables; the drift
\(x\cdot\nabla_y-\partial_t\) is interpreted distributionally. The identity
\eqref{eq:weak-energy-formulation} is equivalent to
\(\mathcal Ku=0\) in the sense of distributions.
By hypoellipticity, every such solution is smooth in the interior of \(G\).
Boundary values are understood pointwise whenever continuity up to the
relevant boundary portion is assumed. In an energy argument, however, a
Sobolev boundary value is invoked only when the solution belongs locally to
the boundary energy class. More precisely and in this context, on every bounded neighborhood
\(V\) of a non-characteristic graph portion $x_m=\psi(x',y,t)$,  we require
\[
u,\nabla_xu\in \mathrm L^2(G\cap V).
\]
By Fubini's theorem, for almost every \((y,t)\), the function
\(u(\cdot,y,t)\) then belongs to \(\mathrm H^1\) on the corresponding \(x\)-slice.
Since that slice is bounded by a Lipschitz graph in the non-degenerate
\(x\)-variables, the usual Sobolev trace theorem defines its trace on $x_m=\psi(x',y,t)$. A zero boundary value in the energy sense means that this trace vanishes for
almost every \((y,t)\).

The preceding up-to-boundary integrability is not a consequence of the local
interior energy condition alone. It must instead follow from an appropriate
boundary Caccioppoli estimate or be included among the hypotheses. Moreover,
this boundary energy condition provides neither a \(y\)-trace nor a
\(t\)-trace.

In this paper, we verify and use this energy-trace formulation only for the
\(y_m\)-independent graph domains appearing in
Definition~\ref{def:global-rigidity}. In that setting,
Proposition~\ref{prop:boundary-energy-zero-extension} establishes the required
up-to-boundary \(\mathrm L^2\)-integrability, the
\(\mathrm H_x^{-1}\)-control of the transport derivative, the vanishing
\(x\)-Sobolev trace, and the cutoff approximation needed to justify boundary
energy tests. Its proof uses the cylindrical boundary comparison and
boundary-layer estimates of Litsg{\aa}rd and the author
\cite{LitsgardNystrom2022}. We make no corresponding assertion here for graph
functions with genuine \(y_m\)-dependence.

\medskip
\noindent\textbf{Conventions.} Constants denoted by \(C\) may change from line to line
and depend only on the parameters displayed in the relevant statement. Unless otherwise indicated, all integrals over subsets of
\(\mathbb R^{2m+1}\) are taken with respect to Lebesgue measure
\(\dd x\,\dd y\,\dd t\), and the integration element is suppressed. Integration with
respect to a different measure, such as surface measure on a boundary
portion, will always be indicated explicitly.

\subsection{The underlying group structure}
\label{subsec:group-structure}

The natural dilations associated with \(\mathcal K\) are
\begin{equation}
\label{eq:K-dilations}
 \delta_r(x,y,t)=(rx,r^3y,r^2t),
 \qquad r>0.
\end{equation}
Moreover, \(\mathbb R^{N+1}\) is equipped with the group law
\begin{equation}
\label{eq:K-group-law}
 (\widetilde X,\widetilde t)\circ(X,t)
 =
 (\widetilde x+x,
  \widetilde y+y-t\widetilde x,
  \widetilde t+t).
\end{equation}
The inverse of \((X,t)=(x,y,t)\) is therefore
\begin{equation}
\label{eq:K-inverse}
 (X,t)^{-1}=(-x,-y-tx,-t).
\end{equation}
Both the vector fields \(\partial_{x_1},\ldots,\partial_{x_m}\) and the drift
\(Y=x\cdot\nabla_y-\partial_t\) are left invariant with respect to
\eqref{eq:K-group-law}, while
\(\mathcal K\) is homogeneous of degree two with respect to
\eqref{eq:K-dilations}.

For points $P=(X,t)=(x,y,t)$, $\widetilde P=(\widetilde X,\widetilde t)
 =(\widetilde x,\widetilde y,\widetilde t)$, we use the homogeneous quasi-norm
\begin{equation}
\label{eq:K-quasinorm}
 \lVert P\rVert_{\mathcal K}
 =
 |x|+|y|^{1/3}+|t|^{1/2}
\end{equation}
and the associated left-invariant quasi-distance
\begin{equation}
\label{eq:K-distance}
 d_{\mathcal K}(P,\widetilde P)
 =
 \bigl\lVert\widetilde P^{-1}\circ P\bigr\rVert_{\mathcal K}.
\end{equation}
The quasi-distance \(d_{\mathcal K}\) is not symmetric. We therefore also
introduce the symmetrized quasi-distance
\begin{equation}
\label{eq:symmetric-K-distance}
 d_{\mathcal K}^{\mathrm s}(P,\widetilde P)
 =
 \max\bigl\{
 d_{\mathcal K}(P,\widetilde P),
 d_{\mathcal K}(\widetilde P,P)
 \bigr\}.
\end{equation}
There exists \(C=C(m)\geq1\) such that
\[
 d_{\mathcal K}(P,\widetilde P)
 \leq
 d_{\mathcal K}^{\mathrm s}(P,\widetilde P)
 \leq
 C d_{\mathcal K}(P,\widetilde P).
\]
Thus all H{\"o}lder estimates below are unchanged, up to their constants, if
\(d_{\mathcal K}\) is replaced by \(d_{\mathcal K}^{\mathrm s}\). For a set
\(E\subset\mathbb R^{2m+1}\), we use the conventions
\[
 d_{\mathcal K}(P,E)
 =
 \inf_{\widetilde P\in E}d_{\mathcal K}(P,\widetilde P),
 \qquad
 d_{\mathcal K}^{\mathrm s}(P,E)
 =
 \inf_{\widetilde P\in E}
 d_{\mathcal K}^{\mathrm s}(P,\widetilde P).
\]

Write
\[
 x=(x',x_m)\in\mathbb R^{m-1}\times\mathbb R,
 \qquad
 y=(y',y_m)\in\mathbb R^{m-1}\times\mathbb R.
\]
Fix \(M\geq1\), which throughout denotes an upper bound for the intrinsic
Lipschitz constant of the graph functions considered below. Let
\(c_M=c(m,M)\geq1\) be a geometric constant. Its precise size will be fixed
below, and it may be enlarged a finite number of times, without changing
notation, to satisfy the geometric requirements imposed later.

We introduce the normalized box
\begin{align}
 Q_{M,r}
 =\bigl\{(x',x_m,y',y_m,t):{}\!
 &|x_i|<r,\quad i=1,\ldots,m-1,
 \nonumber\\
 &|y_i|<c_Mr^3,\quad i=1,\ldots,m,
 \nonumber\\
 &|t|<2r^2,\quad |x_m|<c_Mr\bigr\},
\label{eq:normalized-box}
\end{align}
and, for \(P_0=(X_0,t_0)\in\mathbb R^{2m+1}\), set
\begin{equation}
\label{eq:translated-box}
 Q_{M,r}(P_0)=P_0\circ Q_{M,r}.
\end{equation}

\subsection{Intrinsic graph domains}
\label{subsec:graph-domains}

Let
\[
 \zeta=(x',y',y_m,t)
 \in\mathbb R^{m-1}\times\mathbb R^{m-1}\times\mathbb R\times\mathbb R,
\]
and let \(\psi:\mathbb R^{N}\to\mathbb R\) be continuous. We consider the
unbounded graph domain
\begin{equation}
\label{eq:graph-domain}
 \Omega_\psi
 =
 \bigl\{(x',x_m,y',y_m,t):x_m>\psi(x',y',y_m,t)\bigr\}
\end{equation}
and its graph boundary
\begin{equation}
\label{eq:graph-boundary}
 \Delta_\psi
 =
 \bigl\{(x',x_m,y',y_m,t):x_m=\psi(x',y',y_m,t)\bigr\}.
\end{equation}
For
\(\zeta=(x',y',y_m,t)\) and
\(\widetilde\zeta=(\widetilde x',\widetilde y',\widetilde y_m,
\widetilde t)\), set
\begin{align}
 D_\psi(\zeta,\widetilde\zeta)
 ={}&|x'-\widetilde x'|
 +|y'-\widetilde y'+(t-\widetilde t)\widetilde x'|^{1/3}+
 \bigl|y_m-\widetilde y_m
 +(t-\widetilde t)\psi(\widetilde\zeta)\bigr|^{1/3}
 +|t-\widetilde t|^{1/2}.
\label{eq:graph-quasidistance}
\end{align}
This is the boundary quasi-distance obtained from the group difference after
the \(x_m\)-component has been removed.

\begin{definition}
\label{def:LipK-graph}
Let \(M\geq1\). We say that \(\psi\) is an intrinsic
\(\operatorname{Lip}_{\mathcal K}\) function with constant \(M\), with respect
to the direction \(e_m\), if
\begin{equation}
\label{eq:LipK-condition}
 |\psi(\zeta)-\psi(\widetilde\zeta)|
 \leq M D_\psi(\zeta,\widetilde\zeta)
\end{equation}
whenever \(\zeta,\widetilde\zeta\in\mathbb R^N\). We then call
\(\Omega_\psi\) an unbounded non-characteristic
\(\operatorname{Lip}_{\mathcal K}\) graph domain.
\end{definition}

We now fix the constant \(c_M\) in \eqref{eq:normalized-box}. After
translating a boundary point \(P_0\in\Delta_\psi\) to the origin and
dilating by \(r^{-1}\), the projection of the normalized box onto the graph
variables \((x',y',y_m,t)\) is
\[
 \left\{
 (x',y',y_m,t):
 |x_i|<1,\quad
 |y_i|<c_M,\quad
 |t|<2
 \right\},
\]
where \(i=1,\ldots,m-1\) in the \(x'\)-variables and
\(i=1,\ldots,m\) in the \(y\)-variables. The intrinsic Lipschitz condition
\eqref{eq:LipK-condition} bounds the absolute value of the normalized graph
height over this set by
\[
 C(m)M\bigl(1+c_M^{1/3}\bigr).
\]
Since the right-hand side grows sublinearly in \(c_M\), we may choose
\(c_M=c(m,M)\) sufficiently large that
\begin{equation}
\label{eq:box-constant-choice}
 C(m)M\bigl(1+c_M^{1/3}\bigr)
 \leq {c_M}/{2}.
\end{equation}
Thus, in the coordinates centered at \(P_0\), the graph lies below the level
\(x_m=c_Mr/2\) throughout the base of \(Q_{M,r}(P_0)\). The projection of \(Q_{M,r}(P_0)\) onto the variables
\((x',y',y_m,t)\), in these centered coordinates, is a rectangle. Moreover,
every vertical section of
\(\Omega_\psi\cap Q_{M,r}(P_0)\) contains the horizontal level
\[
 x_m={3c_Mr}/{4}.
\]
Any point of the truncation can therefore be joined by a vertical segment to
this horizontal slice, and any two points in the slice can be joined by a
line segment over the rectangular base. Consequently,
\begin{equation}
\label{eq:connected-graph-truncation}
 \Omega_\psi\cap Q_{M,r}(P_0)
 \quad\text{is path connected whenever }
 P_0\in\Delta_\psi\text{ and }r>0.
\end{equation}

The term containing \(\psi(\widetilde\zeta)\) in
\eqref{eq:graph-quasidistance} is forced by left translation invariance. In
particular, it cannot be replaced by \(|y_m-\widetilde y_m|^{1/3}\) when \(t\ne
\widetilde t\). Notice also that \eqref{eq:LipK-condition} gives at most the
scale-critical regularity \(C^{0,1/3}\) in \(y_m\). We impose the following
additional, supercritical regularity.

\begin{definition}
\label{def:asymptotically-cylindrical}
Let \(\alpha\in(1/3,1]\) and \(L\geq0\). We say that an intrinsic
\(\operatorname{Lip}_{\mathcal K}\) function \(\psi\) is
\((L,\alpha)\)-asymptotically cylindrical in the \(y_m\)-direction if
\begin{equation}
\label{eq:ym-holder}
 \bigl|\psi(x',y',y_m,t)-
 \psi(x',y',\widetilde y_m,t)\bigr|
 \leq L|y_m-\widetilde y_m|^\alpha
\end{equation}
uniformly in \(x',y',t,y_m,\widetilde y_m\).
\end{definition}

The terminology in Definition~\ref{def:asymptotically-cylindrical} is justified
by the blow-up calculation below. It will also be useful to isolate the precise
scale-dependent quantity measured by \eqref{eq:ym-holder}. Define
\begin{equation}
\label{eq:cylindrical-modulus}
 \eta_\psi(r)
 =
 \sup_{\substack{x'\in\mathbb R^{m-1},\ y'\in\mathbb R^{m-1},\ y_m,t\in\mathbb R\\ 0<|h|\leq r^3}}
 \frac{
 |\psi(x',y',y_m+h,t)-\psi(x',y',y_m,t)|
 }{|h|^{1/3}},
 \qquad r>0.
\end{equation}
Then \eqref{eq:ym-holder} implies
\begin{equation}
\label{eq:cylindrical-modulus-holder}
 \eta_\psi(r)\leq Lr^{3\alpha-1}.
\end{equation}
Thus \(\eta_\psi(r)\to0\) as \(r\to0\) precisely because
\(\alpha>1/3\).

To make this observation invariant, fix
\[
 P_0=(x'_0,x_{0,m},y'_0,y_{0,m},t_0)\in\Delta_\psi,
 \qquad
 x_{0,m}=\psi(x'_0,y'_0,y_{0,m},t_0),
\]
and write a point in the translated and dilated coordinates as
\(P=P_0\circ\delta_r(\widehat x,\widehat y,\widehat t)\). In these coordinates
the rescaled boundary is the graph of
\begin{equation}
\label{eq:rescaled-graph}
\begin{aligned}
 \psi_{P_0,r}(\widehat x',\widehat y',\widehat y_m,\widehat t)
 =\frac1r\Bigl[\psi\bigl(&x'_0+r\widehat x',
 y'_0+r^3\widehat y'-r^2\widehat t x'_0,y_{0,m}+r^3\widehat y_m-r^2\widehat t x_{0,m},
 t_0+r^2\widehat t\bigr)-x_{0,m}\Bigr].
\end{aligned}
\end{equation}
The class in Definition~\ref{def:LipK-graph} is invariant under this operation.
Furthermore, \eqref{eq:ym-holder} gives
\begin{align}
 &\bigl|
 \psi_{P_0,r}(\widehat x',\widehat y',\widehat y_m,\widehat t)
 -
 \psi_{P_0,r}(\widehat x',\widehat y',\widetilde{\widehat y}_m,
 \widehat t)
 \bigr|\leq
 Lr^{3\alpha-1}
 |\widehat y_m-\widetilde{\widehat y}_m|^\alpha.
\label{eq:rescaled-ym-holder}
\end{align}
Consequently, every locally uniform limit of a sequence of intrinsic blow-ups is
independent of \(y_m\). At the critical exponent \(\alpha=1/3\), the right-hand
side of \eqref{eq:rescaled-ym-holder} is scale invariant and this conclusion is
no longer available.

For later use, we record the localized, invariant form of the cylindrical
defect. For \(z=(z_1,\ldots,z_k)\in\mathbb R^k\), write
\[
 |z|_\infty=\max_{1\leq i\leq k}|z_i|,
\]
with the convention that the norm of the empty vector is zero. For
\(R\geq1\), let
\[
 \mathcal T_{M,R}
 =
 \left\{
 (x',y',y_m,t):
 |x'|_\infty<R,\quad
 |y'|_\infty<c_MR^3,\quad
 |y_m|<c_MR^3,\quad
 |t|<2R^2
 \right\},
\]
and define
\begin{equation}
\label{eq:localized-cylindrical-defect}
 \operatorname{cyl}_{R,P_0,r}(\psi)
 =
 \sup
 \frac{
 \left|
 \psi_{P_0,r}(x',y',y_m,t)
 -
 \psi_{P_0,r}(x',y',\widetilde y_m,t)
 \right|
 }{|y_m-\widetilde y_m|^{1/3}},
\end{equation}
where the supremum is taken over pairs
\[
 (x',y',y_m,t),\quad
 (x',y',\widetilde y_m,t)
 \in\mathcal T_{M,R}
\]
with \(y_m\neq\widetilde y_m\).

\begin{lemma}
\label{lem:cylindrical-defect-rescaling}
The defect in \eqref{eq:localized-cylindrical-defect} is invariant under
nested intrinsic blow-ups. More precisely, whenever
\(\widehat P\in\Delta_{\psi_{P_0,r}}\),
\[
 \bigl(\psi_{P_0,r}\bigr)_{\widehat P,s}
 =
 \psi_{P_0\circ\delta_r(\widehat P),rs}
\]
in the corresponding coordinates, and hence
\[
 \operatorname{cyl}_{R,\widehat P,s}(\psi_{P_0,r})
 =
 \operatorname{cyl}_{R,P_0\circ\delta_r(\widehat P),rs}(\psi).
\]
Moreover, with \(C_M=(2c_M)^{1/3}\),
\begin{equation}
\label{eq:localized-defect-modulus-bound}
 \operatorname{cyl}_{R,P_0,r}(\psi)
 \leq
 \eta_\psi(C_MRr),
\end{equation}
and, under \eqref{eq:ym-holder},
\begin{equation}
\label{eq:localized-defect-holder-bound}
 \operatorname{cyl}_{R,P_0,r}(\psi)
 \leq
 (2c_M)^{\alpha-1/3}R^{3\alpha-1}Lr^{3\alpha-1}.
\end{equation}
\end{lemma}

\begin{proof}
The identity for the rescaled graph follows by composing the two left
translations and the two Kolmogorov dilations. The equality of the defects is
then immediate from the definition. In the original coordinates, two
\(y_m\)-coordinates occurring in \(\mathcal T_{M,R}\) differ by at most
\(2c_MR^3r^3=(C_MRr)^3\). The factor \(r^{-1}\) in the rescaled graph cancels
the factor \(r\) produced by the critical power \(1/3\). This proves
\eqref{eq:localized-defect-modulus-bound} and applying \eqref{eq:ym-holder}
gives \eqref{eq:localized-defect-holder-bound}.
\end{proof}

\subsection{The Kolmogorov boundary and the Dirichlet problem}
\label{subsec:Kolmogorov-boundary}

We briefly recall the notions of propagation, Kolmogorov boundary, and
continuous Dirichlet solution used below.

Let
\[
X_j=\partial_{x_j},
\qquad j=1,\ldots,m,
\qquad
Y=x\cdot\nabla_y-\partial_t.
\]
Given a domain \(D\subset\mathbb R^{2m+1}\) and \(P\in D\), the propagation
set \(\mathcal A_P(D)\) consists of the points that can be reached from
\(P\) by an absolutely continuous curve \(\gamma:[0,S]\to D\) satisfying
\[
\dot\gamma(s)
=
\sum_{j=1}^m\omega_j(s)X_j(\gamma(s))
+
\lambda(s)Y(\gamma(s))
\]
for almost every \(s\), where the controls \(\omega_j\) are square
integrable and \(\lambda\geq0\). Allowing \(\lambda\) amounts only to a
non-negative reparametrization of the drift direction. In particular,
propagation is directional in time and as the time component of \(Y\) is \(-\partial_t\), every admissible
curve issuing from \(P=(x^0,y^0,t_0)\) satisfies
\[
\frac{\dd}{\dd r}t(\gamma(r))=-\lambda(r)\leq0.
\]
Thus \(\mathcal A_P(D)\) is a backward propagation set with respect to the
coordinate \(t\):
\[
\mathcal A_P(D)\subset D\cap\{t\leq t_0\}.
\]
This orientation is the one appearing in Bony's strong minimum principle.

The \emph{Kolmogorov boundary} of \(D\) is
\begin{equation}
\label{eq:Kolmogorov-boundary}
\partial_{\mathcal K}D
=
\bigcup_{P\in D}
\left(
\overline{\mathcal A_P(D)}\cap\partial D
\right).
\end{equation}
Thus \(\partial_{\mathcal K}D\) is the portion of the topological boundary
which is accessible from the interior along the propagation geometry of
\(\mathcal K\). It plays the role of the parabolic boundary for uniformly
parabolic equations. In general,
\[
\partial_{\mathcal K}D\subsetneq\partial D.
\]

For a bounded truncation $D_r(P_0)=\Omega_\psi\cap Q_{M,r}(P_0)$, the graph portion $\Delta_\psi\cap Q_{M,r}(P_0)$ belongs to \(\partial_{\mathcal K}D_r(P_0)\) and is non-characteristic
because the graph is taken in the non-degenerate \(x_m\)-direction. The
Kolmogorov boundary also contains the appropriate incoming portions of the
artificial boundary of \(Q_{M,r}(P_0)\). It does not contain the terminal
time face. On the artificial \(y\)-faces, whether a point belongs to
\(\partial_{\mathcal K}D_r(P_0)\) is determined by the direction of the
drift \(Y\). We shall not need an explicit enumeration of these faces;
the intrinsic definition \eqref{eq:Kolmogorov-boundary} is sufficient.

Given \(\varphi\in C(\partial_{\mathcal K}D)\), a solution of the continuous
Dirichlet problem
\begin{equation}
\label{eq:continuous-Dirichlet-problem}
\begin{cases}
\mathcal Ku=0 & \text{in }D,\\
u=\varphi     & \text{on }\partial_{\mathcal K}D,
\end{cases}
\end{equation}
means an energy solution in \(D\) which extends continuously to
\(\partial_{\mathcal K}D\) and attains the prescribed data there. No data
are imposed on
\(\partial D\setminus\partial_{\mathcal K}D\).

For the bounded intrinsic graph truncations used in this paper, every point
of \(\partial_{\mathcal K}D_r(P_0)\) is regular for the continuous
Dirichlet problem. Consequently, \eqref{eq:continuous-Dirichlet-problem}
has a unique solution for continuous boundary data. These facts, as well
as their validity for graphs with full \(y_m\)-dependence, follow from the
barrier and approximation arguments in
\cite[Section~2.4]{NystromPolidoro2016} and
\cite[Sections~4.1 and~6]{LitsgardNystrom2022}. In the present paper the
operator has constant coefficients, so no additional coefficient
regularity is required.

We shall also use the following weak maximum principle. If \(D\) is one
of the bounded truncations above and \(w\in C(\overline D)\) is an energy
subsolution satisfying $\mathcal Kw\geq0$ in $D$, $w\leq0$ on $\partial_{\mathcal K}D$, then
\begin{equation}
\label{eq:weak-maximum-principle}
w\leq0\quad\text{in }D.
\end{equation}
Equivalently, if \(u\) and \(v\) are respectively an energy subsolution and
supersolution with \(u\leq v\) on \(\partial_{\mathcal K}D\), then
\(u\leq v\) in \(D\). Boundary inequalities are interpreted pointwise when
the functions are continuous up to the boundary and in the corresponding
energy-trace sense otherwise.

\subsection{Reference points}
\label{subsec:reference-points}

The time-oriented geometry of \(\mathcal K\) requires both forward and backward
reference points. Given \(r>0\) and \(\Lambda>0\), let
\begin{align}
 A^+_{r,\Lambda}
 &=
 \bigl(0',\Lambda r,0',-\tfrac23\Lambda r^3,r^2\bigr),
 \nonumber\\
 A_{r,\Lambda}
 &=
 \bigl(0',\Lambda r,0',0,0\bigr),
 \label{eq:reference-points}\\
 A^-_{r,\Lambda}
 &=
 \bigl(0',\Lambda r,0',\tfrac23\Lambda r^3,-r^2\bigr).
 \nonumber
\end{align}
Here the entries are ordered as \((x',x_m,y',y_m,t)\). For
\(P_0\in\Delta_\psi\), define
\begin{equation}
\label{eq:translated-reference-points}
 A^\pm_{r,\Lambda}(P_0)=P_0\circ A^\pm_{r,\Lambda},
 \qquad
 A_{r,\Lambda}(P_0)=P_0\circ A_{r,\Lambda}.
\end{equation}

We now fix the constants used in the reference-point construction. Let
\(C_*=C_*(m)\) be a constant for which the intrinsic graph estimate used
below holds, and choose \(\Lambda=\Lambda(m,M)\geq1\) so that
\begin{equation}
\label{eq:reference-Lambda-choice}
 C_*M\bigl(1+\Lambda^{1/3}\bigr)
 \leq\frac{\Lambda}{4}.
\end{equation}
Such a choice is possible because the left-hand side grows like
\(\Lambda^{1/3}\), whereas the right-hand side grows linearly. We fix this
value of \(\Lambda\) for the remainder of the paper.

We then make the final choice of the box constant \(c_M=c(m,M)\), requiring
that both \eqref{eq:box-constant-choice} and
\begin{equation}
\label{eq:reference-box-constant}
 c_M\geq C_0(1+\Lambda)
\end{equation}
hold, where \(C_0=C_0(m)\) is sufficiently large. Since
\[
 Q_{M,2r}(P_0)=P_0\circ Q_{M,2r},
\]
the explicit coordinates in \eqref{eq:reference-points} give
\[
 |\Lambda r|<2c_Mr,
 \qquad
 \frac23\Lambda r^3<8c_Mr^3,
 \qquad
 r^2<8r^2.
\]
Thus all three reference points belong to \(Q_{M,2r}(P_0)\).

To verify that they also lie above the graph, translate \(P_0\) to the
origin and dilate by \(r^{-1}\). The rescaled graph passes through the
origin. At the two time-shifted reference points, the intrinsic Lipschitz
condition gives
\[
 \left|
 \psi_{P_0,r}
 \bigl(0',0',\mp\tfrac23\Lambda,\pm1\bigr)
 \right|
 \leq
 C_*M\bigl(1+\Lambda^{1/3}\bigr)
 \leq\frac{\Lambda}{4},
\]
whereas their \(x_m\)-coordinate is \(\Lambda\). The central reference point
lies above the origin of the graph for the same reason. Consequently,
\begin{equation}
\label{eq:reference-points-in-box}
 A^+_{r,\Lambda}(P_0),\quad
 A_{r,\Lambda}(P_0),\quad
 A^-_{r,\Lambda}(P_0)
 \in
 \Omega_\psi\cap Q_{M,2r}(P_0)
\end{equation}
uniformly in \(P_0\in\Delta_\psi\) and \(r>0\).

We shall repeatedly use the following quantitative form of the local
admissible-curve geometry. It is stated here in order to make explicit the
endpoint control, including the endpoint in the distinguished variable
\(y_m\).

\begin{lemma}
\label{lem:controlled-local-chains}
With \(\Lambda\) and \(c_M\) fixed as above, there are constants
\(a_0\geq1\) and \(c_0,c_1,\gamma\in(0,1)\), depending only on \(m\) and
\(M\), such that the following statements hold.

\begin{enumerate}[label=\rm(\roman*)]
\item The three points \(A^+_{r,\Lambda}(P_0)\),
\(A_{r,\Lambda}(P_0)\), and \(A^-_{r,\Lambda}(P_0)\) can be joined by
controlled admissible curves contained in
\(\Omega_\psi\cap Q_{M,2r}(P_0)\). The resulting Harnack chains have a
uniformly bounded number of intrinsic cylinders with radii comparable to
\(r\).

\item Every
\(P\in\Omega_\psi\cap Q_{M,c_0r}(P_0)\) can be connected to both reference
points by admissible curves in the backward propagation direction: there is
a curve from \(A^+_{r/a_0,\Lambda}(P_0)\) to \(P\), and one from \(P\) to
\(A^-_{r/a_0,\Lambda}(P_0)\). Both curves are contained in
\(\Omega_\psi\cap Q_{M,r}(P_0)\). Away from an endpoint neighborhood of
\(P\), they remain at distance comparable to \(r\) from the graph and can
be covered by a uniformly bounded number of intrinsic cylinders with radii
comparable to \(r\). No uniform lower bound on the distance of \(P\) itself
from the graph is asserted.

\item Let
\[
 P=(x',x_m,y',y_m,t)\in\Omega_\psi,
 \qquad
 P^0=P^\partial
 =
 \bigl(x',\psi(x',y',y_m,t),y',y_m,t\bigr),
\]
and put
\[
 d=d_{\mathcal K}^{\mathrm s}(P,P^0).
\]
If \(0<s\leq c_0d\) and
\[
 P_s^\pm=P\circ(0,0,\pm\gamma s^2),
\]
then there is a controlled admissible curve from
\(A^+_{c_1d,\Lambda}(P^0)\) to \(P_s^+\), and one from \(P_s^-\) to
\(A^-_{c_1d,\Lambda}(P^0)\). Both curves remain at distance comparable to
\(d\) from the graph and give Harnack chains of uniformly bounded length
whose radii are comparable to \(d\).
\end{enumerate}
\end{lemma}

\begin{proof}
We first recall the explicit control underlying the construction. Given
\(P=(x^0,y^0,t_0)\) and
\(\widetilde P=(\widetilde x,\widetilde y,\widetilde t)\), with
\(\widetilde t<t_0\), put $T=t_0-\widetilde t$ and
\[
 b=
 \frac{6}{T^3}
 \left(
 \widetilde y-y^0-\frac{T}{2}(x^0+\widetilde x)
 \right).
\]
Then
\begin{align*}
 x(\tau)
 &=
 x^0+\frac{\tau}{T}(\widetilde x-x^0)
 +b\tau(T-\tau),\\
 y(\tau)
 &=
 y^0+\int_0^\tau x(\sigma)\,\dd\sigma,\\
 t(\tau)
 &=
 t_0-\tau,
 \qquad 0\leq\tau\leq T,
\end{align*}
is an admissible curve with the prescribed endpoints. Indeed,
\[
 y(T)
 =
 y^0+\frac{T}{2}(x^0+\widetilde x)+\frac{T^3}{6}b
 =
 \widetilde y.
\]
In particular, the formula fixes every component of \(\widetilde y\),
including \(\widetilde y_m\); no uncontrolled tangential displacement is
hidden in the construction.

If
\[
 T\asymp r^2,
 \qquad
 |\widetilde x-x^0|\lesssim r,
 \qquad
 \left|
 \widetilde y-y^0-\frac{T}{2}(x^0+\widetilde x)
 \right|
 \lesssim r^3,
\]
then $|b|\lesssim r^{-3}$ and hence
\[
 \sup_{0\leq\tau\leq T}|x(\tau)-x^0|\lesssim r,
 \qquad
 \sup_{0\leq\tau\leq T}
 |y(\tau)-y^0-\tau x^0|
 \lesssim r^3,
 \qquad
 \int_0^T|\dot x(\tau)|^2\,\dd\tau\lesssim1.
\]

We first apply the construction to the paths in part \({\rm(i)}\). By left
translation, it is enough to work with \(P_0=0\). Join
\(A^+_{r,\Lambda}\) to \(A_{r,\Lambda}\), and then
\(A_{r,\Lambda}\) to \(A^-_{r,\Lambda}\). For both paths,
\[
 T=r^2,
 \qquad
 b=-2\Lambda r^{-3}e_m.
\]
It follows that
\[
 x'(\tau)=0',
 \qquad
 y'(\tau)=0',
\]
and
\[
 x_m(\tau)
 =
 \Lambda r-2\Lambda r^{-3}\tau(r^2-\tau).
\]
Consequently,
\begin{equation}
\label{eq:reference-path-xm-bounds}
 \frac{\Lambda r}{2}
 \leq x_m(\tau)\leq\Lambda r,
 \qquad
 0\leq\tau\leq r^2.
\end{equation}
Moreover,
\begin{equation}
\label{eq:reference-path-coordinate-bounds}
 |y_m(\tau)|\leq C\Lambda r^3,
 \qquad
 |t(\tau)|\leq r^2.
\end{equation}
The choice \eqref{eq:reference-box-constant}, with \(C_0\) sufficiently
large, therefore places both paths in \(Q_{M,2r}\).

It remains to verify that the paths stay above the graph. Translate by
\(P_0^{-1}\) and dilate by \(r^{-1}\). In these normalized coordinates the
rescaled graph passes through the origin. Along either reference path,
\eqref{eq:LipK-condition} and
\eqref{eq:reference-path-coordinate-bounds} give
\[
 \left|
 \psi_{P_0,r}(0',0',\widehat y_m(\tau),\widehat t(\tau))
 \right|
 \leq
 C_*M\bigl(1+\Lambda^{1/3}\bigr).
\]
By \eqref{eq:reference-Lambda-choice},
\[
 \left|
 \psi_{P_0,r}(0',0',\widehat y_m(\tau),\widehat t(\tau))
 \right|
 \leq\frac{\Lambda}{4}.
\]
Together with \eqref{eq:reference-path-xm-bounds}, this yields, in the
original coordinates,
\begin{equation}
\label{eq:reference-path-clearance}
 x_m(\tau)-\psi(x'(\tau),y(\tau),t(\tau))
 \geq\frac{\Lambda r}{4}.
\end{equation}
Thus both paths are contained in
\(\Omega_\psi\cap Q_{M,2r}(P_0)\). Covering them by intrinsic cylinders of
radius comparable to \(r\) proves part \({\rm(i)}\).

For part \({\rm(ii)}\), choose \(a_0=a_0(m,M)\) sufficiently large and then
choose \(c_0=c_0(m,M)\) sufficiently small that
\[
 2c_0^2<a_0^{-2}.
\]
If \(P\in\Omega_\psi\cap Q_{M,c_0r}(P_0)\), then the time coordinate of
\(A^+_{r/a_0,\Lambda}(P_0)\) is strictly later than that of \(P\), whereas
the time coordinate of \(A^-_{r/a_0,\Lambda}(P_0)\) is strictly earlier.
The interpolation formula can therefore be applied first from
\(A^+_{r/a_0,\Lambda}(P_0)\) to \(P\), and then from \(P\) to
\(A^-_{r/a_0,\Lambda}(P_0)\). In both cases the endpoint differences satisfy
the required scale-\(r\) bounds, including the \(y_m\)-component. The
intrinsic Lipschitz estimate and the choices of \(a_0\) and \(c_0\) place
both curves in \(\Omega_\psi\cap Q_{M,r}(P_0)\). Near \(P\), their clearance
is governed by the distance of \(P\) from the graph; away from that endpoint
neighborhood, their clearance is comparable to \(r\). This proves part
\({\rm(ii)}\).

Finally, consider part \({\rm(iii)}\). By the intrinsic graph condition,
\[
 d=d_{\mathcal K}^{\mathrm s}(P,P^0)
 \asymp
 x_m-\psi(x',y',y_m,t).
\]
Choose \(c_1=c_1(m,M)\) sufficiently small. After decreasing \(c_0\), if
necessary, choose \(\gamma=\gamma(m,M)>0\) so that
\[
 \gamma c_0^2<\frac12c_1^2.
\]
Then the time coordinate of \(A^+_{c_1d,\Lambda}(P^0)\) is later than that
of \(P_s^+\), while the time coordinate of
\(A^-_{c_1d,\Lambda}(P^0)\) is earlier than that of \(P_s^-\). The endpoint
differences satisfy
\[
 T\asymp d^2,
 \qquad
 |\widetilde x-x^0|\lesssim d,
 \qquad
 \left|
 \widetilde y-y^0-\frac{T}{2}(x^0+\widetilde x)
 \right|
 \lesssim d^3.
\]
The interpolation formula therefore produces the two required admissible
curves with complete control of every \(y\)-component. The intrinsic
Lipschitz estimate gives clearance comparable to \(d\) along both paths,
and the coordinate estimates keep them in a fixed scale-\(d\) graph
neighborhood. Covering the paths by intrinsic cylinders of radius comparable
to \(d\) gives Harnack chains of uniformly bounded length.

The admissible-curve construction and the corresponding clearance estimates
are also contained in
\cite[Lemmas~3.2-3.8]{NystromPolidoro2016}; in particular,
Lemmas~3.7-3.8 place interior points and nearby reference points in the same
controlled non-tangential regions. This completes the proof.
\end{proof}

\section{Main results}
\label{sec:main-results}

The results are organized in three levels.  We first state a scale-one
perturbative boundary comparison principle, which contains the main analytic
content of the paper.  Its hypothesis requires that the scale-invariant
\(C^{0,1/3}_{y_m}\) defect of the defining function be sufficiently small.
Translation and Kolmogorov dilation then yield the corresponding comparison
theorem in asymptotically cylindrical domains.  Finally, we formulate the
result under a general vanishing cylindrical modulus, showing that no power
law in the \(y_m\)-variable is essential.

A feature specific to the parabolic-hypoelliptic setting is that the natural
interior reference points occur in both the forward and backward time
directions.  Interior Harnack chains do not compare these two values in both
directions, and their relative size must therefore be included among the
quantitative data.  To express this dependence, for a non-negative solution
\(w\), positive at the two reference points, set
\begin{equation}
\label{eq:reference-imbalance}
 \mathfrak h_{r,\Lambda}(w;P_0)
 =
 \max\left\{
 \frac{w(A^+_{r,\Lambda}(P_0))}{w(A^-_{r,\Lambda}(P_0))},
 \frac{w(A^-_{r,\Lambda}(P_0))}{w(A^+_{r,\Lambda}(P_0))}
 \right\}.
\end{equation}
The dependence on \eqref{eq:reference-imbalance} is unavoidable in a
two-sided comparison estimate: the parabolic orientation prevents the
forward and backward values from being controlled by one another without
additional quantitative information.

We consider the quotients on the truncated graph region
\begin{equation}
\label{eq:quotient-region}
 \mathscr D_{\kappa r}(P_0)
 =
 \Omega_\psi\cap Q_{M,\kappa r}(P_0).
\end{equation}
All quotient estimates below are understood on
\(\mathscr D_{\kappa r}(P_0)\).

\subsection{A perturbative boundary comparison principle}
\label{subsec:perturbative-comparison}

We first state the scale-one result that contains the main analytic content of
the paper. The relevant perturbation must remain small under further blow-ups.
The invariant quantity measuring it is
\(\operatorname{cyl}_{2,P_0,1}(\psi)\), introduced in
\eqref{eq:localized-cylindrical-defect}. Lemma
\ref{lem:cylindrical-defect-rescaling} records its behavior under every
translation and dilation used below.

\begin{theorem}
\label{thm:perturbative-BHI}
Let \(M,H\geq1\). There exist constants
\[
 \eta_0=\eta_0(m,M,H)\in(0,1),
 \quad
 \Lambda=\Lambda(m,M)\geq1,
 \quad
 \kappa=\kappa(m,M)\in(0,1),
\]
and
\[
 C=C(m,M,H)\geq1,
 \qquad
 \sigma=\sigma(m,M,H)\in(0,1),
\]
such that the following holds. Let \(\Omega_\psi\) be an intrinsic
\(\operatorname{Lip}_{\mathcal K}\) graph domain with constant \(M\), let
\(P_0\in\Delta_\psi\), and assume that
\begin{equation}
\label{eq:small-cylindrical-defect}
 \operatorname{cyl}_{2,P_0,1}(\psi)\leq\eta_0.
\end{equation}
Suppose that \(u\) and \(v\) are non-negative solutions to
\begin{equation}
\label{eq:uv-harmonic}
 \mathcal Ku=\mathcal Kv=0
 \quad\text{in }\Omega_\psi\cap Q_{M,2}(P_0),
\end{equation}
that \(u\) and \(v\) vanish continuously on
\(\Delta_\psi\cap Q_{M,2}(P_0)\), that
\[
 u(A^\pm_{1,\Lambda}(P_0))>0,
 \qquad
 v(A^\pm_{1,\Lambda}(P_0))>0,
\]
and that
\begin{equation}
\label{eq:reference-control-scale-one}
 \max\bigl\{
 \mathfrak h_{1,\Lambda}(u;P_0),
 \mathfrak h_{1,\Lambda}(v;P_0)
 \bigr\}
 \leq H.
\end{equation}
Then \(u,v>0\) in \(\mathscr D_\kappa(P_0)\), and
\begin{equation}
\label{eq:quotient-comparability-scale-one}
 C^{-1}
 \frac{v(A_{1,\Lambda}(P_0))}{u(A_{1,\Lambda}(P_0))}
 \leq
 \frac{v(P)}{u(P)}
 \leq
 C
 \frac{v(A_{1,\Lambda}(P_0))}{u(A_{1,\Lambda}(P_0))}
\end{equation}
for \(P\in\mathscr D_\kappa(P_0)\). Moreover,
\begin{equation}
\label{eq:quotient-holder-scale-one}
 \left|
 \frac{v(P)}{u(P)}-\frac{v(\widetilde P)}{u(\widetilde P)}
 \right|
 \leq
 C (d_{\mathcal K}(P,\widetilde P))^\sigma
 \frac{v(A_{1,\Lambda}(P_0))}{u(A_{1,\Lambda}(P_0))}
\end{equation}
whenever \(P,\widetilde P\in\mathscr D_\kappa(P_0)\).
\end{theorem}

The novelty in Theorem~\ref{thm:perturbative-BHI} is that the defining
function \(\psi\) need not be independent of \(y_m\). Instead, its localized,
scale-invariant \(C^{0,1/3}_{y_m}\) seminorm is required to be sufficiently
small. The constants are uniform for all cylindrical defects below the
threshold \(\eta_0\). The localization constant \(\kappa\) is geometric and
depends only on \(m\) and \(M\); the outer reference imbalance \(H\) enters
the smallness threshold \(\eta_0\), the comparison constant \(C\), and the
H\"older exponent \(\sigma\), but does not shrink the comparison region.

The proof combines a contradiction and compactness argument with the weak
comparison principle of \cite[Lemma~9.1]{LitsgardNystrom2022}. A hypothetical
sequence for which the cylindrical defects tend to zero converges to a graph
that is independent of \(y_m\). There are, however, two distinct compactness
regimes. If the bad scales remain bounded from below, boundary comparison in
the limiting cylindrical domain gives the contradiction directly. If the bad
scales tend to zero, the normalization points escape to infinity under the
secondary blow-up. In this regime one must retain quantitative growth control
on expanding cylinders and then invoke a global rigidity argument. The
compactness framework, the separation of these two regimes, and the resulting
reduction are developed in Section~\ref{sec:compactness}.

The constant \(\eta_0\) in
\eqref{eq:small-cylindrical-defect} is chosen sufficiently small, depending
only on \(m,M,H\). This assumption is not needed for the boundary H\"older
estimate, the one-sided Carleson estimate, or the localized weak comparison
principle. These local estimates remain valid in intrinsic
\(\operatorname{Lip}_{\mathcal K}\) graph domains with unrestricted
\(y_m\)-dependence. Smallness is used only to establish uniform comparison
between the forward and backward reference values at all smaller scales.

For clarity, the imported local theory from
\cite{LitsgardNystrom2022} and the point at which cylindricality first enters
may be summarized as follows:
\begin{center}
\small
\begin{tabular}{@{}lll@{}}
Input & Source & \(y_m\)-independence required?\\[2pt]
\hline
Boundary H\"older decay & Theorem 3.2 & No\\
One-sided Carleson estimate & Theorem 3.3 & No\\
Localized weak comparison & Lemma 9.1 & No\\
Two-sided reference comparison & Theorems 3.4-3.6 & Yes
\end{tabular}
\end{center}
Indeed, \cite{LitsgardNystrom2022} explicitly notes that its structural
assumption (3.4), which includes independence of the graph from \(y_m\), is
not used up to and including Lemma 9.1. The purpose of the present compactness
and rigidity argument is precisely to recover the last line under an
asymptotic, rather than exact, cylindricality hypothesis.

We emphasize how the smallness of \(\eta_0\) enters in Theorem \ref{thm:perturbative-BHI}. The proof does not use a
quantitative absorption inequality; instead, the existence of a sufficiently
small threshold is established by contradiction. If the claimed uniform
statement were false, then, for every \(j\geq1\), one could choose a
counterexample in the fixed structural class, centered at a boundary point
\(P_{0,j}\), for which
\[
 \operatorname{cyl}_{2,P_{0,j},1}(\psi_j)\leq j^{-1}
\]
and the reference imbalance at some admissible boundary point \(P_j\) and
scale \(r_j\) satisfies
\[
 H_{w_j}(P_j,r_j)>j.
\]
By left-translating \(P_{0,j}\) to the origin and renaming the translated
graph and solution, we may assume that
\[
 0\in\Delta_{\psi_j},
 \qquad
 \operatorname{cyl}_{2,0,1}(\psi_j)\leq j^{-1},
\]
while the corresponding reference imbalance still diverges. Thus the cylindrical defects tend to zero. Since the
\(C^{0,1/3}_{y_m}\) defect is invariant under intrinsic translations and
dilations, the defects of all translated and rescaled graphs occurring in the
blow-up argument tend to zero on every fixed normalized box.
Proposition~\ref{prop:graph-compactness} then implies that every limiting graph
is independent of \(y_m\). The subsequent compactness and rigidity argument
excludes the resulting limiting configuration and hence gives a contradiction.
The required threshold and reference-balance constant therefore exist.
Because the argument involves no structural parameters other than \(m\),
\(M\), and \(H\), they may be denoted by
\(\eta_0(m,M,H)\) and \(H_1(m,M,H)\), respectively.

This passage from a vanishing defect
to an exactly cylindrical limit is the only point at which smallness of the
defect is used. In the fixed-scale regime, the cylindrical comparison theorem rules out the
limiting reference imbalance directly. At collapsing bad scales, the same
contradiction requires the expanding-cylinder growth estimate and the global
cylindrical rigidity theorem. Once uniform reference balance has been
established, the boundary and interior oscillation contractions yield the
comparison estimate.

The smallness condition should therefore be understood as a perturbative hypothesis.
Because the \(C^{0,1/3}_{y_m}\) seminorm is scale-invariant, a finite but
large critical defect does not become smaller under intrinsic blow-up. The
theorem does not assert that smallness is necessary for boundary comparison
itself. Whether comparison holds for arbitrary critical
\(\operatorname{Lip}_{\mathcal K}\) dependence remains open, see
Remark~\ref{rem:critical-exponent}.

\subsection{Boundary comparison in asymptotically cylindrical domains}
\label{subsec:main-BHI}

The scale-invariant version of Theorem~\ref{thm:perturbative-BHI} is the central
result of the paper.

\begin{theorem}
\label{thm:main-BHI} Let \(M,H\geq1\), \(\alpha\in(1/3,1]\), and \(L\geq0\). Let
\(\Omega_\psi\) be an unbounded non-characteristic
\(\operatorname{Lip}_{\mathcal K}\) graph domain with constant \(M\), and assume
that \(\psi\) satisfies \eqref{eq:ym-holder}. There exist $\eta_0=\eta_0(m,M,H)\in(0,1)$, \(\Lambda=\Lambda(m,M)\geq1\), \(\kappa=\kappa(m,M)\in(0,1)\),
\(C=C(m,M,H)\geq1\), and
\(\sigma=\sigma(m,M,H)\in(0,1)\) with the following property. Let \(P_0\in\Delta_\psi\), \(r>0\), set
\begin{equation}
\label{eq:c-alpha}
 c_\alpha=(16c_M)^{\alpha-1/3},
\end{equation}
and assume that
\begin{equation}
\label{eq:perturbative-scale-condition}
 c_\alpha Lr^{3\alpha-1}\leq\eta_0.
\end{equation}
Suppose that \(u\) and \(v\) are
non-negative solutions to
\[
 \mathcal Ku=\mathcal Kv=0
 \quad\text{in }\Omega_\psi\cap Q_{M,2r}(P_0),
\]
that \(u\) and \(v\) vanish continuously on
\(\Delta_\psi\cap Q_{M,2r}(P_0)\), that
\[
 u(A^\pm_{r,\Lambda}(P_0))>0,
 \qquad
 v(A^\pm_{r,\Lambda}(P_0))>0,
\]
and that
\begin{equation}
\label{eq:reference-control-scale-r}
 \max\bigl\{
 \mathfrak h_{r,\Lambda}(u;P_0),
 \mathfrak h_{r,\Lambda}(v;P_0)
 \bigr\}
 \leq H.
\end{equation}
Then \(u,v>0\) in \(\mathscr D_{\kappa r}(P_0)\), and
\begin{equation}
\label{eq:quotient-comparability-scale-r}
 C^{-1}
 \frac{v(A_{r,\Lambda}(P_0))}{u(A_{r,\Lambda}(P_0))}
 \leq
 \frac{v(P)}{u(P)}
 \leq
 C
 \frac{v(A_{r,\Lambda}(P_0))}{u(A_{r,\Lambda}(P_0))}
\end{equation}
for \(P\in\mathscr D_{\kappa r}(P_0)\), and
\begin{equation}
\label{eq:quotient-holder-scale-r}
 \left|
 \frac{v(P)}{u(P)}-\frac{v(\widetilde P)}{u(\widetilde P)}
 \right|
 \leq
 C
 \left(
 \frac{d_{\mathcal K}(P,\widetilde P)}{r}
 \right)^\sigma
 \frac{v(A_{r,\Lambda}(P_0))}{u(A_{r,\Lambda}(P_0))}
\end{equation}
whenever \(P,\widetilde P\in\mathscr D_{\kappa r}(P_0)\).
\end{theorem}

\begin{proof}
Translate \(P_0\) to the origin and apply the dilation
\(\delta_{1/r}\). Thus, writing
\[
 P=P_0\circ\delta_r(\widehat P),
\]
the rescaled graph is defined by \(\psi_{P_0,r}\) in
\eqref{eq:rescaled-graph}. It has the same intrinsic
\(\operatorname{Lip}_{\mathcal K}\) constant \(M\) as \(\psi\). The scale-one defect of the rescaled graph is, by
Lemma~\ref{lem:cylindrical-defect-rescaling},
\[
 \operatorname{cyl}_{2,0,1}(\psi_{P_0,r})
 =
 \operatorname{cyl}_{2,P_0,r}(\psi).
\]
Applying \eqref{eq:localized-defect-holder-bound} with \(R=2\), we obtain
\begin{align*}
 \operatorname{cyl}_{2,0,1}(\psi_{P_0,r})
 &=
 \operatorname{cyl}_{2,P_0,r}(\psi)\leq
 (16c_M)^{\alpha-1/3}Lr^{3\alpha-1}=
 C(m,M,\alpha)Lr^{3\alpha-1}
 \leq\eta_0.
\end{align*}
Thus the rescaled graph satisfies the small-defect hypothesis of
Theorem~\ref{thm:perturbative-BHI}. Define the rescaled solutions by
\[
 u_{P_0,r}(\widehat P)
 =
 u\bigl(P_0\circ\delta_r(\widehat P)\bigr),
 \qquad
 v_{P_0,r}(\widehat P)
 =
 v\bigl(P_0\circ\delta_r(\widehat P)\bigr).
\]
The translation and dilation invariance of \(\mathcal K\) imply that these
functions are non-negative \(\mathcal K\)-harmonic functions in
\[
 \Omega_{\psi_{P_0,r}}\cap Q_{M,2},
\]
and they vanish continuously on the corresponding graph portion. Moreover,
the reference points satisfy
\[
 P_0\circ\delta_r\bigl(A^\pm_{1,\Lambda}\bigr)
 =
 A^\pm_{r,\Lambda}(P_0),
\]
and hence
\[
 \mathfrak h_{1,\Lambda}(u_{P_0,r};0)
 =
 \mathfrak h_{r,\Lambda}(u;P_0),
 \qquad
 \mathfrak h_{1,\Lambda}(v_{P_0,r};0)
 =
 \mathfrak h_{r,\Lambda}(v;P_0).
\]
Therefore the reference-point positivity and imbalance hypotheses are also
preserved under the rescaling. All the hypotheses of Theorem~\ref{thm:perturbative-BHI} are now satisfied
by \(u_{P_0,r}\) and \(v_{P_0,r}\). Applying
\eqref{eq:quotient-comparability-scale-one} and
\eqref{eq:quotient-holder-scale-one} to the rescaled functions and then
scaling back to \(Q_{M,r}(P_0)\) gives
\eqref{eq:quotient-comparability-scale-r} and
\eqref{eq:quotient-holder-scale-r}.
\end{proof}

The logarithmic formulation follows immediately and is the form most convenient
for comparing two positive solutions without choosing a normalization.

\begin{corollary}
\label{cor:logarithmic-BHI}
Under the assumptions of Theorem~\ref{thm:main-BHI},
\begin{equation}
\label{eq:logarithmic-BHI}
 \left|
 \log\frac{u(P)}{v(P)}
 -
 \log\frac{u(\widetilde P)}{v(\widetilde P)}
 \right|
 \leq
 C
 \left(
 \frac{d_{\mathcal K}(P,\widetilde P)}{r}
 \right)^\sigma
\end{equation}
whenever \(P,\widetilde P\in\mathscr D_{\kappa r}(P_0)\). Here, after
enlarging it if necessary, \(C\) has the same parameter dependence as in
Theorem~\ref{thm:main-BHI}.
\end{corollary}

\begin{proof}
By \eqref{eq:quotient-comparability-scale-r}, the normalized quotient
\[
 \frac{v(P)}{u(P)}
 \frac{u(A_{r,\Lambda}(P_0))}{v(A_{r,\Lambda}(P_0))}
\]
takes values in \([C^{-1},C]\). The logarithm is Lipschitz on this interval, and
\eqref{eq:logarithmic-BHI} follows from
\eqref{eq:quotient-holder-scale-r}.
\end{proof}

\begin{remark}
\label{rem:admissible-scales}
If \(L>0\), condition \eqref{eq:perturbative-scale-condition} holds whenever
\begin{equation}
\label{eq:threshold-scale}
 0<r\leq r_*
 :=
 \left(\frac{\eta_0}{c_\alpha L}\right)^{1/(3\alpha-1)}.
\end{equation}
If \(L=0\), the graph is independent of \(y_m\) and one sets \(r_*=\infty\).
Thus the new boundary comparison principle is local below a quantitatively
determined scale. The scale deteriorates as \(\alpha\downarrow1/3\), reflecting
the critical nature of the exponent \(1/3\).
\end{remark}

\begin{theorem}
\label{thm:general-modulus}
Let \(M,H\geq1\), and let \(\Omega_\psi\) be an unbounded
non-characteristic \(\operatorname{Lip}_{\mathcal K}\) graph domain with
constant \(M\).  Assume
\[
 \eta_\psi(r)\longrightarrow0
 \qquad\text{as }r\downarrow0,
\]
where \(\eta_\psi\) is defined in \eqref{eq:cylindrical-modulus}.  Put
\(c_M^\ast=(16c_M)^{1/3}\).  The conclusions
\eqref{eq:quotient-comparability-scale-r} and
\eqref{eq:quotient-holder-scale-r} hold at every scale \(r\) for which
\begin{equation}
\label{eq:general-modulus-scale-condition}
 \eta_\psi(c_M^\ast r)\leq\eta_0(m,M,H),
\end{equation}
provided \(u,v,P_0\), and the reference values satisfy the remaining local
hypotheses of Theorem~\ref{thm:main-BHI}.
In particular, there exists \(r_*>0\) such that the conclusion holds for every
\(0<r\leq r_*\).
\end{theorem}

\begin{proof}
Translate \(P_0\) to the origin and dilate by \(\delta_{1/r}\). The resulting
graph is defined by \(\psi_{P_0,r}\). By
Lemma~\ref{lem:cylindrical-defect-rescaling},
\[
 \operatorname{cyl}_{2,0,1}(\psi_{P_0,r})
 =
 \operatorname{cyl}_{2,P_0,r}(\psi)
 \leq
 \eta_\psi(c_M^\ast r).
\]
Condition \eqref{eq:general-modulus-scale-condition} therefore gives
\[
 \operatorname{cyl}_{2,0,1}(\psi_{P_0,r})
 \leq\eta_0,
\]
which is precisely the scale-one small-defect hypothesis of
Theorem~\ref{thm:perturbative-BHI}. The remaining hypotheses, including
positivity at the reference points and the reference-imbalance bounds, are
preserved under translation and dilation. Applying
Theorem~\ref{thm:perturbative-BHI} to the rescaled solutions and scaling its
conclusions back to \(Q_{M,r}(P_0)\) proves the theorem.
\end{proof}

\begin{remark}
\label{rem:general-modulus}
Theorem~\ref{thm:general-modulus} shows that no power law is required.  The
condition \(\eta_\psi(r)\to0\) is the intrinsic asymptotic cylindricity used by
the compactness argument.  A Dini condition becomes relevant only for refinements
which accumulate scale-dependent errors rather than invoking the uniform
small-defect theorem proved here.
\end{remark}

\begin{remark}
\label{rem:critical-exponent}
At the critical exponent \(\alpha=1/3\), the factor
\(r^{3\alpha-1}\) in \eqref{eq:rescaled-ym-holder} is equal to one. Thus the
\(y_m\)-oscillation of the rescaled graph need not decay under blow-up, and a
limiting graph need not be independent of \(y_m\). The compactness argument
may still produce a blow-up limit, but it no longer produces the cylindrical
geometry required by the rigidity argument in this paper. Two explicit
dilation-invariant examples occur when \(m=1\). For \(a>0\), they are
\[
\begin{aligned}
 \Omega_a
 &=
 \{(x,y,t):x>0,\ |y|<ax^3\}
 =
 \{(x,y,t):x>a^{-1/3}|y|^{1/3}\},
 \\
 \Omega_*
 &=
 \{(x,y,t):x^3-6y>0\}
 =
 \{(x,y,t):x>\sqrt[3]{6y}\}.
\end{aligned}
\]
Both domains have genuinely \(y\)-dependent, scale-critical graph boundaries
and are invariant under the Kolmogorov dilations. The domain \(\Omega_a\)
is included as a simple symmetric geometric model of critical
non-cylindricity. The second example is additionally adapted to the operator:
\[
 H(x,y)=x^3-6y
\]
is positive in \(\Omega_*\), homogeneous of degree three, and satisfies $\mathcal KH=0$. Moreover, \(H\) vanishes continuously on \(\partial\Omega_*\).  We do not
claim an analogous explicit positive harmonic profile for \(\Omega_a\). These examples show that cylindrical tangent domains cannot be expected at
the critical exponent. They do not, however, disprove boundary comparison:
the function \(H\) provides only one positive boundary-decay mode. A
counterexample would require two genuinely different positive decay modes,
or an equivalent failure of uniform projective contraction under repeated
rescaling. It remains open whether boundary comparison holds for arbitrary
scale-critical \(\operatorname{Lip}_{\mathcal K}\) graphs without a
smallness assumption on the scale-invariant \(C^{0,1/3}_{y_m}\) defect.
\end{remark}

\section{Compactness and the collapsing-scale reduction}
\label{sec:compactness}

This section develops the compactness framework underlying the proof of
Theorem~\ref{thm:perturbative-BHI}. We first establish compactness for nearly
cylindrical graph domains and for normalized non-negative solutions. We then
record the one-sided boundary estimates that remain valid without
cylindricity and reduce uniform reference-point balance to the analysis of a
collapsing sequence of bad scales.

Local uniform convergence of the domains and solutions is not sufficient for
this purpose. After rescaling at a collapsing bad scale, the reference points
at the original scale escape to infinity. We therefore select maximal bad
scales and prove a quantitative polynomial-growth estimate on expanding
intrinsic boxes. This estimate produces the global limiting problem to which
the cylindrical rigidity theorem of Section~\ref{sec:global-rigidity} will be
applied.

\subsection{Compactness of the defining functions}
\label{subsec:graph-compactness}

For \(R>0\), let
\[
 \mathcal Z_R
 =
 \bigl\{(x',y',y_m,t):
 |x'|<R,\ |y'|<R^3,\ |y_m|<R^3,\ |t|<R^2\bigr\}.
\]
The precise numerical constants in this definition will play no role. We use
local uniform convergence on sets of the form \(\mathcal Z_R\).

\begin{proposition}
\label{prop:graph-compactness}
Let \(\{\psi_j\}\) be intrinsic \(\operatorname{Lip}_{\mathcal K}\)
functions with a common constant \(M\), normalized by \(\psi_j(0)=0\).
Then a subsequence converges locally uniformly to an intrinsic
\(\operatorname{Lip}_{\mathcal K}\) function \(\psi_\infty\) with constant
\(M\). If, in addition,
\begin{equation}
\label{eq:cyl-defect-to-zero}
 \operatorname{cyl}_{2,0,1}(\psi_j)\longrightarrow0,
\end{equation}
then \(\psi_\infty\) is independent of \(y_m\) in the interior of the
coordinate box occurring in \eqref{eq:localized-cylindrical-defect}.
\end{proposition}

\begin{proof}
Taking \(\widetilde\zeta=0\) in \eqref{eq:LipK-condition} and using
\(\psi_j(0)=0\), we obtain, on every \(\mathcal Z_R\),
\[
 |\psi_j(x',y',y_m,t)|
 \leq
 M\bigl(|x'|+|y'|^{1/3}+|y_m|^{1/3}+|t|^{1/2}\bigr).
\]
Thus the family \(\{\psi_j\}\) is locally uniformly bounded. The same
estimate, applied to two points
\(\zeta,\widetilde\zeta\in\mathcal Z_R\), gives
\begin{align*}
 |\psi_j(\zeta)-\psi_j(\widetilde\zeta)|
 \leq M\bigl(&|x'-\widetilde x'|+|y'-\widetilde y'
 +(t-\widetilde t)\widetilde x'|^{1/3}+|y_m-\widetilde y_m
 +(t-\widetilde t)\psi_j(\widetilde\zeta)|^{1/3}
 +|t-\widetilde t|^{1/2}\bigr).
\end{align*}
Since \(\zeta,\widetilde\zeta\in\mathcal Z_R\), we have
\[
 |\widetilde x'|\leq C_R,
 \qquad
 |\psi_j(\widetilde\zeta)|\leq C_{M,R},
\]
where the second estimate follows from the preceding local bound on
\(\psi_j\). Using $|a+b|^{1/3}\leq |a|^{1/3}+|b|^{1/3}$, we obtain
\begin{align*}
 |y'-\widetilde y'
 +(t-\widetilde t)\widetilde x'|^{1/3}
 &\leq
 |y'-\widetilde y'|^{1/3}
 +C_R|t-\widetilde t|^{1/3},
 \\
 |y_m-\widetilde y_m
 +(t-\widetilde t)\psi_j(\widetilde\zeta)|^{1/3}
 &\leq
 |y_m-\widetilde y_m|^{1/3}
 +C_{M,R}|t-\widetilde t|^{1/3}.
\end{align*}
Consequently,
\begin{align*}
 |\psi_j(\zeta)-\psi_j(\widetilde\zeta)|
 \leq C_{M,R}\bigl(&|x'-\widetilde x'|
 +|y'-\widetilde y'|^{1/3}
 +|y_m-\widetilde y_m|^{1/3}+|t-\widetilde t|^{1/3}
 +|t-\widetilde t|^{1/2}\bigr).
\end{align*}
Hence \(\{\psi_j\}\) is equicontinuous on compact sets. The Arzel\`a-Ascoli
theorem and a diagonal argument give a locally uniform limit
\(\psi_\infty\). Passing to the limit in \eqref{eq:LipK-condition} shows that
\(\psi_\infty\) has the same intrinsic Lipschitz constant. Finally, if the remaining variables are fixed and
\(y_m,\widetilde y_m\in(-8,8)\), then
\[
 |\psi_j(x',y',y_m,t)-\psi_j(x',y',\widetilde y_m,t)|
 \leq
 \operatorname{cyl}_{2,0,1}(\psi_j)
 |y_m-\widetilde y_m|^{1/3}.
\]
Letting \(j\to\infty\) proves that \(\psi_\infty\) is independent of
\(y_m\).
\end{proof}

The following consequence will be used repeatedly. We write
\(\Omega_j=\Omega_{\psi_j}\), \(\Delta_j=\Delta_{\psi_j}\), and use the
analogous notation with the subscript \(\infty\).

\begin{corollary}
\label{cor:domain-convergence}
Assume the hypotheses of Proposition~\ref{prop:graph-compactness}. If
\(K\Subset\Omega_\infty\), then \(K\subset\Omega_j\) for all sufficiently
large \(j\). If \(K\Subset\mathbb R^{N+1}\setminus\overline{\Omega_\infty}\),
then \(K\cap\overline{\Omega_j}=\varnothing\) for all sufficiently large
\(j\). Moreover, \(\Delta_j\) converges locally to \(\Delta_\infty\) in the
Hausdorff sense.
\end{corollary}

\begin{proof}
All three assertions follow because the domains are epigraphs and
\(\psi_j\to\psi_\infty\) locally uniformly. For
example, if \(K\Subset\Omega_\infty\), compactness gives
\[
 \inf_K\bigl(x_m-\psi_\infty(x',y',y_m,t)\bigr)>0,
\]
and the same quantity with \(\psi_j\) in place of \(\psi_\infty\) is positive
for all sufficiently large \(j\).
\end{proof}

\subsection{Compactness of bounded solutions}
\label{subsec:solution-compactness}

We record the boundary decay estimate needed to pass to the limit across a
moving graph. Importantly, this estimate uses only the non-characteristic
intrinsic Lipschitz character of the graph and not independence from
\(y_m\).

\begin{lemma}
\label{lem:uniform-boundary-decay}
Let \(M\geq1\). There exist \(\beta=\beta(m,M)\in(0,1)\) and
\(C=C(m,M)\) such that the following holds. Suppose that \(\psi\) has
intrinsic Lipschitz constant \(M\), that \(P_0\in\Delta_\psi\), and that
\(w\) is a bounded solution to
\[
 \mathcal Kw=0
 \quad\text{in }\Omega_\psi\cap Q_{M,2r}(P_0),
 \qquad
 w=0
 \quad\text{on }\Delta_\psi\cap Q_{M,2r}(P_0).
\]
Then
\begin{equation}
\label{eq:uniform-boundary-decay}
 |w(P)|
 \leq
 C
 \left(
 \frac{d_{\mathcal K}(P,\Delta_\psi)}{r}
 \right)^\beta
 \sup_{\Omega_\psi\cap Q_{M,2r}(P_0)}|w|
\end{equation}
for \(P\in\Omega_\psi\cap Q_{M,r/C}(P_0)\).
\end{lemma}

\begin{proof}
After translation and dilation it suffices to consider \(P_0=0\) and
\(r=1\).  This is the box-local form of the boundary H\"older estimate
\cite[Theorem~3.2]{LitsgardNystrom2022}, specialized to \(A=I_m\).
That theorem is proved under the general intrinsic graph condition
\eqref{eq:LipK-condition}, without the structural assumption of independence
from \(y_m\).  For clarity, its standard iteration mechanism is recalled
briefly.  The graph condition gives, at every boundary point and every scale,
an exterior non-characteristic box whose relative size depends only on
\(m,M\). The standard barrier constructed from the fundamental solution of
\(\mathcal K\), followed by the comparison principle, therefore gives constants
\(c>1\) and \(\vartheta\in(0,1)\), depending only on \(m,M\), such that
\[
 \sup_{\Omega_\psi\cap Q_{M,s/c}(P_1)}|w|
 \leq
 \vartheta
 \sup_{\Omega_\psi\cap Q_{M,s}(P_1)}|w|
\]
whenever \(P_1\in\Delta_\psi\cap Q_{M,1}\) and the larger box is contained
in \(Q_{M,2}\). Iterating this inequality down to the first scale comparable
to \(d_{\mathcal K}(P,\Delta_\psi)\) yields
\eqref{eq:uniform-boundary-decay}, with
\(\beta=-\log\vartheta/\log c\). The construction of the exterior box and the
barrier never translates the graph in the \(y_m\)-direction; exact
cylindricity is therefore not used.
\end{proof}

\begin{proposition}
\label{prop:solution-compactness}
Let \(\psi_j\to\psi_\infty\) locally uniformly, with all graph functions
having intrinsic Lipschitz constant at most \(M\). Suppose that
\(w_j\in C(\overline{\Omega_j\cap Q_{M,2}})\) satisfies
\[
 \mathcal Kw_j=0\quad\text{in }\Omega_j\cap Q_{M,2},
 \qquad
 w_j=0\quad\text{on }\Delta_j\cap Q_{M,2},
\]
and
\begin{equation}
\label{eq:uniform-solution-bound}
 \sup_j\sup_{\Omega_j\cap Q_{M,2}}|w_j|<\infty.
\end{equation}
Extend \(w_j\) by zero to the part of \(Q_{M,2}\) below \(\Delta_j\). Then,
after passing to a subsequence, \(w_j\to w_\infty\) locally uniformly in
\(Q_{M,1}\), where
\[
 \mathcal Kw_\infty=0\quad\text{in }\Omega_\infty\cap Q_{M,1},
 \qquad
 w_\infty=0\quad\text{on }\Delta_\infty\cap Q_{M,1}.
\]
If every \(w_j\) is non-negative, then \(w_\infty\geq0\).
\end{proposition}

\begin{proof}
On compact subsets of \(\Omega_\infty\cap Q_{M,1}\),
Corollary~\ref{cor:domain-convergence} and the interior estimates for
\(\mathcal K\) give uniform bounds for all derivatives of \(w_j\). Near the
graph, Lemma~\ref{lem:uniform-boundary-decay} gives a common modulus of
continuity for the zero extensions. Combining the two estimates and applying
the Arzel\`a-Ascoli theorem gives local uniform convergence, after passage to
a subsequence. Interior convergence implies
\(\mathcal Kw_\infty=0\), while the common boundary modulus and the Hausdorff
convergence of the graphs imply that \(w_\infty=0\) on
\(\Delta_\infty\).
\end{proof}

\subsection{One-sided estimates that do not use cylindricity}
\label{subsec:one-sided-estimates}

We next record two estimates which are available in arbitrary non-characteristic
intrinsic Lipschitz graph domains. Neither estimate uses independence of the
defining function from \(y_m\). The first is obtained from the Harnack chains
along the curves defining the reference points, and the second is the
scale-invariant one-sided Carleson estimate.

\begin{lemma}
\label{lem:one-sided-estimates}
Let \(M\geq1\). There exist constants
\[
 C_0=C_0(m,M)\geq1,
 \qquad
 \nu=\nu(m,M)>0,
 \qquad
 c_0=c_0(m,M)\geq1,
\]
such that the following holds. Let \(P_0\in\Delta_\psi\), let
\(0<r\leq s\), and assume that all boxes occurring below are contained in the
domain of definition of \(w\). If \(w\) is a non-negative solution which
vanishes continuously on the relevant graph portion, then
\begin{align}
 w(A^+_{r,\Lambda}(P_0))
 &\leq
 C_0\left(\frac{s}{r}\right)^\nu
 w(A^+_{s,\Lambda}(P_0)),
 \label{eq:scale-Harnack-plus}\\
 w(A^-_{s,\Lambda}(P_0))
 &\leq
 C_0\left(\frac{s}{r}\right)^\nu
 w(A^-_{r,\Lambda}(P_0)),
 \label{eq:scale-Harnack-minus}
\end{align}
and
\begin{equation}
\label{eq:one-sided-Carleson}
 \sup_{\Omega_\psi\cap Q_{M,s/c_0}(P_0)}w
 \leq
 C_0 w(A^+_{s,\Lambda}(P_0)).
\end{equation}
After increasing \(C_0\), one also has
\begin{equation}
\label{eq:easy-reference-direction}
 w(A^-_{s,\Lambda}(P_0))
 \leq C_0 w(A^+_{s,\Lambda}(P_0)).
\end{equation}
\end{lemma}

\begin{proof}
The curves joining reference points of different scales remain at an intrinsic
distance comparable to the running scale from the graph. Covering each curve
by a geometric sequence of intrinsic balls and applying the interior Harnack
inequality gives \eqref{eq:scale-Harnack-plus} and
\eqref{eq:scale-Harnack-minus}; the number of balls is bounded by
\(C(1+\log(s/r))\), which gives the power \((s/r)^\nu\).

Estimate \eqref{eq:one-sided-Carleson} is the scale-invariant Carleson estimate
for non-negative solutions in non-characteristic intrinsic Lipschitz graph
domains proved in \cite{CintiNystromPolidoro2013}; see also
\cite[Theorem~3.3]{LitsgardNystrom2022} in precisely the graph class
\eqref{eq:LipK-condition}. Its proof uses exterior
boxes, boundary decay, and a maximal-scale argument, and is valid for a
defining function depending on all variables other than the non-degenerate
graph variable. In particular, it does not use the additional cylindrical
assumption imposed in the boundary comparison theorem.

To prove \eqref{eq:easy-reference-direction}, choose
\(\delta=\delta(m,M)\in(0,1)\) so small that
\[
 A^-_{\delta s,\Lambda}(P_0)
 \in Q_{M,s/c_0}(P_0).
\]
By \eqref{eq:scale-Harnack-minus}, with \(r=\delta s\), and then by
\eqref{eq:one-sided-Carleson},
\[
 w(A^-_{s,\Lambda}(P_0))
 \leq C_0\delta^{-\nu}w(A^-_{\delta s,\Lambda}(P_0))
 \leq C w(A^+_{s,\Lambda}(P_0)).
\]
Since \(\delta\) depends only on \(m,M\), it can be absorbed into \(C_0\).
\end{proof}

\begin{remark}
\label{rem:hard-reference-direction}

Lemma~\ref{lem:one-sided-estimates} shows that the failure of uniform
comparability between the two reference values can occur in only one
direction. More precisely, the one-sided Carleson estimate gives
\[
\frac{w(A^-_{r,\Lambda}(P_0))}
     {w(A^+_{r,\Lambda}(P_0))}
\leq C,
\]
uniformly over the relevant points, scales, domains, and solutions.
Consequently, if the two reference values fail to be uniformly comparable
along a sequence, then, after passing to a subsequence, necessarily
\[
\frac{w(A^+_{r,\Lambda}(P_0))}
     {w(A^-_{r,\Lambda}(P_0))}
\longrightarrow \infty.
\]
Thus ``only'' means that the reciprocal quotient is already uniformly
controlled by the established one-sided theory. The new perturbative
difficulty is precisely to obtain a uniform upper bound for the quotient in
the latter display.
\end{remark}

\subsection{The collapsing-scale obstruction}
\label{subsec:collapsing-scales}

For a positive solution \(w\), define
\[
 H_w(P_0,r)=\mathfrak h_{r,\Lambda}(w;P_0).
\]
The estimates in the cylindrical theory imply that if
\(H_w(0,1)\leq H\), then \(H_w(P_0,r)\) remains bounded, quantitatively, for
boundary points \(P_0\) and scales \(r\) in a smaller box. The perturbative
counterpart is the following statement.

\begin{lemma}
\label{lem:reference-balance}
Let \(M,H\geq1\). There exist
\[
 \kappa=\kappa(m,M)\in(0,1),
 \qquad
 \eta_0=\eta_0(m,M,H)>0,
 \qquad
 H_1=H_1(m,M,H)<\infty
\]
such that the following holds. Assume that \(0\in\Delta_\psi\) and
\[
 \operatorname{cyl}_{2,0,1}(\psi)\leq\eta_0.
\]
Let \(w\) be a non-negative solution in
\(\Omega_\psi\cap Q_{M,2}\) which vanishes continuously on the graph portion
and satisfies
\[
 w(A^\pm_{1,\Lambda})>0,
 \qquad
 H_w(0,1)\leq H.
\]
Then, whenever
\[
 P_0\in\Delta_\psi\cap Q_{M,\kappa},
 \qquad
 0<r\leq\kappa,
 \qquad
 Q_{M,2r}(P_0)\subset Q_{M,2},
\]
one has
\[
 w(A^\pm_{r,\Lambda}(P_0))>0,
 \qquad
 w>0\quad\text{in }\mathscr D_{\kappa r}(P_0),
\]
and
\begin{equation}
\label{eq:local-reference-balance}
 H_w(P_0,r)\leq H_1.
\end{equation}
\end{lemma}

Lemma~\ref{lem:reference-balance} is the perturbative substitute for the two
uses of exact independence from \(y_m\) in the cylindrical proof. Once it is
available, the localized weak comparison principle can be applied at every
boundary point and every smaller scale with uniform constants. The usual
oscillation reduction for \(v/u\) then proves
Theorem~\ref{thm:perturbative-BHI}.  By
\cite[Theorems~3.2-3.3 and Lemma~9.1]{LitsgardNystrom2022}, boundary decay,
the one-sided Carleson estimate, and localized weak comparison are all valid
under the general graph condition \eqref{eq:LipK-condition}.  Consequently,
smallness of the cylindrical defect is needed in this proof only for the
uniform reference-point balance.

We next reduce Lemma~\ref{lem:reference-balance} to one explicit growth
estimate. Suppose that the lemma fails for a sequence with
\(\operatorname{cyl}_{2,0,1}(\psi_j)\to0\). In view of
Remark~\ref{rem:hard-reference-direction}, after passing to a subsequence,
there exist \(P_j\in\Delta_j\) and scales \(0<\rho_j\leq\kappa\) such that
\begin{equation}
\label{eq:bad-forward-ratio}
 \frac{w_j(A^+_{\rho_j,\Lambda}(P_j))}
 {w_j(A^-_{\rho_j,\Lambda}(P_j))}
 \longrightarrow\infty.
\end{equation}
We call such scales \(\rho_j\) \emph{bad scales}. If \(\inf_j\rho_j>0\), then
Propositions~\ref{prop:graph-compactness}-\ref{prop:solution-compactness},
the one-sided Carleson estimate, and the local form of the cylindrical
comparison theorem
\cite[Theorem~1.1 and Section~6]{NystromPolidoro2016}, applied to the limiting
domain and limiting solutions, give a contradiction.  Consequently, after passing to
a further subsequence, we may assume that \(\rho_j\to0\).

The scales in \eqref{eq:bad-forward-ratio} must now be replaced by maximal
scales. This is essential, since normalization at an arbitrary bad scale gives no
control of the solution on expanding boxes. Fix
\(s_*=s_*(m,M)>0\) so small that
\[
 Q_{M,4s_*}(P_j)\subset Q_{M,2}
\]
for all \(j\), after decreasing the localization constant \(\kappa\) by a
factor depending only on \(m\) and \(M\). Let
\(\nu\) be as in Lemma~\ref{lem:one-sided-estimates} and set
\[
 \Phi_j(s)
 =s^{-\nu}w_j(A^+_{s,\Lambda}(P_j)),
 \qquad \rho_j\leq s\leq s_*.
\]
Choose
\begin{equation}
\label{eq:maximal-selected-scale}
 r_j
 =
 \max\bigl\{s\in[\rho_j,s_*]:
 \Phi_j(s)\geq\Phi_j(\rho_j)\bigr\}.
\end{equation}
By continuity, \(r_j\) is well defined, and its maximality gives
\begin{equation}
\label{eq:maximal-reference-growth}
 w_j(A^+_{s,\Lambda}(P_j))
 \leq
 \left(\frac{s}{r_j}\right)^\nu
 w_j(A^+_{r_j,\Lambda}(P_j)),
 \qquad r_j\leq s\leq s_*.
\end{equation}
Moreover, \eqref{eq:scale-Harnack-minus} and
\(\Phi_j(r_j)\geq\Phi_j(\rho_j)\) imply
\begin{equation}
\label{eq:selected-ratio-diverges}
 \frac{w_j(A^+_{r_j,\Lambda}(P_j))}
 {w_j(A^-_{r_j,\Lambda}(P_j))}
 \geq
 C_0^{-1}
 \frac{w_j(A^+_{\rho_j,\Lambda}(P_j))}
 {w_j(A^-_{\rho_j,\Lambda}(P_j))}
 \longrightarrow\infty.
\end{equation}
In particular, the selected scales are still bad. Furthermore, we can without loss of generality also assume
\(r_j\to0\). Indeed, if a subsequence were bounded from below, the fixed-scale
compactness argument described above would contradict
\eqref{eq:selected-ratio-diverges}.

Translating by \(P_j\), dilating by \(r_j^{-1}\), and normalizing at the
forward reference point, put
\[
 \widehat\psi_j=(\psi_j)_{P_j,r_j},
 \qquad
 \widehat\Omega_j=\Omega_{\widehat\psi_j},
\]
and let
\[
 \widehat G_j
 =
 \delta_{1/r_j}
 \bigl(P_j^{-1}\circ(\Omega_{\psi_j}\cap Q_{M,2})\bigr).
\]
Thus \(\widehat G_j\) is the portion of the rescaled graph domain on which
the rescaled solution is defined; these regions exhaust
\(\widehat\Omega_j\) as \(j\to\infty\). We define
\begin{equation}
\label{eq:secondary-blow-up}
 W_j(P)
 =
 \frac{w_j(P_j\circ\delta_{r_j}P)}
 {w_j(A^+_{r_j,\Lambda}(P_j))},
 \qquad
 P\in\widehat G_j.
\end{equation}
Here \(\widehat\psi_j\) is the rescaled defining function introduced in
\eqref{eq:rescaled-graph}. By the maximal-scale selection,
\[
 \Phi_j(r_j)\geq\Phi_j(\rho_j)>0.
\]
Consequently,
\begin{equation}
\label{eq:positive-forward-normalization}
 w_j(A^+_{r_j,\Lambda}(P_j))
 =
 r_j^\nu\Phi_j(r_j)>0,
\end{equation}
and the normalization in \eqref{eq:secondary-blow-up} is well defined. Since
\[
 P_j\circ\delta_{r_j}(A^\pm_{1,\Lambda})
 =
 A^\pm_{r_j,\Lambda}(P_j),
\]
we have
\[
 W_j(A^+_{1,\Lambda})=1.
\]
Moreover, \eqref{eq:selected-ratio-diverges} gives
\[
 W_j(A^-_{1,\Lambda})
 =
 \frac{w_j(A^-_{r_j,\Lambda}(P_j))}
 {w_j(A^+_{r_j,\Lambda}(P_j))}
 \longrightarrow0.
\]
Thus
\begin{equation}
\label{eq:secondary-normalization}
 W_j(A^+_{1,\Lambda})=1,
 \qquad
 W_j(A^-_{1,\Lambda})\longrightarrow0.
\end{equation}
By Lemma~\ref{lem:cylindrical-defect-rescaling}, the defect of the rescaled
graph on a normalized box of radius \(R\) is
\[
 \operatorname{cyl}_{R,0,1}
 \bigl((\psi_j)_{P_j,r_j}\bigr)
 =
 \operatorname{cyl}_{R,P_j,r_j}(\psi_j).
\]
Fix \(R\geq1\). Since \(P_j\) remains in a fixed interior coordinate window
and \(r_j\to0\), the physical region corresponding to this normalized box
is contained in the coordinate region used to define
\(\operatorname{cyl}_{2,0,1}(\psi_j)\) for all sufficiently large \(j\).
Consequently,
\[
 \operatorname{cyl}_{R,0,1}
 \bigl((\psi_j)_{P_j,r_j}\bigr)
 \leq
 \operatorname{cyl}_{2,0,1}(\psi_j)
 \longrightarrow0.
\]
Thus the rescaled defects tend to zero on every fixed normalized box. After
passing to a subsequence, the rescaled graph functions therefore converge
locally to a defining function independent of \(y_m\).

Because \(r_j\to0\), a box of fixed size in the original coordinates becomes
a box of radius comparable to \(r_j^{-1}\) in the secondary blow-up
coordinates. To obtain a global limiting solution, we therefore need uniform
control of \(W_j\) on an expanding family of boxes. This is provided by the
following estimate.

\begin{lemma}
\label{lem:expanding-growth}
There exist constants \(C\geq1\), \(\gamma>0\), and \(c>0\), depending only
on \(m,M,H\), such that the functions defined in
\eqref{eq:secondary-blow-up} satisfy
\begin{equation}
\label{eq:polynomial-growth}
 \sup_{\widehat G_j\cap Q_{M,R}}W_j
 \leq CR^\gamma
\end{equation}
whenever
\[
 1\leq R\leq cr_j^{-1}
\]
and the physical boxes required in the corresponding one-sided Carleson
estimate are contained in the original domain of definition.
\end{lemma}

\begin{proof}
We prove the estimate with \(\gamma=\nu\), where \(\nu\) is the exponent in
Lemma~\ref{lem:one-sided-estimates}. Let \(C_0\) and \(c_0\) be the constants
in \eqref{eq:one-sided-Carleson}. Fix \(R\geq1\) and set $s=c_0Rr_j$. We assume that
\begin{equation}
\label{eq:growth-admissible-R}
 c_0Rr_j\leq s_*,
\end{equation}
so that \(r_j\leq s\leq s_*\) and all boxes needed below lie in the original
domain of definition. Applying \eqref{eq:one-sided-Carleson} at the boundary point \(P_j\) and
physical scale \(s\), and using
\[
 Q_{M,s/c_0}(P_j)=Q_{M,Rr_j}(P_j),
\]
we obtain
\[
 \sup_{\Omega_{\psi_j}\cap Q_{M,Rr_j}(P_j)}w_j
 \leq
 C_0w_j(A^+_{c_0Rr_j,\Lambda}(P_j)).
\]
Since \(r_j\leq c_0Rr_j\leq s_*\), the maximal-scale estimate
\eqref{eq:maximal-reference-growth} gives
\[
 w_j(A^+_{c_0Rr_j,\Lambda}(P_j))
 \leq
 (c_0R)^\nu
 w_j(A^+_{r_j,\Lambda}(P_j)).
\]
Consequently,
\begin{equation}
\label{eq:physical-expanding-growth}
 \sup_{\Omega_{\psi_j}\cap Q_{M,Rr_j}(P_j)}w_j
 \leq
 C_0c_0^\nu R^\nu
 w_j(A^+_{r_j,\Lambda}(P_j)).
\end{equation}
Under the change of variables
\[
 P=P_j\circ\delta_{r_j}(\widehat P),
\]
the box \(Q_{M,R}\) is mapped onto \(Q_{M,Rr_j}(P_j)\), and
\[
 \widehat\Omega_j
 =
 \delta_{1/r_j}\bigl(P_j^{-1}\circ\Omega_{\psi_j}\bigr).
\]
Dividing \eqref{eq:physical-expanding-growth} by the positive normalization
factor \(w_j(A^+_{r_j,\Lambda}(P_j))\) therefore yields
\[
 \sup_{\widehat G_j\cap Q_{M,R}}W_j
 \leq
 C_0c_0^\nu R^\nu.
\]
Finally, \eqref{eq:growth-admissible-R} holds whenever
\[
 1\leq R\leq cr_j^{-1},
 \qquad
 c=\frac{s_*}{c_0}.
\]
The choice of \(s_*\) ensures that all physical boxes used above are
contained in the original domain of definition. This proves
\eqref{eq:polynomial-growth}.
\end{proof}

We now record the compactness consequence of the expanding-growth estimate.

\begin{lemma}
\label{lem:secondary-blow-up-limit}
Let \(W_j\) be the functions defined in
\eqref{eq:secondary-blow-up}, with \(\widehat\psi_j\),
\(\widehat\Omega_j\), and \(\widehat G_j\) as above.
After passing to a subsequence, there exist an unbounded intrinsic
\(\operatorname{Lip}_{\mathcal K}\) graph domain
\(\Omega_\infty=\Omega_{\psi_\infty}\), whose defining function is
independent of \(y_m\), and a function
\(W_\infty\in C(\overline{\Omega_\infty})\) such that
\(\widehat\psi_j\to\psi_\infty\) locally uniformly and the zero extensions
of \(W_j\) converge locally uniformly to
\(W_\infty\mathbf 1_{\Omega_\infty}\). Here the zero extension is
taken below the graph within the expanding coordinate region on which
\(W_j\) is defined; every fixed compact set is contained in that region for
all sufficiently large \(j\). Moreover,
\begin{equation}
\label{eq:secondary-limit-equation}
 \mathcal KW_\infty=0
 \quad\text{in }\Omega_\infty,
 \qquad
 W_\infty=0
 \quad\text{on }\partial\Omega_\infty,
 \qquad
 W_\infty\geq0,
\end{equation}
and
\begin{equation}
\label{eq:secondary-limit-growth}
 \sup_{\Omega_\infty\cap Q_{M,R}}W_\infty
 \leq CR^\gamma,
 \qquad R\geq1.
\end{equation}
Finally,
\begin{equation}
\label{eq:secondary-limit-normalization}
 W_\infty(A^+_{1,\Lambda})=1,
 \qquad
 W_\infty(A^-_{1,\Lambda})=0.
\end{equation}
\end{lemma}

\begin{proof}
The rescaled graph functions \(\widehat\psi_j\) have a common intrinsic
Lipschitz constant \(M\) and satisfy \(\widehat\psi_j(0)=0\).
Proposition~\ref{prop:graph-compactness}, applied on successively larger
boxes, therefore gives, after a diagonal subsequence,
\(\widehat\psi_j\to\psi_\infty\) locally uniformly in all of
\(\mathbb R^{2m}\). The limit has intrinsic Lipschitz constant at most
\(M\). By Lemma~\ref{lem:cylindrical-defect-rescaling}, the cylindrical defect of
\(\widehat\psi_j\) on every fixed normalized box is the corresponding
defect of \(\psi_j\) at \(P_j\) and scale \(r_j\). These defects tend to
zero. Hence, for every fixed \(x',y',t,y_m,\widetilde y_m\),
\[
 \bigl|
 \widehat\psi_j(x',y',y_m,t)
 -
 \widehat\psi_j(x',y',\widetilde y_m,t)
 \bigr|
 \longrightarrow0.
\]
Passing to the locally uniform limit shows that
\[
 \psi_\infty(x',y',y_m,t)
 =
 \psi_\infty(x',y',\widetilde y_m,t).
\]
Thus \(\psi_\infty\) is independent of \(y_m\), and
\(\Omega_\infty\) is an unbounded cylindrical graph domain.

Fix \(R\geq1\). Since \(r_j\to0\), the range
\(1\leq R\leq cr_j^{-1}\) in Lemma~\ref{lem:expanding-growth} contains this
fixed \(R\) for all sufficiently large \(j\). Therefore,
\[
 \sup_{\widehat G_j\cap Q_{M,R}}W_j
 \leq CR^\gamma.
\]
Proposition~\ref{prop:solution-compactness}, applied on each fixed box,
gives local uniform convergence, after passage to a further diagonal
subsequence, of the zero extensions of \(W_j\) to
$W_\infty\mathbf 1_{\Omega_\infty}$. This application is legitimate because
the regions \(\widehat G_j\) contain every fixed box, together with the
fixed enlargement required by that proposition, for all sufficiently large
\(j\). The limit is non-negative, \(\mathcal K\)-harmonic in \(\Omega_\infty\),
continuous up to the graph, and vanishes there. This proves
\eqref{eq:secondary-limit-equation}. Passing to the limit in the preceding
growth estimate gives \eqref{eq:secondary-limit-growth}. Finally, the reference points \(A^\pm_{1,\Lambda}\) remain a fixed positive
distance from the limiting graph. Hence local uniform convergence and
\eqref{eq:secondary-normalization} give
\[
 W_\infty(A^+_{1,\Lambda})=1,
 \qquad
 W_\infty(A^-_{1,\Lambda})=0,
\]
which proves \eqref{eq:secondary-limit-normalization}.
\end{proof}
Lemma~\ref{lem:secondary-blow-up-limit} shows that a collapsing bad sequence
produces a non-negative \(\mathcal K\)-harmonic function of polynomial growth
in an unbounded cylindrical graph domain which vanishes at
\(A^-_{1,\Lambda}\) but takes the value one at \(A^+_{1,\Lambda}\). To
complete the contradiction, it remains to prove that no such function can
exist. This global rigidity statement is established in the next section.

\section{Global cylindrical rigidity}
\label{sec:global-rigidity}

The precise exponent in Lemma~\ref{lem:expanding-growth} is immaterial. What
matters is that \eqref{eq:polynomial-growth} is uniform in \(j\) and holds on
a family of boxes whose radii tend to infinity. As shown in
Lemma~\ref{lem:secondary-blow-up-limit}, this produces a non-negative
\(\mathcal K\)-harmonic function of polynomial growth in an unbounded
cylindrical graph domain, with value zero at the backward reference point and
value one at the forward reference point. The purpose of this section is to
show that such a configuration is impossible.

\subsection{The rigidity property}
\label{subsec:global-rigidity}

We first formulate the precise rigidity statement required by the
contradiction argument.

\begin{definition}
\label{def:global-rigidity}
Let \(\Omega_\psi\) be an unbounded intrinsic
\(\operatorname{Lip}_{\mathcal K}\) graph domain with constant \(M\), where
\(\psi(0)=0\) and \(\psi\) is independent of \(y_m\). Let
\(W\in C(\overline{\Omega_\psi})\) be non-negative and satisfy
\begin{equation}
\label{eq:global-rigidity-equation}
 \mathcal KW=0\quad\text{in }\Omega_\psi,
 \qquad
 W=0\quad\text{on }\partial\Omega_\psi.
\end{equation}
Assume, in addition, that \(W\) has polynomial growth: for some
\(C\geq1\) and \(\gamma\geq0\),
\begin{equation}
\label{eq:global-rigidity-growth}
 \sup_{\Omega_\psi\cap Q_{M,R}}W
 \leq C(1+R)^\gamma,
 \qquad R\geq1.
\end{equation}
We say that the global cylindrical rigidity property holds if every such
function satisfies
\begin{equation}
\label{eq:global-rigidity-implication}
 W(A^-_{1,\Lambda})=0
 \quad\Longrightarrow\quad
 W(A^+_{1,\Lambda})=0.
\end{equation}
\end{definition}

Only the implication \eqref{eq:global-rigidity-implication}, rather than a
complete classification of polynomial-growth solutions, is needed here. It
cannot be obtained directly from the strong minimum principle. Indeed, Bony's
principle yields only
\begin{equation}
\label{eq:zero-on-propagation-set}
 W=0
 \quad\text{on }
 \overline{\mathcal A_{A^-_{1,\Lambda}}(\Omega_\psi)},
\end{equation}
where \(\mathcal A_P(\Omega_\psi)\) denotes the propagation set issuing from
\(P\). This is a causal, one-sided conclusion: the Kolmogorov propagation set
is directed backward in time and need not contain an entire Euclidean time
section. Consequently, \eqref{eq:zero-on-propagation-set} does not provide
zero Cauchy data on a complete section of the unbounded domain, and standard
forward uniqueness cannot be invoked at this stage; see
\cite{Bony1969,CintiNystromPolidoro2010}.

Exact cylindricity nevertheless gives considerably more information about the
zero set. Denote by
\[
 \pi(x',x_m,y',y_m,t)=(x',x_m,y',t)
\]
the projection which suppresses the invariant variable.

\subsection{Propagation along invariant fibres}
\label{subsec:fibre-propagation}

\begin{lemma}
\label{lem:fibre-half-ray}
Let \(\Omega_\psi\) be as in Definition~\ref{def:global-rigidity}, and let
\(P=(x^0,y^0,t_0)\in\Omega_\psi\). Given \(\bar t<t_0\) and
\[
 \bar\zeta=(\bar x',\bar x_m,\bar y',\bar t)
 \in \pi(\Omega_\psi\cap\{t=\bar t\}),
\]
there exists \(b=b(P,\bar\zeta)\in\mathbb R\) such that
\begin{equation}
\label{eq:fibre-half-ray-propagation}
 (\bar x',\bar x_m,\bar y',s,\bar t)
 \in \mathcal A_P(\Omega_\psi)
 \qquad\text{for every }s>b.
\end{equation}
Consequently, if \(W\) satisfies \eqref{eq:global-rigidity-equation} and
\(W(P)=0\), then
\begin{equation}
\label{eq:fibre-half-ray-vanishing}
 W(\bar x',\bar x_m,\bar y',s,\bar t)=0
 \qquad\text{for every }s\geq b.
\end{equation}
\end{lemma}
\begin{proof}
Put \(T=t_0-\bar t>0\). We use the admissible-curve and
propagation-set formalism developed in
\cite{CintiNystromPolidoro2010,NystromPolidoro2016}. In the present
parametrization, the control system is
\begin{equation}
\label{eq:control-system-fibre}
 \dot x(\tau)=\omega(\tau),\qquad
 \dot y(\tau)=x(\tau),\qquad
 \dot t(\tau)=-1,
 \qquad 0\leq\tau\leq T.
\end{equation}
The explicit construction below is a controllability interpolation for this
system. Its additional use here is that, because the domain is invariant in
the \(y_m\)-direction, the terminal \(y_m\)-coordinate can be varied without
changing whether the curve remains in the domain.

We first connect $(x^{0\prime},y^{0\prime},t_0)$ to $(\bar x',\bar y',\bar t)$. Set $I'=\bar y'-y^{0\prime}$ and define
\[
 c'
 :=
 \frac{6}{T^3}
 \left(
 I'
 -
 \frac{T}{2}\bigl(x^{0\prime}+\bar x'\bigr)
 \right).
\]
We choose
\begin{equation}
\label{eq:tangential-control-polynomial}
 x'(\tau)
 =
 x^{0\prime}
 +
 \frac{\tau}{T}\bigl(\bar x'-x^{0\prime}\bigr)
 +
 c'\tau(T-\tau),
 \qquad 0\leq\tau\leq T.
\end{equation}
This is a vector-valued quadratic polynomial, interpreted componentwise. It
satisfies
\[
 x'(0)=x^{0\prime},
 \qquad
 x'(T)=\bar x'.
\]
Moreover,
\begin{align*}
 \int_0^T x'(\tau)\,\dd\tau
 &=
 \frac{T}{2}\bigl(x^{0\prime}+\bar x'\bigr)
 +
 c'\int_0^T\tau(T-\tau)\,\dd\tau \\
 &=
 \frac{T}{2}\bigl(x^{0\prime}+\bar x'\bigr)
 +
 \frac{T^3}{6}c'
 =
 I'
 =
 \bar y'-y^{0\prime}.
\end{align*}
Thus, on defining
\[
 y'(\tau)
 =
 y^{0\prime}+\int_0^\tau x'(s)\,\dd s,
 \qquad
 t(\tau)=t_0-\tau,
\]
we obtain
\[
 y'(0)=y^{0\prime},
 \qquad
 y'(T)=\bar y',
 \qquad
 t(0)=t_0,
 \qquad
 t(T)=\bar t.
\]
The corresponding control in the \(x'\)-variables is
\(\omega'=\dot x'\), which is smooth and therefore square integrable.

We next connect $x_m^0$ to $\bar x_m$. Since \(\psi\) is independent of \(y_m\), the function
\[
 g(\tau)=\psi(x'(\tau),y'(\tau),t(\tau))
\]
does not depend on the \(y_m\)-component of the curve and is continuous.
At \(\tau=0\), the endpoint identities and the inclusion
\(P\in\Omega_\psi\) give
\[
 x_m^0>
 \psi(x^{0\prime},y^{0\prime},y_m^0,t_0)
 =
 \psi(x^{0\prime},y^{0\prime},t_0)
 =
 g(0).
\]
Moreover,
\[
 (\bar x',\bar x_m,\bar y',\bar t)
 \in\pi(\Omega_\psi\cap\{t=\bar t\})
\]
means that, for some \(s\in\mathbb R\),
\[
 (\bar x',\bar x_m,\bar y',s,\bar t)\in\Omega_\psi.
\]
Using the independence of \(\psi\) from \(y_m\) and the terminal endpoint
identities, we obtain
\[
 \bar x_m>
 \psi(\bar x',\bar y',s,\bar t)
 =
 \psi(\bar x',\bar y',\bar t)
 =
 g(T).
\]
Consequently,
\[
 d_0:=x_m^0-g(0)>0,
 \qquad
 d_T:=\bar x_m-g(T)>0.
\]

Define
\[
 d(\tau)
 =
 \left(1-\frac{\tau}{T}\right)d_0
 +
 \frac{\tau}{T}d_T,
 \qquad 0\leq\tau\leq T.
\]
Then \(d(\tau)>0\) on \([0,T]\), and the continuous function
\[
 h(\tau)=g(\tau)+d(\tau)
\]
satisfies
\[
 h(0)=x_m^0,\qquad h(T)=\bar x_m,
 \qquad h(\tau)>g(\tau).
\]
Since \(d\) is continuous and strictly positive on the compact interval
\([0,T]\), we have
\[
 d_*:=\min_{0\leq\tau\leq T}d(\tau)>0.
\]
The function \(h=g+d\) is continuous on \([0,T]\) and satisfies
\[
 h(0)=x_m^0,
 \qquad
 h(T)=\bar x_m.
\]
We now approximate \(h\) by a smooth function while preserving these endpoint
values. Choose \(q\in C^\infty([0,T])\) such that
\[
 \|q-h\|_{L^\infty(0,T)}<\frac14d_*,
\]
and define
\[
 \xi_m^0(\tau)
 =
 q(\tau)
 +
 \left(1-\frac{\tau}{T}\right)\bigl(h(0)-q(0)\bigr)
 +
 \frac{\tau}{T}\bigl(h(T)-q(T)\bigr).
\]
The affine correction preserves smoothness and gives
\[
 \xi_m^0(0)=h(0)=x_m^0,
 \qquad
 \xi_m^0(T)=h(T)=\bar x_m.
\]
Moreover,
\[
 \begin{aligned}
 |\xi_m^0(\tau)-h(\tau)|
 &\leq
 |q(\tau)-h(\tau)|
 +
 \left(1-\frac{\tau}{T}\right)|h(0)-q(0)|
 +
 \frac{\tau}{T}|h(T)-q(T)| <
 \frac12d_*,
 \end{aligned}
\]
after choosing the initial smooth approximation slightly more accurately if
necessary. Consequently,
\[
 \|\xi_m^0-h\|_{L^\infty(0,T)}
 <
 \frac12\min_{[0,T]}d.
\]
Since \(h(\tau)=g(\tau)+d(\tau)\) and \(d(\tau)\geq d_*\), it follows that
\[
 \xi_m^0(\tau)
 >
 h(\tau)-\frac12d_*
 \geq
 g(\tau)+\frac12d_*
 >
 g(\tau),
 \qquad 0\leq\tau\leq T.
\]
Thus the smoothed \(x_m\)-component retains a uniform positive distance from
the graph while joining the prescribed initial and terminal values.

Let \(\varphi\in C_0^\infty(0,T)\) be non-negative and satisfy
\[
 \int_0^T\varphi(\tau)\,\dd\tau=1.
\]
For \(a\geq0\), set
\[
 x_m^a(\tau)=\xi_m^0(\tau)+a\varphi(\tau),
 \qquad
 y_m^a(\tau)
 =
 y_m^0+\int_0^\tau x_m^a(s)\,\dd s.
\]
Consider the curve
\[
 \gamma_a(\tau)
 =
 \bigl(
 x'(\tau),x_m^a(\tau),
 y'(\tau),y_m^a(\tau),t(\tau)
 \bigr).
\]
It satisfies \eqref{eq:control-system-fibre} with control
\[
 \omega_a(\tau)
 =
 \bigl(\dot x'(\tau),\dot x_m^a(\tau)\bigr),
\]
which is smooth and hence square integrable. Since
\[
 x_m^a(\tau)
 \geq\xi_m^0(\tau)
 >g(\tau)
 =
 \psi(x'(\tau),y'(\tau),t(\tau)),
\]
the curve remains in \(\Omega_\psi\).

Because \(\varphi\) is compactly supported in \((0,T)\), adding
\(a\varphi\) does not alter the endpoint values of the \(x_m\)-component.
The resulting curve therefore has the prescribed initial point and the
prescribed terminal values in every coordinate except possibly \(y_m\).
Its terminal \(y_m\)-coordinate is
\begin{align*}
 y_m^a(T)
 =
 y_m^0+\int_0^T x_m^a(\tau)\,\dd\tau &=y_m^0+\int_0^T\xi_m^0(\tau)\,\dd\tau
 +a\int_0^T\varphi(\tau)\,\dd\tau
 \\
 &=
 y_m^0+\int_0^T\xi_m^0(\tau)\,\dd\tau+a.
\end{align*}
Set
\[
 b
 :=
 y_m^0+\int_0^T\xi_m^0(\tau)\,\dd\tau.
\]
Then \(y_m^a(T)=b+a\). Consequently, for every \(s\geq b\), choosing
\(a=s-b\) produces an admissible curve from \(P\) to $(\bar x',\bar x_m,\bar y',s,\bar t)$. Thus the terminal \(y_m\)-coordinates fill the half-ray \([b,\infty)\), which
proves \eqref{eq:fibre-half-ray-propagation}. Notice that \(b\) depends on
the chosen interpolation and on the two endpoints; no quantitative bound for
\(b\) is needed. The final assertion follows from the Bony minimum principle
and the continuity of \(W\). \end{proof}

\begin{remark}
\label{rem:fibre-half-ray-geometry}
The conclusion of Lemma~\ref{lem:fibre-half-ray} has a simple geometric
interpretation. For $\bar\zeta=(\bar x',\bar x_m,\bar y',\bar t)$, consider the \(y_m\)-fibre
\[
 \mathcal F_{\bar\zeta}
 =
 \left\{
 (\bar x',\bar x_m,\bar y',s,\bar t):s\in\mathbb R
 \right\}.
\]
Since \(\psi\) is independent of \(y_m\), the defining inequality $\bar x_m>\psi(\bar x',\bar y',\bar t)$ does not depend on the fibre coordinate \(s\).   Hence, if one point of
\(\mathcal F_{\bar\zeta}\) belongs to \(\Omega_\psi\), then the entire fibre
does. Equivalently,
\[
 \bar\zeta\in\pi(\Omega_\psi\cap\{t=\bar t\})
 \quad\Longrightarrow\quad
 \mathcal F_{\bar\zeta}\subset\Omega_\psi.
\]
The lemma states that, at every earlier time \(\bar t<t_0\), the propagation
set from \(P\) contains a terminal half-ray in every such fibre:
\[
 \mathcal A_P(\Omega_\psi)\cap\mathcal F_{\bar\zeta}
 \supset
 \left\{
 (\bar x',\bar x_m,\bar y',s,\bar t):s>b
 \right\}.
\]
This one-sided filling of the fibre reflects the control system $\dot y_m=x_m$. Once a path with the prescribed projected endpoint has been constructed, a
non-negative bump can be added to its \(x_m\)-component. The endpoints in
the \(x_m\)-variable remain fixed, the path remains above the graph, and the
terminal \(y_m\)-coordinate increases by the integral of the bump. By varying the parameter \(a\geq0\), the amplitude of the added bump
\(a\varphi\), the terminal coordinate becomes
\[
 y_m^a(T)=b+a.
\]
Hence every \(s\geq b\) is attained by choosing \(a=s-b\), and the reachable
terminal values of \(y_m\) contain the half-ray \([b,\infty)\). The lemma does
not assert that the entire fibre belongs to the propagation set, since the
constraint imposed by the graph generally prevents arbitrary displacement in
the opposite direction. Thus the strong minimum principle initially yields
only a half-ray of zeros in each invariant fibre. The subsequent tangential
analyticity argument extends this one-sided zero set to the remainder of the
fibre.
\end{remark}

The purpose of the rest of the section is to prove that polynomial intrinsic growth forces finite-dimensionality of
the translation orbit in the invariant variable.  The conclusion is slightly
different from tangential polynomiality: the fibres are exponential
polynomials.  This is the natural conclusion of the translation argument and
is sufficient, since an exponential polynomial which vanishes on a half-line
vanishes identically.

\subsection{Boundary energy and zero extension}
\label{subsec:boundary-energy-zero-extension}

Throughout this subsection, \(\Omega_\psi\) is an intrinsic Lipschitz graph
domain whose defining function is independent of \(y_m\); thus
\(\psi=\psi(x',y',t)\). We isolate the boundary-energy properties needed in
the rigidity argument. Although the slicewise energy spaces and
\(x\)-Sobolev traces introduced below are meaningful also for graph functions
with genuine \(y_m\)-dependence, the results proved in this subsection are
used only in the \(y_m\)-independent setting. This independence enters through
the boundary comparison and boundary-layer estimates that yield the required
up-to-boundary energy bounds.

Let \(\mathcal U=B_x\times E_{y,t}\) be a bounded product neighborhood and
put $D=\Omega_\psi\cap\mathcal U$. For almost every \((y,t)\in E_{y,t}\), set
\[
 D(y,t)=\{x\in B_x:(x,y,t)\in D\}.
\]
We write
\[
 u\in\mathrm L^2_{y,t}(\mathrm H_x^1(D))
\]
if
\[
 \int_{E_{y,t}}
 \|u(\cdot,y,t)\|_{\mathrm H^1(D(y,t))}^2\,\dd y\,\dd t
 <\infty,
\]
and use
\(\mathrm L^2_{y,t}(\mathrm H_x^{-1}(D))\) for the corresponding measurable
family of dual spaces. The local kinetic energy space is
\begin{equation}
\label{eq:kinetic-boundary-energy-space}
 \mathbb W(D)
 =
 \left\{
 u\in\mathrm L^2_{y,t}(\mathrm H_x^1(D)):
 Yu\in\mathrm L^2_{y,t}(\mathrm H_x^{-1}(D))
 \right\},
\end{equation}
where \(Y=x\cdot\nabla_y-\partial_t\) is interpreted distributionally.

For \(u\in\mathbb W(D)\), a zero boundary-energy trace on the graph portion
means that, for almost every \((y,t)\), the ordinary
\(\mathrm H_x^1\)-trace of \(u(\cdot,y,t)\) vanishes on the corresponding
Lipschitz graph in the \(x\)-variables. No trace in the \(y\)- or
\(t\)-variables is asserted. For the solution considered in the proposition
below, the admissibility of products with smooth test functions will be
established directly by means of cutoffs supported away from the graph, with
the resulting boundary-layer errors shown to vanish. Thus no abstract density
assertion in the full graph norm is required.

The proposition is stated for arbitrary real-valued solutions, although the
solution obtained from the secondary blow-up is non-negative, being a locally
uniform limit of non-negative solutions. The more general formulation is
needed later when the proposition is applied to the vector space
\(\mathscr H_\gamma(\Omega_\psi)\) introduced in Subsection \ref{subsec:tangential-analyticity}. In this vector spaces, the elements and their linear combinations
need not have a fixed sign. For a non-negative solution, the comparison
argument applies directly; for a signed solution, we first construct a
positive $\mathcal K$-harmonic majorant of its absolute value.

We first record the boundary-layer estimate that makes the cutoff
approximation possible.

\begin{lemma}
\label{lem:vanishing-boundary-layer}
Let \(\Omega_\psi\) be an intrinsic Lipschitz graph domain whose defining
function is independent of \(y_m\). Suppose that
\(U\in C(\overline{\Omega_\psi})\) is \(\mathcal K\)-harmonic in
\(\Omega_\psi\) and vanishes on \(\partial\Omega_\psi\). Fix \(P_0\in\partial\Omega_\psi\) and \(r>0\). Let
\(\phi_r\in C_0^\infty(Q_{M,2r}(P_0))\) satisfy
\[
 \phi_r\equiv1\quad\text{on }Q_{M,r}(P_0),
 \qquad
 r|\nabla_x\phi_r|+r^2|Y\phi_r|\leq C.
\]
For \(\varepsilon>0\), set
\[
 D_{2r}^{\varepsilon}(P_0)
 =
 \left\{
 P\in\Omega_\psi\cap Q_{M,2r}(P_0):
 d_{\mathcal K}^{\mathrm s}(P,\partial\Omega_\psi)>\varepsilon
 \right\}
\]
and
\[
 S_\varepsilon(P_0)
 =
 D_{2r}^{\varepsilon/2}(P_0)
 \setminus
 D_{2r}^{\varepsilon}(P_0).
\]
Then
\begin{equation}
\label{eq:vanishing-boundary-layer}
 J_\varepsilon(U;P_0,r)
 :=
 \frac1{\varepsilon^2}
 \int_{S_\varepsilon(P_0)}U^2\phi_r^2
 \longrightarrow0
 \qquad\text{as }\varepsilon\downarrow0.
\end{equation}
\end{lemma}

\begin{proof}
The argument is local. Choose a bounded intrinsic graph truncation
\(\widetilde D\subset\Omega_\psi\) such that
\[
 \Omega_\psi\cap Q_{M,4r}(P_0)\subset\widetilde D.
\]
We choose \(\widetilde D\) from the class of bounded truncations for which the
continuous Dirichlet problem and the comparison principle stated above are
valid. Since \(U\in C(\overline{\Omega_\psi})\), its restriction to
\(\partial_{\mathcal K}\widetilde D\) is continuous. Define
\[
 \varphi=|U|
 \quad\text{on }\partial_{\mathcal K}\widetilde D.
\]
On the physical graph portion one has \(\varphi=0\), because \(U\) vanishes
there. Thus the boundary data are continuous also at the intersections of the
physical and artificial boundary portions.

Let \(W\) be the continuous Dirichlet solution
\[
 \begin{cases}
  \mathcal KW=0 & \text{in }\widetilde D,\\
  W=\varphi     & \text{on }\partial_{\mathcal K}\widetilde D.
 \end{cases}
\]
The comparison principle gives \(W\geq0\). Moreover, on
\(\partial_{\mathcal K}\widetilde D\),
\[
 U\leq |U|=W,
 \qquad
 -U\leq |U|=W.
\]
Since both \(U\) and \(-U\) are \(\mathcal K\)-harmonic, applying the
comparison principle to the pairs \((U,W)\) and \((-U,W)\) yields
\[
 |U|\leq W
 \qquad\text{in }\widetilde D,
\]
and hence in particular
\[
 |U|\leq W
 \qquad\text{in }\Omega_\psi\cap Q_{M,4r}(P_0).
\]
If \(W\equiv0\), then \(U\equiv0\) in the comparison region and there is
nothing to prove. We may therefore assume that \(W\not\equiv0\).

Choose a Green function \(g\) for \(\mathcal K\), with pole outside
\(Q_{M,4r}(P_0)\) and in the appropriate causal direction, so that \(g\) is
positive and \(\mathcal K\)-harmonic in the comparison region and vanishes on
its physical graph boundary. Put
\[
 \widetilde W=W+g.
\]
Then \(\widetilde W\) is non-negative and \(\mathcal K\)-harmonic in the
comparison region, vanishes on the physical graph boundary, and satisfies
\[
 \widetilde W(A^\pm_{\rho_0,\Lambda}(P_0))
 \geq
 g(A^\pm_{\rho_0,\Lambda}(P_0))
 >0,
\]
where \(\rho_0\asymp r\) is the reference scale occurring in the localized
boundary comparison theorem. Thus both \(\widetilde W\) and \(g\) satisfy
the required positivity assumptions at the forward and backward reference
points.

Since the graph is independent of \(y_m\), the localized form of the
cylindrical boundary comparison theorem
\cite[Theorem~3.5]{LitsgardNystrom2022} applies to \(\widetilde W\) and
\(g\). After normalizing \(g\) at a fixed interior reference point, the
theorem gives
\begin{equation}
\label{eq:majorant-Green-comparison}
 |U|
 \leq W
 \leq \widetilde W
 \leq C_{U,r}g
 \qquad\text{in }\Omega_\psi\cap Q_{M,2r}(P_0).
\end{equation}
Here \(C_{U,r}\) may depend on the corresponding reference values of
\(\widetilde W\) and \(g\). Its precise value is irrelevant; only its
finiteness is needed.

The forward counterpart of
\cite[Lemma~8.2]{LitsgardNystrom2022}, obtained by interchanging
\(\mathcal K\) and \(\mathcal K^*\) in the proof, introduces, up to harmless
changes in the box constants, the same boundary layer
\(S_\varepsilon(P_0)\). Applied to \(g\), the corresponding boundary-layer
quantity is
\[
 J_\varepsilon(g;P_0,r)
 =
 \frac1{\varepsilon^2}
 \int_{S_\varepsilon(P_0)}g^2\phi_r^2.
\]
The argument following equation~(8.10) in
\cite{LitsgardNystrom2022}, based on boundary H\"older decay, the Carleson
estimate, and a Whitney decomposition, yields
\begin{equation}
\label{eq:Green-boundary-layer}
 \frac1{\varepsilon^2}
 \int_{S_\varepsilon(P_0)}g^2\phi_r^2
 \leq
 C_{g,r}
 \left(\frac{\varepsilon}{r}\right)^\alpha
 \longrightarrow0
 \qquad\text{as }\varepsilon\downarrow0
\end{equation}
for some \(\alpha>0\). Consequently,
\eqref{eq:majorant-Green-comparison} and
\eqref{eq:Green-boundary-layer} imply
\[
 J_\varepsilon(U;P_0,r)
 \leq
 C_{U,r}^2
 \frac1{\varepsilon^2}
 \int_{S_\varepsilon(P_0)}g^2\phi_r^2
 \longrightarrow0.
\]
This proves \eqref{eq:vanishing-boundary-layer}.
\end{proof}

\begin{proposition}
\label{prop:boundary-energy-zero-extension}
Let \(\Omega_\psi\) be an intrinsic Lipschitz graph domain whose defining
function is independent of \(y_m\). Suppose that
\(U\in C(\overline{\Omega_\psi})\) is \(\mathcal K\)-harmonic in
\(\Omega_\psi\) and vanishes on \(\partial\Omega_\psi\). Then, on every
bounded graph neighborhood,
\[
 U\in\mathrm L^2_{y,t}(\mathrm H_x^1),
 \qquad
 YU\in\mathrm L^2_{y,t}(\mathrm H_x^{-1}),
\]
and \(U\) has zero \(x\)-Sobolev trace on the graph. Moreover, for every
\(\phi\in C_0^\infty(\mathbb R^{2m+1})\), the test function \(U\phi\) is
admissible by approximation with test functions supported a positive
distance from the graph. If
\[
 V=U^2\mathbf 1_{\Omega_\psi},
\]
then, for every non-negative
\(\phi\in C_0^\infty(\mathbb R^{2m+1})\),
\begin{equation}
\label{eq:rigorous-zero-extension-identity}
 \langle\mathcal KV,\phi\rangle
 =
 2\int_{\Omega_\psi}|\nabla_xU|^2\phi
 \geq0.
\end{equation}
Thus the zero extension of \(U^2\) is a global weak subsolution.
\end{proposition}

\begin{proof}
It suffices to work in one bounded graph neighborhood. Fix
\(P_0\in\partial\Omega_\psi\) and \(r>0\), and use the notation of
Lemma~\ref{lem:vanishing-boundary-layer}. For every sufficiently small
\(\varepsilon>0\), choose
\(\eta_\varepsilon\in C^\infty(\Omega_\psi)\) such that
\[
 \eta_\varepsilon=0
 \quad\text{if }
 d_{\mathcal K}^{\mathrm s}(P,\partial\Omega_\psi)
 \leq\frac{\varepsilon}{2},\qquad
 \eta_\varepsilon=1
 \quad\text{if }
 d_{\mathcal K}^{\mathrm s}(P,\partial\Omega_\psi)
 \geq\varepsilon,
\]
and
\begin{equation}
\label{eq:boundary-cutoff-bounds}
 \varepsilon|\nabla_x\eta_\varepsilon|
 +
 \varepsilon^2|Y\eta_\varepsilon|
 \leq C.
\end{equation}
The cutoff can be constructed by a partition of unity subordinate to a
Whitney decomposition of \(\Omega_\psi\), as in the proof of
\cite[Lemma~8.2]{LitsgardNystrom2022}. In particular, the derivatives of
\(\eta_\varepsilon\) are supported in
\(S_\varepsilon(P_0)\) inside \(Q_{M,2r}(P_0)\).

Since \(\eta_\varepsilon\) vanishes in a neighborhood of the graph and
\(U\) is smooth in the interior, the function
\(U\eta_\varepsilon^2\phi_r^2\) is an admissible interior test function. Put
\[
 E_\varepsilon
 =
 \int_{\Omega_\psi\cap Q_{M,2r}(P_0)}
 |\nabla_xU|^2\eta_\varepsilon^2\phi_r^2.
\]
Writing
\begin{align*}
 E_\varepsilon
 &=
 \int_{\Omega_\psi}
 \nabla_xU\cdot
 \nabla_x\bigl(U\eta_\varepsilon^2\phi_r^2\bigr)-
 \int_{\Omega_\psi}
 U\nabla_xU\cdot
 \nabla_x\bigl(\eta_\varepsilon^2\phi_r^2\bigr)
 =:I_{1,\varepsilon}+I_{2,\varepsilon},
\end{align*}
and using \(\mathcal KU=0\), the divergence-free identity
\[
 \operatorname{div}_{x,y,t}(0_x,x,-1)=0
\]
for the transport field \(Y=x\cdot\nabla_y-\partial_t\),  and the fact that
the test function is compactly supported in \(\Omega_\psi\), we obtain
\[
 |I_{1,\varepsilon}|
 \leq
 C
 \int_{\Omega_\psi}
 U^2
 \left|
 Y\bigl(\eta_\varepsilon^2\phi_r^2\bigr)
 \right|.
\]
The cutoff bounds give
\begin{equation}
\label{eq:I1-boundary-energy}
 |I_{1,\varepsilon}|
 \leq
 C J_\varepsilon(U;P_0,r)
 +
 \frac{C}{r^2}
 \int_{\Omega_\psi\cap Q_{M,2r}(P_0)}U^2.
\end{equation}
Similarly, Cauchy-Schwarz and
\eqref{eq:boundary-cutoff-bounds} give, for every
\(\delta\in(0,1)\),
\begin{equation}
\label{eq:I2-boundary-energy}
 |I_{2,\varepsilon}|
 \leq
 \delta E_\varepsilon
 +
 C_\delta J_\varepsilon(U;P_0,r)
 +
 \frac{C_\delta}{r^2}
 \int_{\Omega_\psi\cap Q_{M,2r}(P_0)}U^2.
\end{equation}
Combining \eqref{eq:I1-boundary-energy} and
\eqref{eq:I2-boundary-energy}, and choosing \(\delta\) sufficiently small,
we obtain
\begin{equation}
\label{eq:boundary-Caccioppoli-cutoff}
 E_\varepsilon
 \leq
 C J_\varepsilon(U;P_0,r)
 +
 \frac{C}{r^2}
 \int_{\Omega_\psi\cap Q_{M,2r}(P_0)}U^2.
\end{equation}
Lemma~\ref{lem:vanishing-boundary-layer} gives
\[
 J_\varepsilon(U;P_0,r)\longrightarrow0.
\]
Letting \(\varepsilon\downarrow0\) in
\eqref{eq:boundary-Caccioppoli-cutoff} and using Fatou's lemma yields
\begin{equation}
\label{eq:boundary-Caccioppoli-U}
 \int_{\Omega_\psi\cap Q_{M,r}(P_0)}|\nabla_xU|^2
 \leq
 \frac{C}{r^2}
 \int_{\Omega_\psi\cap Q_{M,2r}(P_0)}U^2.
\end{equation}
Consequently, $U\in\mathrm L^2_{y,t}(\mathrm H_x^1)$ locally up to the graph. Since $YU=-\Delta_xU$ in the sense of distributions, we also have
\[
 \|YU\|_{\mathrm L^2_{y,t}(\mathrm H_x^{-1})}
 \leq
 \|\nabla_xU\|_{\mathrm L^2},
\]
locally on every bounded graph neighborhood. Hence
\[
 YU\in\mathrm L^2_{y,t}(\mathrm H_x^{-1}).
\]
The continuity of \(U\) up to the graph, together with its slicewise
\(\mathrm H_x^1\)-regularity, implies that its \(x\)-Sobolev trace vanishes
there.

It remains to establish the zero-extension identity. Let
\(\phi\in C_0^\infty(\mathbb R^{2m+1})\). By using a finite covering, it
suffices to consider the part of its support contained in one graph
neighborhood. For every \(\varepsilon>0\), the function
\(U\phi\eta_\varepsilon^2\) is an admissible interior test function. The terms in which a derivative
falls on \(\eta_\varepsilon\) are supported in
\(S_\varepsilon(P_0)\). By \eqref{eq:boundary-cutoff-bounds},
\[
 \left|
 \int_{\Omega_\psi}
 U^2\phi\,Y(\eta_\varepsilon^2)
 \right|
 \leq
 C_\phi J_\varepsilon(U;P_0,r)
 \longrightarrow0.
\]
The diffusive cutoff error satisfies
\begin{align*}
 &\left|
 \int_{\Omega_\psi}
 U\phi\eta_\varepsilon
 \nabla_xU\cdot\nabla_x\eta_\varepsilon
 \right|\leq
 C_\phi
 J_\varepsilon(U;P_0,r)^{1/2}
 \left(
 \int_{S_\varepsilon(P_0)}|\nabla_xU|^2
 \right)^{1/2}
 \longrightarrow0.
\end{align*}
Here the second factor is bounded by
\eqref{eq:boundary-Caccioppoli-U}, while the first factor tends to zero by
Lemma~\ref{lem:vanishing-boundary-layer}. All remaining terms converge by
the local \(\mathrm L^2\)-integrability of \(U\) and \(\nabla_xU\).

Testing \(\mathcal KU=0\) with
\(U\phi\eta_\varepsilon^2\) and letting
\(\varepsilon\downarrow0\) therefore gives
\begin{align}
0
&=
-\int_{\Omega_\psi}|\nabla_xU|^2\phi
-\int_{\Omega_\psi}U\nabla_xU\cdot\nabla_x\phi
-\frac12\int_{\Omega_\psi}U^2Y\phi
\nonumber\\
&=
-\int_{\Omega_\psi}|\nabla_xU|^2\phi
+\frac12\int_{\Omega_\psi}U^2
\bigl(\Delta_x\phi-Y\phi\bigr).
\label{eq:zero-extension-energy-identity}
\end{align}
Since $\mathcal K^*=\Delta_x-Y$,
identity \eqref{eq:zero-extension-energy-identity} is precisely
\eqref{eq:rigorous-zero-extension-identity}. No boundary term involving
\(\nabla_xU\) has been used.
\end{proof}

We finally record why the preceding properties apply to the cylindrical
solution obtained in the secondary blow-up. No boundary-energy assertion is
required for the genuinely \(y_m\)-dependent approximating domains.

\begin{lemma}
\label{lem:boundary-energy-cylindrical-limit}
Let \(\Omega_{\psi_j}\) be intrinsic Lipschitz graph domains with a common
Lipschitz constant, and assume that
\[
 \psi_j\longrightarrow\psi_\infty
 \qquad\text{locally uniformly},
\]
where \(\psi_\infty\) is independent of \(y_m\). Let \(U_j\) be non-negative
continuous solutions defined in
\(\Omega_{\psi_j}\cap\mathcal O_j\), where the open sets
\(\mathcal O_j\) exhaust \(\mathbb R^{2m+1}\) in the sense that every compact
set is contained in \(\mathcal O_j\) for all sufficiently large \(j\). Assume
that
\[
 \mathcal KU_j=0\quad\text{in }\Omega_{\psi_j}\cap\mathcal O_j,
 \qquad
 U_j=0\quad\text{on }
 \partial\Omega_{\psi_j}\cap\mathcal O_j.
\]
Extend \(U_j\) by zero below the graph within \(\mathcal O_j\), and denote
these extensions by
\[
 U_j^0=U_j\mathbf 1_{\Omega_{\psi_j}},
\]
and assume that
\[
 U_j^0\longrightarrow U_\infty^0
 =
 U_\infty\mathbf 1_{\Omega_{\psi_\infty}}
 \qquad\text{locally uniformly in }\mathbb R^{2m+1}.
\]
Here and below, local convergence for the sequence is understood on each
compact set after that set is contained in \(\mathcal O_j\).
Then \(U_\infty\) is non-negative and continuous on
\(\overline{\Omega_{\psi_\infty}}\), satisfies
\[
 \mathcal KU_\infty=0
 \quad\text{in }\Omega_{\psi_\infty},
 \qquad
 U_\infty=0
 \quad\text{on }\partial\Omega_{\psi_\infty},
\]
and belongs locally up to the graph to the boundary energy space. More precisely, let $P_*\in\partial\Omega_{\psi_\infty}$ and let $\mathcal U=B_x\times E_{y,t}$ be any bounded product neighborhood of \(P_*\). Setting
\[
 D_\infty=\Omega_{\psi_\infty}\cap\mathcal U,
\]
one has
\[
 U_\infty\in
 \mathrm L^2_{y,t}\bigl(\mathrm H_x^1(D_\infty)\bigr),
 \qquad
 YU_\infty\in
 \mathrm L^2_{y,t}\bigl(\mathrm H_x^{-1}(D_\infty)\bigr).
\]
Moreover, \(U_\infty\) has zero \(x\)-Sobolev trace on
\(\partial\Omega_{\psi_\infty}\cap\mathcal U\) and,
for every non-negative
\(\phi\in C_0^\infty(\mathbb R^{2m+1})\),
\begin{equation}
\label{eq:cylindrical-limit-zero-extension}
 \left\langle
 \mathcal K
 \bigl(U_\infty^2\mathbf 1_{\Omega_{\psi_\infty}}\bigr),
 \phi
 \right\rangle
 =
 2\int_{\Omega_{\psi_\infty}}
 |\nabla_xU_\infty|^2\phi
 \geq0.
\end{equation}
\end{lemma}

\begin{proof}
Since the zero extensions \(U_j^0\) are non-negative and converge locally
uniformly, their limit \(U_\infty^0\) is non-negative and continuous. In
particular, \(U_\infty\) is continuous up to the limiting graph and vanishes
there. The graph convergence, the exhaustion by \(\mathcal O_j\), and the
local uniform convergence of the solutions imply, by
Proposition~\ref{prop:solution-compactness} applied on each fixed box, that
\(U_\infty\) is \(\mathcal K\)-harmonic in
\(\Omega_{\psi_\infty}\). Since \(\psi_\infty\) is independent of \(y_m\),
the limiting domain is cylindrical, and
Proposition~\ref{prop:boundary-energy-zero-extension} applies directly to
\(U_\infty\). It gives the asserted boundary-energy regularity, the zero
\(x\)-Sobolev trace, the admissibility of the boundary energy calculation,
and \eqref{eq:cylindrical-limit-zero-extension}.
\end{proof}

\subsection{Polynomial growth and tangential analyticity}
\label{subsec:tangential-analyticity}

Throughout this subsection, \(\Omega_\psi\) is an unbounded intrinsic
\(\operatorname{Lip}_{\mathcal K}\) graph domain as in
Definition~\ref{def:global-rigidity}. In particular,
\(\psi(0)=0\) and \(\psi\) is independent of \(y_m\). Put
\[
 \mathfrak q=4m+2,
 \qquad
 D_R=\Omega_\psi\cap Q_{M,R}.
\]
Thus \(\mathfrak q\) is the homogeneous dimension of space-time associated
with the dilations \(\delta_r\). For \(\gamma\geq0\), let
\(\mathscr H_\gamma(\Omega_\psi)\) be the vector space of all
\(U\in C(\overline{\Omega_\psi})\) such that
\begin{equation}
\label{eq:Hgamma-definition}
 \mathcal KU=0\quad\text{in }\Omega_\psi,
 \qquad
 U=0\quad\text{on }\partial\Omega_\psi,
 \qquad
 \sup_{D_R}|U|\leq C_U(1+R)^\gamma,
 \qquad R\geq1,
\end{equation}
where \(C_U<\infty\) may depend on \(U\) and on the fixed exponent
\(\gamma\), but not on \(R\). Solutions are understood in the local energy sense defined in
Section~\ref{sec:geometry}. Elements of
\(\mathscr H_\gamma(\Omega_\psi)\) are not assumed to be non-negative. This
is essential for the vector-space argument below, which involves arbitrary
linear combinations of tangential translates.

Since \(\psi\) is independent of \(y_m\),
Proposition~\ref{prop:boundary-energy-zero-extension} applies to every
\(U\in\mathscr H_\gamma(\Omega_\psi)\). In particular,
\[
 \mathcal K\bigl(U^2\mathbf 1_{\Omega_\psi}\bigr)\geq0
 \qquad\text{weakly in }\mathbb R^{2m+1}.
\]
Lemma~\ref{lem:boundary-energy-cylindrical-limit} shows that this conclusion
also holds for the non-negative cylindrical limits produced by the compactness
argument, without requiring corresponding up-to-boundary energy estimates in
the genuinely \(y_m\)-dependent approximating domains. This global
subsolution property is the only boundary-energy input used in the dimension
argument below.

\begin{lemma}
\label{lem:boundary-mean-value}
There exists \(C=C(m,M)\) such that, if
\(U\in\mathscr H_\gamma(\Omega_\psi)\), \(R\geq1\), and
\(0<\varepsilon\leq1\), then
\begin{equation}
\label{eq:boundary-mean-value}
 \sup_{D_R}|U|^2
 \leq
 \frac{C}{(\varepsilon R)^{\mathfrak q}}
 \int_{D_{(1+\varepsilon)R}}|U|^2.
\end{equation}
\end{lemma}

\begin{proof}
Let \(V=U^2\mathbf 1_{\Omega_\psi}\). Proposition
\ref{prop:boundary-energy-zero-extension} gives
\[
 \mathcal KV\geq0
 \qquad\text{weakly in }\mathbb R^{2m+1}.
\]
We now use the standard local boundedness estimate for Kolmogorov subsolutions:
if \(v\geq0\), \(\mathcal Kv\geq0\) weakly in an intrinsic cylinder
\(\mathcal Q^-_{2r}(P)\), then
\begin{equation}
\label{eq:subsolution-local-boundedness}
 v(P)
 \leq
 \frac{C(m)}{r^{\mathfrak q}}
 \int_{\mathcal Q^-_{2r}(P)}v.
\end{equation}
Here, up to fixed inessential constants,
\[
 \mathcal Q^-_r(P)
 =
 P\circ
 \bigl\{(\xi,\eta,s):
 |\xi|<r,\ |\eta|<r^3,\ -r^2<s<0\bigr\}.
\]
Estimate \eqref{eq:subsolution-local-boundedness} follows from the intrinsic
Caccioppoli inequality and Moser iteration, or equivalently from the
mean-value formula for \(\mathcal K\); see
\cite{GarofaloLanconelli1990,LanconelliPolidoro1994}. Its exponent is the
Jacobian exponent \(\mathfrak q=m+3m+2=4m+2\). We apply it to \(V\).

It remains only to choose the cylinder uniformly. From the group law
\eqref{eq:K-group-law} and the definition of \(Q_{M,R}\), there exists
\(c_*=c_*(m,M)\in(0,1)\) such that, for every \(R\geq1\),
\(0<\varepsilon\leq1\), and \(P\in Q_{M,R}\),
\begin{equation}
\label{eq:box-inclusion-mean-value}
 \mathcal Q^-_{2r}(P)\subset Q_{M,(1+\varepsilon)R},
 \qquad r=c_*\varepsilon R.
\end{equation}
Indeed, the increments in \(x,t,y\) are bounded respectively by
\(Cr\), \(Cr^2\), and \(C(r^3+r^2R)\); after decreasing \(c_*\), these are
bounded by the corresponding differences between the radii \(R\) and
\((1+\varepsilon)R\). Combining
\eqref{eq:subsolution-local-boundedness} and
\eqref{eq:box-inclusion-mean-value}, and recalling that \(V=0\) outside
\(\Omega_\psi\), gives
\[
 |U(P)|^2
 \leq
 \frac{C(m,M)}{(\varepsilon R)^{\mathfrak q}}
 \int_{D_{(1+\varepsilon)R}}|U|^2.
\]
Taking the supremum over \(P\in D_R\) proves
\eqref{eq:boundary-mean-value}.
\end{proof}

\begin{proposition}
\label{prop:finite-dimensional-Hgamma}
For every \(\gamma\geq0\),
\begin{equation}
\label{eq:Hgamma-finite-dimensional}
 \dim\mathscr H_\gamma(\Omega_\psi)
 \leq C(m,M)(1+\gamma)^{\mathfrak q}.
\end{equation}
In particular, every linear subspace of
\(\mathscr H_\gamma(\Omega_\psi)\) is finite-dimensional.
\end{proposition}

\begin{proof}
The proof is an intrinsic Kolmogorov analogue of the standard
dimension-counting method for spaces of harmonic functions of polynomial
growth; compare \cite{ColdingMinicozzi1997} and
\cite[Section~3]{Kleiner2010}. As in Kleiner's formulation, we use an
increasing family of local \(\mathrm L^2\)-forms, which becomes positive
definite on each fixed finite-dimensional subspace at all sufficiently large
scales. We then select a scale at which the \(\mathrm L^2\)-mass has
controlled growth, orthonormalize a basis with respect to the
\(\mathrm L^2\)-inner product on the larger domain, and estimate the trace of
the resulting finite-dimensional quadratic form on the smaller domain.
Lemma~\ref{lem:boundary-mean-value} supplies the pointwise-to-\(\mathrm L^2\)
estimate required in the final step, replacing the corresponding elliptic
mean-value estimate.

The space \(\mathscr H_\gamma(\Omega_\psi)\) always contains the zero
function. If it contains no nonzero functions, the assertion is immediate.
More generally, it suffices to bound the dimension of an arbitrary
finite-dimensional subspace $V\subset\mathscr H_\gamma(\Omega_\psi)$ by a constant independent of \(V\). If \(V=\{0\}\), there is nothing to
prove. We may therefore assume that  $d:=\dim V\geq1$. For \(R>0\), define the quadratic form
\[
 A_R(U,Z)=\int_{D_R}UZ,
 \qquad U,Z\in V.
\]
For an arbitrary \(R\), the form \(A_R\) is a priori only positive
semidefinite: although
\[
 A_R(U,U)=\int_{D_R}U^2\geq0,
\]
it remains to show that equality can occur only when \(U\) is the zero
element of \(V\). This is necessary because the subsequent argument uses
\(A_R\) as an inner product on \(V\), in particular to construct
\(A_R\)-orthonormal bases and invertible Gram matrices. Positive
semidefiniteness alone would allow a nonzero element of \(V\) to have zero
\(A_R\)-norm. We shall prove that \(A_R\) is positive definite for all
sufficiently large \(R\). Indeed, exhaustion by the sets \(D_R\) detects
every fixed nonzero function individually, while the finite-dimensionality
of \(V\) ensures that a single sufficiently large set \(D_R\) detects all
nonzero elements of \(V\) simultaneously.

To prove this, let
\[
 K_R
 =
 \left\{
 U\in V:A_R(U,U)=0
 \right\}
 =
 \left\{
 U\in V:U=0\ \text{on }D_R
 \right\}.
\]
Indeed, \(A_R(U,U)=0\) initially implies that \(U=0\) almost everywhere on
\(D_R\), and hence everywhere there by continuity. In fact, hypoellipticity
implies that \(U\) is smooth in the interior of \(\Omega_\psi\). Since \(D_{R_1}\subset D_{R_2}\) whenever \(R_1<R_2\), the family
\(\{K_R\}_{R\geq1}\) is non-increasing:
\[
 R_1<R_2
 \quad\Longrightarrow\quad
 K_{R_2}\subseteq K_{R_1}.
\]
Indeed, if \(R_1<R_2\), then every element that vanishes almost everywhere
on \(D_{R_2}\) also vanishes almost everywhere on \(D_{R_1}\). Moreover,
since
\[
 \bigcup_{R\geq1}D_R=\Omega_\psi,
\]
we have
\[
 \bigcap_{R\geq1}K_R=\{0\}.
\]
Indeed, an element of this intersection vanishes almost everywhere on every
\(D_R\), and hence almost everywhere in \(\Omega_\psi\). Since elements of
\(V\) are continuous, it must vanish identically.

It remains to observe that the non-increasing family \(\{K_R\}_{R\geq1}\)
stabilizes. The function
\[
 R\longmapsto\dim K_R
\]
is non-increasing and takes values in the finite set
\(\{0,\ldots,d\}\). Choose \(R_0\geq1\) such that
\(\dim K_{R_0}\) is minimal. For every \(R\geq R_0\),
\[
 K_R\subseteq K_{R_0}.
\]
Minimality of \(\dim K_{R_0}\) gives
\[
 \dim K_R=\dim K_{R_0},
\]
and therefore
\[
 K_R=K_{R_0},
 \qquad R\geq R_0.
\]
On the other hand, if \(1\leq R<R_0\), then
\(K_{R_0}\subseteq K_R\). Consequently,
\[
 K_{R_0}
 =
 \bigcap_{R\geq1}K_R
 =
 \{0\}.
\]
Thus \(K_R=\{0\}\) for every \(R\geq R_0\), and hence \(A_R\) is positive
definite on \(V\) for every such \(R\). The positive definiteness of \(A_R\) for large \(R\) can equivalently be
seen by compactness, see Remark \ref{rem:positive-definiteness-compactness}.

Fix a basis
\(E_1,\ldots,E_d\) of \(V\), and denote by
\[
G(R)=\left(\int_{D_R}E_iE_j\right)_{i,j=1}^d
\]
the corresponding Gram matrix. For each \(i\), the defining growth condition for
\(\mathscr H_\gamma(\Omega_\psi)\) gives
\[
\sup_{D_R}|E_i|\leq C_i(1+R)^\gamma.
\]
Together with \(|D_R|\leq CR^{\mathfrak q}\), this implies, for \(R\)
sufficiently large,
\[
\int_{D_R}E_i^2\leq C_{i,V}R^{2\gamma+\mathfrak q}.
\]
Hadamard's determinant inequality therefore yields
\begin{equation}
\label{eq:Gram-polynomial-growth}
\det G(R)
\leq
\prod_{i=1}^d\Bigl(\int_{D_R}E_i^2\Bigr)
\leq
C_VR^{d(2\gamma+\mathfrak q)},
\qquad R\geq R_V.
\end{equation}
In the definition of $U\in \mathscr H_\gamma(\Omega_\psi)$, no uniform bound on the individual growth constants \(C_U\) is assumed. For
the fixed finite-dimensional subspace \(V\), these constants enter only
through the finite prefactor \(C_V\) in
\eqref{eq:Gram-polynomial-growth}. They affect the choice of a sufficiently
large good scale but not the resulting dimension bound.

Fix \(0<\varepsilon\leq1\), and set
\[
 \beta=2\gamma+\mathfrak q+1.
\]
We claim that there exist arbitrarily large radii \(R\) such that
\begin{equation}
\label{eq:Gram-good-scale}
 \det G((1+\varepsilon)R)
 \leq
 (1+\varepsilon)^{d\beta}\det G(R).
\end{equation}
To prove this, fix an arbitrary
\(R_*\geq\max\{R_0,R_V\}\) and suppose, to the
contrary, that \eqref{eq:Gram-good-scale} fails for every \(R\geq R_*\).
Writing
\[
 R_k=(1+\varepsilon)^kR_*,
\]
iteration gives
\[
 \det G(R_k)
 >
 (1+\varepsilon)^{kd\beta}\det G(R_*).
\]
Since \(R_*\geq R_0\), the matrix \(G(R_*)\) is positive definite and hence
\(\det G(R_*)>0\). On the other hand,
\eqref{eq:Gram-polynomial-growth} gives
\[
 \det G(R_k)
 \leq
 C_VR_*^{d(2\gamma+\mathfrak q)}
 (1+\varepsilon)^{kd(2\gamma+\mathfrak q)}.
\]
Because
\[
 \beta=2\gamma+\mathfrak q+1,
\]
the two estimates are incompatible as \(k\to\infty\). Thus, for every
\(R_*\geq R_0\), there exists \(R\geq R_*\) satisfying
\eqref{eq:Gram-good-scale}, which proves the claim.

Choose a radius \(R\) satisfying \eqref{eq:Gram-good-scale}, and select an
\(A_{(1+\varepsilon)R}\)-orthonormal basis
\(U_1,\ldots,U_d\) of \(V\). Let
\(\lambda_1,\ldots,\lambda_d\) be the eigenvalues of the Gram matrix
\[
\left(A_R(U_i,U_j)\right)_{i,j=1}^d.
\]
Since \(D_R\subset D_{(1+\varepsilon)R}\), the two quadratic forms satisfy
\[
0<A_R(U,U)\leq A_{(1+\varepsilon)R}(U,U),
\qquad U\in V\setminus\{0\},
\]
and hence \(0<\lambda_j\leq1\). The ratio of the determinants of two Gram
matrices is invariant under a common change of basis. Therefore,  \eqref{eq:Gram-good-scale} implies
\[
\prod_{j=1}^d\lambda_j
=
\frac{\det G(R)}{\det G((1+\varepsilon)R)}
\geq
(1+\varepsilon)^{-d\beta}.
\]
The arithmetic-geometric mean inequality now gives
\begin{equation}
\label{eq:Gram-trace-lower}
\sum_{j=1}^d\int_{D_R}U_j^2
=
\sum_{j=1}^d\lambda_j
\geq
d\Bigl(\prod_{j=1}^d\lambda_j\Bigr)^{1/d}
\geq
d(1+\varepsilon)^{-\beta}.
\end{equation}

We next derive an upper bound for the same quantity. For every \(P\in D_R\),
orthonormality implies
\[
\sum_{j=1}^d(U_j(P))^2
=
\sup_{\substack{U\in V\\
A_{(1+\varepsilon)R}(U,U)=1}}
|U(P)|^2.
\]
Indeed, if \(U=\sum_{j=1}^da_jU_j\), then
\(A_{(1+\varepsilon)R}(U,U)=\sum_j a_j^2\), and the supremum follows from
the Cauchy-Schwarz inequality, with equality for coefficients proportional
to \((U_1(P),\ldots,U_d(P))\).

The separation between \(D_R\) and the artificial boundary of
\(D_{(1+\varepsilon)R}\) is comparable, in the intrinsic geometry, to
\(\varepsilon R\). Lemma~\ref{lem:boundary-mean-value} therefore gives,
uniformly for every \(U\) occurring in the preceding supremum,
\[
|U(P)|^2
\leq
C(\varepsilon R)^{-\mathfrak q}
\int_{D_{(1+\varepsilon)R}}U^2
=
C(\varepsilon R)^{-\mathfrak q}.
\]
Consequently,
\[
\sum_{j=1}^d(U_j(P))^2
\leq
C(\varepsilon R)^{-\mathfrak q},
\qquad P\in D_R.
\]
Integrating over \(D_R\) and using
\(|D_R|\leq CR^{\mathfrak q}\), we obtain
\[
\sum_{j=1}^d\int_{D_R}U_j^2
\leq
C\varepsilon^{-\mathfrak q}.
\]
Combining this with \eqref{eq:Gram-trace-lower} yields
\[
d
\leq
C\varepsilon^{-\mathfrak q}(1+\varepsilon)^\beta.
\]

Finally, take
\[
\varepsilon=\min\{1,\beta^{-1}\}.
\]
Since \(\beta=2\gamma+\mathfrak q+1>1\), our choice gives
\(\varepsilon=\beta^{-1}\). Hence,  using \(\log(1+s)\leq s\) we deduce \((1+\varepsilon)^\beta\leq C\) and
\(\varepsilon^{-\mathfrak q}\leq C(1+\beta)^{\mathfrak q}\). Since
\(\beta=2\gamma+\mathfrak q+1\), it follows that
\[
d\leq C(1+\gamma)^{\mathfrak q},
\]
which is precisely \eqref{eq:Hgamma-finite-dimensional}. The estimate is
independent of the chosen subspace \(V\), hence the full
space \(\mathscr H_\gamma(\Omega_\psi)\) is finite-dimensional and satisfies
the same bound.
\end{proof}

\begin{remark}
\label{rem:individual-growth-constants}
No uniform bound on the constants \(C_U\) in
\eqref{eq:Hgamma-definition} is required. For a fixed finite-dimensional
subspace \(V\subset\mathscr H_\gamma(\Omega_\psi)\), choosing a basis
produces a finite constant \(C_V\) in
\eqref{eq:Gram-polynomial-growth}. This constant may affect how large the
good scale in \eqref{eq:Gram-good-scale} must be, but it does not enter the
resulting dimension estimate, which depends only on the competing growth
exponents. This independence is necessary: multiplying an element \(U\) by
an arbitrary scalar changes its growth constant without changing either its
membership in \(\mathscr H_\gamma(\Omega_\psi)\) or the dimension of the
space.
\end{remark}

\begin{remark}
\label{rem:positive-definiteness-compactness}
The positive definiteness of \(A_R\) for large \(R\) can equivalently be
seen by compactness. Assume that \(V\neq\{0\}\), fix any norm on \(V\), and
let
\[
 S_V=\{U\in V:\|U\|_V=1\}.
\]
For every \(U\in S_V\), exhaustion by the sets \(D_R\) gives a radius
\(R_U\) such that $A_{R_U}(U,U)>0$. By continuity of \(Z\mapsto A_{R_U}(Z,Z)\), the same strict inequality
holds in a neighborhood \(\mathcal O_U\) of \(U\) in \(S_V\). Since \(V\)
is finite-dimensional, its unit sphere \(S_V\) is compact,  and hence finitely many such
neighborhoods,
\[
 \mathcal O_{U_1},\ldots,\mathcal O_{U_N},
\]
cover \(S_V\). Setting
\[
 \widetilde R_0
 =
 \max_{1\leq i\leq N}R_{U_i},
\]
and using the monotonicity of \(A_R(Z,Z)\) in \(R\), we obtain
\[
 A_{\widetilde R_0}(Z,Z)>0
 \qquad\text{for every }Z\in S_V.
\]
Homogeneity then gives
\[
 A_R(Z,Z)>0
 \qquad
 \text{for every }Z\in V\setminus\{0\}
 \quad\text{and every }R\geq\widetilde R_0.
\]
This is the compactness counterpart of the stabilization argument for the
kernels \(K_R\).
\end{remark}

For \(h\in\mathbb R\), denote translation in the invariant direction by
\begin{equation}
\label{eq:ym-translation}
 (T_hU)(x',x_m,y',y_m,t)
 =U(x',x_m,y',y_m+h,t).
\end{equation}
Exact cylindricity implies that \(T_h\) preserves both \(\Omega_\psi\) and
its boundary, and \(T_h\mathcal K=\mathcal KT_h\).

\begin{proposition}
\label{prop:tangential-exponential-polynomiality}
Let \(U\in\mathscr H_\gamma(\Omega_\psi)\). There exist an integer
\(J\geq1\), real numbers \(\xi_1,\ldots,\xi_J\), non-negative integers
\(d_1,\ldots,d_J\), and complex-valued coefficient functions
\[
 a_{j,k}=a_{j,k}(x',x_m,y',t)
\]
such that
\begin{equation}
\label{eq:tangential-exponential-polynomial}
 U(x',x_m,y',y_m,t)
 =
 \sum_{j=1}^J e^{i\xi_jy_m}
 \sum_{k=0}^{d_j}
 a_{j,k}(x',x_m,y',t)y_m^k.
\end{equation}
The expansion may be chosen invariant under complex conjugation: the terms
with nonzero frequencies occur in conjugate pairs, while the coefficients
associated with the zero frequency are real-valued. Hence the right-hand
side of \eqref{eq:tangential-exponential-polynomial} is real-valued. Moreover,
for every fixed \((x',x_m,y',t)\), the function
\[
 s\longmapsto U(x',x_m,y',s,t)
\]
is real-analytic on \(\mathbb R\); indeed, it extends to an entire function
of the complex variable \(s\).
\end{proposition}

\begin{proof}
Let
\[
V_U=\operatorname{span}\{T_hU:h\in\mathbb R\}.
\]
Since \(\Omega_\psi\) is cylindrical in \(y_m\), translation in the
\(y_m\)-direction preserves both \(\Omega_\psi\) and its boundary.
Moreover, \(\mathcal K\) is invariant under these translations. Hence, if
\(U\) is \(\mathcal K\)-harmonic in \(\Omega_\psi\) and vanishes on
\(\partial\Omega_\psi\), then the same is true of \(T_hU\).

For each fixed \(h\), translation in \(y_m\) changes the intrinsic distance
from the origin by at most a quantity comparable to \(|h|^{1/3}\). Hence
\[
 T_h(D_R)
 \subset
 D_{R+C|h|^{1/3}}
\]
up to a harmless change in the box constants. Consequently,
\[
 \sup_{D_R}|T_hU|
 \leq
 C_U\bigl(1+R+C|h|^{1/3}\bigr)^\gamma
 \leq
 C_{U,h}(1+R)^\gamma,
\]
where \(C_{U,h}\) is independent of \(R\). Thus every \(T_hU\) has the same
polynomial growth order as \(U\). Since
\(\mathscr H_\gamma(\Omega_\psi)\) is a vector space, it follows that
\[
 V_U
 =
 \operatorname{span}\{T_hU:h\in\mathbb R\}
 \subset
 \mathscr H_\gamma(\Omega_\psi).
\]
Proposition~\ref{prop:finite-dimensional-Hgamma} therefore implies that
\(V_U\) is finite-dimensional.

The space \(V_U\) is invariant under every translation \(T_h\), and
\[
 T_0=I,
 \qquad
 T_{h+k}=T_hT_k,
 \qquad
 T_h^{-1}=T_{-h}.
\]
Thus every restriction \(T_h|_{V_U}\) is an invertible linear map on
\(V_U\). To identify these operators as a continuous one-parameter group and
to estimate their norms, we first introduce a convenient norm on \(V_U\).

Write \(d=\dim V_U\). If \(d=0\), then \(U\equiv0\) and there is nothing to
prove, so assume that \(d\geq1\). We construct points \(P_1,\ldots,P_d\in\Omega_\psi\) such that the
evaluation map
\begin{equation}
\label{map}
 \mathcal E:V_U\longrightarrow\mathbb R^d,
 \qquad
 \mathcal E(W)
 =
 \bigl(W(P_1),\ldots,W(P_d)\bigr),
\end{equation}
is injective.

Set $\mathcal N_0:=V_U$. Assuming \(d=\dim V_U\geq1\), choose a nonzero function
\(W_1\in\mathcal N_0\). Since \(W_1\not\equiv0\), there exists
\(P_1\in\Omega_\psi\) such that $W_1(P_1)\neq0$. Define
\[
 \mathcal N_1
 =
 \{W\in\mathcal N_0:W(P_1)=0\}.
\]
Because evaluation at \(P_1\) defines a nonzero linear functional
\[
 \ell_{P_1}:\mathcal N_0\to\mathbb R,
 \qquad
 \ell_{P_1}(W)=W(P_1),
\]
its image is one-dimensional. The rank-nullity theorem therefore shows
that its kernel has codimension one in \(\mathcal N_0\), i.e.,
\[
 \dim\mathcal N_1=d-1.
\]
We now proceed inductively. Assume inductively that points \(P_1,\ldots,P_k\) have been chosen so that
\[
 \mathcal N_k
 =
 \{W\in V_U:W(P_1)=\cdots=W(P_k)=0\}
\]
satisfies $\dim\mathcal N_k=d-k$. If \(k<d\), choose a nonzero
\(W_{k+1}\in\mathcal N_k\). There exists
\(P_{k+1}\in\Omega_\psi\) such that
\[
 W_{k+1}(P_{k+1})\neq0.
\]
Setting
\[
 \mathcal N_{k+1}
 =
 \{W\in\mathcal N_k:W(P_{k+1})=0\},
\]
we obtain
\[
 \dim\mathcal N_{k+1}=d-k-1.
\]
After \(d\) steps,
\[
 \mathcal N_d=\{0\}.
\]
Equivalently, the evaluation map \(\mathcal E\) in \eqref{map}
has trivial kernel and is therefore injective. Since the map \(\mathcal E\) is linear and injective, and since both \(V_U\) and
\(\mathbb R^d\) have dimension \(d\), the rank-nullity theorem implies that
\(\mathcal E\) is also surjective. Hence
\[
 \mathcal E:V_U\longrightarrow\mathbb R^d
\]
is a linear isomorphism. Consequently,
\[
 \|W\|_*
 =
 \max_{1\leq\ell\leq d}|W(P_\ell)|
\]
defines a norm on \(V_U\). Its non-degeneracy follows from the injectivity of
\(\mathcal E\): if \(\|W\|_*=0\), then \(\mathcal E(W)=0\), and hence
\(W=0\). Since \(V_U\) is finite-dimensional, \(\|\cdot\|_*\) is equivalent
to every other norm on \(V_U\). This norm is particularly convenient because
estimates for \(T_hW\) reduce to estimates at the finitely many fixed points
\(P_1,\ldots,P_d\).

We now verify continuity of the translation representation. For every
\(W\in V_U\), continuity of \(W\) gives
\[
 (T_hW)(P_\ell)\longrightarrow W(P_\ell)
 \qquad\text{as }h\to0,
 \qquad \ell=1,\ldots,d.
\]
The translated points remain in \(\Omega_\psi\) because the domain is
invariant under translations in \(y_m\). Therefore,
\[
 \|T_hW-W\|_*
 =
 \max_{1\leq\ell\leq d}
 |(T_hW)(P_\ell)-W(P_\ell)|
 \longrightarrow0
 \qquad\text{as }h\to0.
\]
Continuity at an arbitrary \(h_0\in\mathbb R\) follows from
\[
 T_hW-T_{h_0}W
 =
 T_{h_0}\bigl(T_{h-h_0}W-W\bigr).
\]
Thus \(h\mapsto T_hW\) is continuous for every \(W\in V_U\). In a
finite-dimensional space, continuity on each vector implies continuity in
operator norm, as can be seen by applying the preceding argument to a fixed
basis of \(V_U\). Hence
\[
 h\longmapsto T_h|_{V_U}
\]
is a continuous one-parameter group in \(\operatorname{GL}(V_U)\).

The standard theorem for continuous one-parameter matrix groups now gives a
unique linear map \(B:V_U\to V_U\) such that
\begin{equation}
\label{eq:finite-dimensional-translation-group}
 T_h|_{V_U}=e^{hB},
 \qquad h\in\mathbb R.
\end{equation}
Indeed, after choosing a basis of \(V_U\), the operators \(T_h\) form a
continuous matrix-valued homomorphism from \((\mathbb R,+)\) into
\(\operatorname{GL}(d,\mathbb R)\), and every such homomorphism is generated
by a matrix exponential. Its generator is
\[
 B
 =
 \left.\frac{\dd}{\dd h}T_h\right|_{h=0}.
\]
Differentiating the group identity gives
\[
 \frac{\dd}{\dd h}T_h
 =
 BT_h
 =
 T_hB,
 \qquad
 T_0=I,
\]
whose unique solution is \(T_h=e^{hB}\). At interior points, where
hypoellipticity makes the elements of \(V_U\) smooth, the generator agrees
with differentiation in the translation direction: $BW=\partial_{y_m}W$. This is the finite-dimensional translation-invariant-space mechanism
underlying the classical structure theorem of Anselone and Korevaar
\cite{AnseloneKorevaar1964}. The generator \(B\) need not be diagonalizable;
its possible Jordan blocks will account for the polynomial factors appearing
below. We next use the evaluation norm \(\|\cdot\|_*\) to transfer the
pointwise polynomial-growth estimates for elements of \(V_U\) to a
two-sided polynomial bound for the operator group \(e^{hB}\). That bound
will exclude eigenvalues of \(B\) with nonzero real part and force its
spectrum onto the imaginary axis.

Choose a basis of \(V_U\) consisting of translates of \(U\), say
\[
W_j=T_{s_j}U,\qquad j=1,\ldots,d.
\]
For every \(h\in\mathbb R\),
\[
(T_hW_j)(P_\ell)
=
(T_{h+s_j}U)(P_\ell).
\]
The point obtained from \(P_\ell\) by a displacement of size
\(|h+s_j|\) in the \(y_m\)-direction lies in an intrinsic box of radius at
most
\[
C\bigl(1+|h+s_j|^{1/3}\bigr)
\leq
C_j\bigl(1+|h|^{1/3}\bigr).
\]
The growth assumption on \(U\) therefore yields
\[
\|T_hW_j\|_*
\leq
C_j\bigl(1+|h|^{1/3}\bigr)^\gamma
\leq
C_j(1+|h|)^{\gamma/3}.
\]
Since \(W_1,\ldots,W_d\) is a basis and all norms on \(V_U\) are
equivalent, this implies
\begin{equation}
\label{eq:translation-polynomial-growth}
\|T_h|_{V_U}\|
=
\|e^{hB}\|
\leq
C(1+|h|)^{\gamma/3},
\qquad h\in\mathbb R,
\end{equation}
for any fixed norm on \(V_U\).

The space \(V_U\) is a finite-dimensional real vector space. A real linear
map need not have real eigenvalues or a basis of real eigenvectors, so we pass
to the complexification
\[
 V_{U,\mathbb C}
 =
 V_U\otimes_{\mathbb R}\mathbb C
 =
 \{W_1+iW_2:W_1,W_2\in V_U\}.
\]
The original space \(V_U\) is identified with the real subspace
\(\{W+i0:W\in V_U\}\). The real linear map \(B:V_U\to V_U\) extends
uniquely to a complex-linear map
\[
 B_{\mathbb C}:V_{U,\mathbb C}\longrightarrow V_{U,\mathbb C},
 \qquad
 B_{\mathbb C}(W_1+iW_2)=BW_1+iBW_2.
\]
Similarly, \(T_h=e^{hB}\) extends to
\[
 T_h^{\mathbb C}=e^{hB_{\mathbb C}}.
\]
For example, if \(V_{U,\mathbb C}\) is equipped with the norm
\[
 \|W_1+iW_2\|_{\mathbb C}
 =
 \|W_1\|+\|W_2\|,
\]
then the polynomial bound \eqref{eq:translation-polynomial-growth} passes
directly from \(T_h\) to \(T_h^{\mathbb C}\). Since all norms on the
finite-dimensional space \(V_{U,\mathbb C}\) are equivalent, the same bound
holds for any chosen norm, up to a change in its constant. For simplicity,
we henceforth write \(B\) and \(T_h\) also for their complexifications.

Let \(\lambda\in\mathbb C\) be an eigenvalue of \(B\), and let \(Z\ne0\) be
a corresponding eigenvector. Since \(B^kZ=\lambda^kZ\), the exponential
series gives
\[
 e^{hB}Z=e^{h\lambda}Z.
\]
Thus \(e^{h\lambda}\) is an eigenvalue of \(e^{hB}\). If the exponent in
\eqref{eq:translation-polynomial-growth} is denoted by \(L\), then
\[
 e^{h\operatorname{Re}\lambda}\|Z\|
 =
 \|e^{hB}Z\|
 \leq
 C(1+|h|)^L\|Z\|,
 \qquad h\in\mathbb R.
\]
If \(\operatorname{Re}\lambda>0\), the left-hand side grows exponentially
as \(h\to+\infty\), contradicting the polynomial bound. If
\(\operatorname{Re}\lambda<0\), the same contradiction is obtained by
letting \(h\to-\infty\). The fact that the bound holds for both positive and
negative \(h\) is essential here. Consequently, every eigenvalue of \(B\) is
purely imaginary:
\[
 \lambda=i\xi,
 \qquad
 \xi\in\mathbb R.
\]

Let \(d=\dim V_U\), and let
\(i\xi_1,\ldots,i\xi_J\) be the distinct eigenvalues of \(B\). The generalized-eigenspace decomposition, obtained by grouping together
the Jordan blocks associated with the same eigenvalue, gives
\[
 V_{U,\mathbb C}
 =
 \bigoplus_{\nu=1}^J E_\nu,
 \qquad
 E_\nu
 =
 \ker\bigl(B-i\xi_\nu I\bigr)^d.
\]
On \(E_\nu\), define
\[
 N_\nu
 =
 \bigl(B-i\xi_\nu I\bigr)|_{E_\nu}.
\]
The operator \(N_\nu\) is nilpotent. If \(q_\nu\) is its nilpotency index,
that is, the smallest positive integer such that
\(N_\nu^{q_\nu}=0\), then
\[
 E_\nu
 =
 \ker\bigl(B-i\xi_\nu I\bigr)^{q_\nu}
\]
and
\[
 B|_{E_\nu}=i\xi_\nu I+N_\nu,
 \qquad
 N_\nu^{q_\nu}=0.
\]
Thus the restriction of \(B\) to \(E_\nu\) consists of the scalar part
\(i\xi_\nu I\), which produces oscillation, and the nilpotent part
\(N_\nu\), which produces polynomial factors. Since these two operators
commute,
\[
 e^{hB}|_{E_\nu}
 =
 e^{i\xi_\nu h}e^{hN_\nu}.
\]
The exponential series for \(N_\nu\) terminates because \(N_\nu\) is
nilpotent, and hence
\[
 e^{hB}|_{E_\nu}
 =
 e^{i\xi_\nu h}
 \sum_{k=0}^{q_\nu-1}\frac{h^kN_\nu^k}{k!}.
\]

Let \(P_\nu\) be the projection onto \(E_\nu\) along the remaining
generalized eigenspaces. Thus
\[
 I=\sum_{\nu=1}^JP_\nu,
 \qquad
 P_\nu P_\mu=0\quad\text{if }\nu\ne\mu.
\]
Applying the preceding formula to each component \(P_\nu Z\) of a vector
\(Z\in V_{U,\mathbb C}\) gives the operator identity
\begin{equation}
\label{eq:translation-group-expansion}
 e^{hB}
 =
 \sum_{\nu=1}^J
 e^{i\xi_\nu h}
 \sum_{k=0}^{q_\nu-1}h^kC_{\nu,k},
\end{equation}
where
\[
 C_{\nu,k}
 =
 \frac1{k!}N_\nu^kP_\nu.
\]
Here \(P_\nu\) first maps into \(E_\nu\), so the composition
\(N_\nu^kP_\nu\) is a well-defined complex-linear map on all of
\(V_{U,\mathbb C}\).

We now translate this finite-dimensional operator identity into a formula
for the original solution. Recall that \(T_h=e^{hB}\) and
\[
 (T_hU)(x',x_m,y',y_m,t)
 =
 U(x',x_m,y',y_m+h,t).
\]
Identifying \(U\) with its natural image \(U+i0\) in \(V_{U,\mathbb C}\),
we apply \eqref{eq:translation-group-expansion} to \(U\) and evaluate at
\(y_m=0\). This gives
\begin{align*}
 U(x',x_m,y',h,t)
 &=
 (T_hU)(x',x_m,y',0,t)=
 \sum_{\nu=1}^J
 e^{i\xi_\nu h}
 \sum_{k=0}^{q_\nu-1}
 h^k(C_{\nu,k}U)(x',x_m,y',0,t).
\end{align*}
The quantities in the last factor do not depend on \(h\). Defining
\[
 a_{\nu,k}(x',x_m,y',t)
 =
 (C_{\nu,k}U)(x',x_m,y',0,t)
\]
and setting \(h=y_m\), we obtain
\begin{equation*}
 U(x',x_m,y',y_m,t)
 =
 \sum_{\nu=1}^{J}
 e^{i\xi_\nu y_m}
 \sum_{k=0}^{q_\nu-1}
 a_{\nu,k}(x',x_m,y',t)y_m^k.
\end{equation*}
The frequencies \(\xi_\nu\) and the integers \(q_\nu\) are determined by the
finite-dimensional translation representation generated by \(U\), whereas
the coefficients \(a_{\nu,k}\) depend on the remaining variables.

The coefficients \(a_{\nu,k}\) are generally complex-valued. This is only
an artifact of passing from the real space \(V_U\) to its complexification.
Indeed, since \(B\) is the complexification of a real linear operator, it
commutes with complex conjugation:
\[
 B\overline Z=\overline{BZ}.
\]
Consequently, the generalized eigenspaces associated with non-real
eigenvalues occur in conjugate pairs. More precisely, for every index
\(\nu\) with \(\xi_\nu\ne0\), there is an index \(\bar\nu\) such that
\[
 \xi_{\bar\nu}=-\xi_\nu,
 \qquad
 q_{\bar\nu}=q_\nu,
 \qquad
 a_{\bar\nu,k}=\overline{a_{\nu,k}}.
\]
The corresponding pair of terms in
\eqref{eq:tangential-exponential-polynomial} can therefore be written as
\[
 2\operatorname{Re}
 \left(
 e^{i\xi_\nu y_m}
 \sum_{k=0}^{q_\nu-1}
 a_{\nu,k}(x',x_m,y',t)y_m^k
 \right).
\]
The coefficients associated with the zero frequency may be chosen
real-valued. Hence the full expansion is real-valued whenever \(U\) is
real-valued.

Finally, fix \((x',x_m,y',t)\) such that the corresponding \(y_m\)-fibre
lies in \(\Omega_\psi\). This entire fibre is contained in the domain because
\(\psi\) is independent of \(y_m\). Formula
\eqref{eq:tangential-exponential-polynomial} shows that
\[
 s\longmapsto U(x',x_m,y',s,t)
\]
is a finite sum of functions of the form $s^ke^{i\xi s}$. Each such function extends to the entire function
\(z\mapsto z^ke^{i\xi z}\) on \(\mathbb C\). Thus every \(y_m\)-fibre of
\(U\) is real-analytic, indeed the restriction of an entire function. Notice
that this conclusion does not require the frequencies \(\xi_\nu\) to vanish:
the rigidity argument uses analyticity of the fibres, not polynomial
dependence on \(y_m\).
\end{proof}

The preceding proposition gives exactly the tangential regularity needed in
the rigidity argument. It does not assert that all frequencies \(\xi_j\) in
\eqref{eq:tangential-exponential-polynomial} vanish. Ruling out the
oscillatory modes \(e^{i\xi_jy_m}\) with \(\xi_j\ne0\) would be necessary to
deduce polynomial dependence on \(y_m\), but no such classification is
required here.

We next establish the uniqueness statement used after fibrewise analyticity
has extended the one-sided zero set to the full fibres. Since this uniqueness
argument does not use the cylindrical structure of the domain, we formulate
it in a more general setting.

\subsection{Cauchy uniqueness and completion of the rigidity argument}
\label{subsec:cauchy-uniqueness}

\begin{lemma}
\label{lem:polynomial-growth-cauchy-uniqueness}
Let \(\Omega_\psi\) be an unbounded intrinsic
\(\operatorname{Lip}_{\mathcal K}\) graph domain with constant \(M\),
possibly depending on \(y_m\). No normalization of \(\psi\) at the origin is
assumed. Let \(\tau\in\mathbb R\), and suppose that
\(Z\in C(\overline{\Omega_\psi\cap\{t\geq\tau\}})\) is an energy solution of
\begin{equation}
\label{eq:cauchy-uniqueness-problem}
 \mathcal KZ=0\quad\text{in }\Omega_\psi\cap\{t>\tau\},
 \qquad
 Z=0\quad\text{on }
 \bigl(\partial\Omega_\psi\cap\{t\geq\tau\}\bigr)
 \cup
 \bigl(\overline{\Omega_\psi}\cap\{t=\tau\}\bigr).
\end{equation}
Assume that, for some \(\gamma\geq0\) and every \(T>\tau\),
\begin{equation}
\label{eq:cauchy-uniqueness-growth}
 \sup_{\Omega_\psi\cap Q_{M,R}\cap\{\tau\leq t\leq T\}}|Z|
 \leq C_T(1+R)^\gamma,
 \qquad R\geq1.
\end{equation}
Then \(Z\equiv0\) in \(\Omega_\psi\cap\{t\geq\tau\}\).
\end{lemma}

\begin{proof}
Fix \(T>\tau\). Throughout the argument, $Q_{M,R}=Q_{M,R}(0)$ denotes the intrinsic box centered at the group identity. For \(R\)
sufficiently large that
\[
[\tau,T]\subset(-2R^2,2R^2),
\]
set
\[
G_{R,T}
=
\Omega_\psi\cap Q_{M,R}\cap\{\tau<t<T\}.
\]
Since the boxes \(Q_{M,R}\) exhaust \(\mathbb R^{2m+1}\) as
\(R\to\infty\), every fixed point of
\(\Omega_\psi\cap\{\tau<t<T\}\) belongs to \(G_{R,T}\) for all sufficiently
large \(R\).

We write
\[
\Gamma_{R,T}^{\mathrm{art}}
=
\overline{\Omega_\psi}\cap\partial Q_{M,R}
\cap\{\tau<t<T\}
\]
for the artificial lateral boundary. Since the time faces of \(Q_{M,R}\)
lie outside the fixed slab \(\{\tau<t<T\}\), every point of
\(\Gamma_{R,T}^{\mathrm{art}}\) lies on a spatial face of \(Q_{M,R}\).

Choose an integer \(k\geq1\) such that \(2k>\gamma\), set
\[
q(x,y)=1+|x|^2+|y|^2,
\]
and define
\begin{equation}
\label{eq:tikhonov-barrier}
\Phi_k(x,y,t)
=
e^{\lambda_k(t-\tau)}q(x,y)^k,
\qquad
\lambda_k=2km+4k(k-1)+k+1.
\end{equation}
This is a polynomial-growth version of the classical Tikhonov barrier
construction for parabolic uniqueness; see, for example,
\cite[Chapter~1]{Friedman1964}.
A direct calculation gives
\begin{align*}
\Delta_xq^k
&=
2kmq^{k-1}
+
4k(k-1)|x|^2q^{k-2},\qquad x\cdot\nabla_yq^k
=
2k(x\cdot y)q^{k-1}.
\end{align*}
Since \(q\geq1\) and
\[
2|x\cdot y|
\leq
|x|^2+|y|^2
\leq q,
\]
we obtain
\begin{align*}
\mathcal K\Phi_k
&=
e^{\lambda_k(t-\tau)}
\left(
\Delta_xq^k+x\cdot\nabla_yq^k-\lambda_kq^k
\right)\leq
\bigl(
2km+4k(k-1)+k-\lambda_k
\bigr)\Phi_k
=
-\Phi_k.
\end{align*}
Thus
\begin{equation}
\label{eq:tikhonov-super-solution}
\mathcal K\Phi_k\leq-\Phi_k
\qquad\text{in }\mathbb R^{2m+1}.
\end{equation}

By the definition of \(Q_{M,R}\), there exists
\(c_0=c_0(m,M)>0\) such that
\begin{equation}
\label{eq:box-barrier-lower-bound}
q(x,y)\geq1+c_0R^2
\qquad\text{on }\Gamma_{R,T}^{\mathrm{art}}.
\end{equation}
Indeed, on an \(x\)-face at least one \(x\)-coordinate has size comparable
to \(R\), whereas on a \(y\)-face at least one \(y\)-coordinate has size
comparable to \(R^3\). Consequently, the weakest lower bound is attained
on the \(x\)-faces; on the \(y\)-faces one has the stronger estimate $q(x,y)\geq cR^6$.

We next bound \(Z\) on the artificial boundary. Since
\[
\overline{Q_{M,R}}\subset Q_{M,2R},
\]
the growth assumption \eqref{eq:cauchy-uniqueness-growth}, applied in
\(Q_{M,2R}\), together with the continuity of \(Z\), gives
\begin{equation}
\label{eq:artificial-box-boundary-growth}
|Z|
\leq
C_T'(1+R)^\gamma
\qquad\text{on }\Gamma_{R,T}^{\mathrm{art}},
\end{equation}
where \(C_T'\) is independent of \(R\). Define
\[
a_R
=
\frac{C_T'(1+R)^\gamma}
     {(1+c_0R^2)^k}.
\]
Since \(t\geq\tau\) on the relevant boundary,
\[
e^{\lambda_k(t-\tau)}\geq1.
\]
It follows from \eqref{eq:box-barrier-lower-bound} and
\eqref{eq:artificial-box-boundary-growth} that
\begin{equation}
\label{eq:artificial-boundary-domination}
|Z|\leq a_R\Phi_k
\qquad\text{on }\Gamma_{R,T}^{\mathrm{art}}.
\end{equation}

By the definition of the Kolmogorov boundary,
\(\partial_{\mathcal K}G_{R,T}\) is contained in the union of the initial
section \(t=\tau\), the portion of the graph boundary
\(\partial\Omega_\psi\) contained in the slab, and the appropriate
incoming portion of the artificial lateral boundary
\(\Gamma_{R,T}^{\mathrm{art}}\). The terminal section \(t=T\) does not
belong to \(\partial_{\mathcal K}G_{R,T}\). We control the function on the
entire artificial lateral boundary, and therefore, by
\eqref{eq:cauchy-uniqueness-problem} and
\eqref{eq:artificial-boundary-domination},
\[
Z-a_R\Phi_k\leq0
\qquad\text{on }\partial_{\mathcal K}G_{R,T}.
\]
Moreover, since \(\mathcal KZ=0\), estimate
\eqref{eq:tikhonov-super-solution} gives
\[
\mathcal K(Z-a_R\Phi_k)
=
-a_R\mathcal K\Phi_k
\geq0
\qquad\text{in }G_{R,T}.
\]
The weak maximum principle consequently yields
\begin{equation}
\label{eq:cauchy-upper-comparison}
Z(P)\leq a_R\Phi_k(P),
\qquad P\in G_{R,T}.
\end{equation}

Now fix $P\in\Omega_\psi\cap\{\tau<t<T\}$. Then \(P\in G_{R,T}\) for all sufficiently large \(R\). Since \(P\) is
fixed and \(2k>\gamma\),
\[
a_R\Phi_k(P)
=
C_T'
\frac{(1+R)^\gamma}
     {(1+c_0R^2)^k}
\Phi_k(P)
\longrightarrow0
\qquad\text{as }R\to\infty.
\]
It follows from \eqref{eq:cauchy-upper-comparison} that \(Z(P)\leq0\).
Applying the same argument to \(-Z\) gives \(Z(P)\geq0\). Hence
\[
Z=0
\qquad\text{in }\Omega_\psi\cap\{\tau<t<T\}.
\]
Since \(T>\tau\) was arbitrary and \(Z=0\) on the initial section
\(t=\tau\), we conclude that
\[
Z\equiv0
\qquad\text{in }\Omega_\psi\cap\{t\geq\tau\}.
\]
The proof is complete.
\end{proof}

\begin{theorem}
\label{thm:global-cylindrical-rigidity}
Let \(\Omega_\psi\) and \(W\) satisfy the hypotheses of
Definition~\ref{def:global-rigidity}. If
\[
 W(A^-_{1,\Lambda})=0,
\]
then
\[
 W\equiv0
 \qquad\text{in }\Omega_\psi.
\]
In particular, the global cylindrical rigidity implication
\eqref{eq:global-rigidity-implication} holds.
\end{theorem}

\begin{proof}
Set $P=A^-_{1,\Lambda}$. The hypothesis of the theorem gives $W(P)=0$. Since \(W\geq0\) and \(\mathcal KW=0\), the strong minimum principle implies
that \(W\) vanishes on the closure of the propagation set issuing from \(P\).
This is precisely the vanishing hypothesis used in the second conclusion of
Lemma~\ref{lem:fibre-half-ray}.

Fix \(\bar t<t(P)\) and
\[
 \bar\zeta
 =
 (\bar x',\bar x_m,\bar y',\bar t)
 \in
 \pi(\Omega_\psi\cap\{t=\bar t\}).
\]
Lemma~\ref{lem:fibre-half-ray} provides
\(b=b(P,\bar\zeta)\in\mathbb R\) such that
\[
 W(\bar x',\bar x_m,\bar y',s,\bar t)=0
 \qquad\text{for every }s\geq b.
\]
Thus the function
\[
 s\longmapsto
 W(\bar x',\bar x_m,\bar y',s,\bar t)
\]
vanishes on a half-line. By
Proposition~\ref{prop:tangential-exponential-polynomiality}, this function
is real-analytic on \(\mathbb R\). The identity principle therefore gives
\[
 W(\bar x',\bar x_m,\bar y',s,\bar t)=0
 \qquad\text{for every }s\in\mathbb R.
\]
Since \(\bar t<t(P)\) and \(\bar\zeta\) were arbitrary, it follows that
\[
 W=0
 \qquad\text{in }\Omega_\psi\cap\{t<t(P)\}.
\]

By continuity, this vanishing extends to the time section
\[
 W=0
 \qquad\text{on }
 \overline{\Omega_\psi}\cap\{t=t(P)\}.
\]
Moreover, the boundary condition in
Definition~\ref{def:global-rigidity} gives
\[
 W=0
 \qquad\text{on }
 \partial\Omega_\psi\cap\{t\geq t(P)\}.
\]
Because \(W\geq0\), the global polynomial-growth estimate
\eqref{eq:global-rigidity-growth} also gives the absolute-value growth bound
\eqref{eq:cauchy-uniqueness-growth} on every finite time slab. We may
therefore apply
Lemma~\ref{lem:polynomial-growth-cauchy-uniqueness} with
\[
 \tau=t(P),
 \qquad
 Z=W.
\]
It follows that
\[
 W=0
 \qquad\text{in }\Omega_\psi\cap\{t\geq t(P)\}.
\]
Combining the two time regions yields $W\equiv0$ in $\Omega_\psi$. In particular, $W(A^+_{1,\Lambda})=0$, which proves \eqref{eq:global-rigidity-implication}.
\end{proof}

\section{Reference-point balance and oscillation reduction}
\label{sec:oscillation-reduction}

The global rigidity theorem now completes the compactness argument and yields
uniform comparison of the forward and backward reference values. We then
combine this reference balance with the localized weak comparison principle
to obtain boundary oscillation contraction. An interior contraction and a
nearest-boundary argument complete the proof of
Theorem~\ref{thm:perturbative-BHI}.

\subsection{Uniform reference-point comparison}
\label{subsec:uniform-reference-comparison}

We now complete the compactness argument and record its quantitative
consequence. This is the point at which the global rigidity theorem enters
the local boundary theory.

\begin{proposition}
\label{prop:uniform-reference-comparison}
Under the hypotheses of Lemma~\ref{lem:reference-balance}, the constants
\[
 \kappa=\kappa(m,M)
 \quad\text{and}\quad
 \eta_0=\eta_0(m,M,H)
\]
may be chosen sufficiently small so that there exists
\(H_1=H_1(m,M,H)\geq1\) for which
\begin{equation}
\label{eq:uniform-reference-comparison}
 H_1^{-1}
 \leq
 \frac{w(A^+_{r,\Lambda}(P_0))}
      {w(A^-_{r,\Lambda}(P_0))}
 \leq H_1
\end{equation}
at every boundary point $P_0$ and scale $r$ allowed there. Moreover,
\[
 w>0
 \qquad\text{in }\mathscr D_{\kappa r}(P_0).
\]
Thus Lemma~\ref{lem:reference-balance} holds.
\end{proposition}

\begin{proof}
Fix first a localization constant
\(\kappa_*=\kappa_*(m,M)\in(0,1)\) furnished by the cylindrical
reference-point comparison theorem, decreased so that every box used below
is contained in \(Q_{M,2}\). We may also decrease \(\kappa_*\), by a factor
depending only on \(m\) and \(M\), so that the controlled curves in
Lemma~\ref{lem:controlled-local-chains}, followed by the curves joining
reference points at comparable scales, connect every local forward reference
point to a local backward reference point and then to the fixed backward
reference point at scale one. Hence, for every solution \(w\) satisfying
the outer positivity hypothesis \(w(A^-_{1,\Lambda})>0\), the strong minimum
principle gives
\[
 w(A^\pm_{r,\Lambda}(P))>0
\]
for every boundary point and scale occurring below. Thus all reference
quotients and normalizations in the argument are well defined.

Suppose that the
conclusion is false. Then, for every \(j\geq1\), there exist an intrinsic
graph \(\psi_j\), a non-negative solution \(w_j\), a boundary point
\[
 P_j\in\Delta_{\psi_j}\cap Q_{M,\kappa_*},
\]
and a scale \(0<\rho_j\leq\kappa_*\) such that
\[
 \operatorname{cyl}_{2,0,1}(\psi_j)\leq j^{-1},
 \qquad
 H_{w_j}(0,1)\leq H,
 \qquad
 H_{w_j}(P_j,\rho_j)\longrightarrow\infty.
\]
By \eqref{eq:easy-reference-direction}, only the forward quotient can
diverge. After passing to a subsequence, we therefore have
\begin{equation}
\label{eq:uniform-reference-bad-sequence}
 \frac{w_j(A^+_{\rho_j,\Lambda}(P_j))}
      {w_j(A^-_{\rho_j,\Lambda}(P_j))}
 \longrightarrow\infty.
\end{equation}

Suppose first that \(\inf_j\rho_j>0\). After passing to a subsequence, we
may assume that
\[
 \rho_j\longrightarrow\rho_\infty\in(0,\kappa_*].
\]
In particular, after discarding finitely many indices, there exists
\(\rho_*>0\) such that
\[
 \rho_*\leq\rho_j\leq\kappa_*
 \qquad\text{for every }j.
\]
Translate \(P_j\) to the origin and define
\begin{equation}
\label{eq:fixed-scale-normalization}
 \widetilde w_j(P)
 =
 \frac{w_j(P_j\circ P)}
      {w_j(A^+_{\rho_j,\Lambda}(P_j))},
 \qquad
 P\in\widetilde G_j,
\end{equation}
where
\[
 \widetilde\Omega_j
 =
 P_j^{-1}\circ\Omega_{\psi_j},
 \qquad
 \widetilde G_j
 =
 P_j^{-1}\circ
 \bigl(\Omega_{\psi_j}\cap Q_{M,2}\bigr).
\]
The normalization is well defined because the reference values are
positive. Moreover,
\[
 \widetilde w_j(A^+_{\rho_j,\Lambda})=1,
\]
whereas \eqref{eq:uniform-reference-bad-sequence} gives
\[
 \widetilde w_j(A^-_{\rho_j,\Lambda})
 =
 \frac{w_j(A^-_{\rho_j,\Lambda}(P_j))}
      {w_j(A^+_{\rho_j,\Lambda}(P_j))}
 \longrightarrow0.
\]
Because \(P_j\in\Delta_{\psi_j}\cap Q_{M,\kappa_*}\), the constant
\(\kappa_*\) may be chosen, depending only on \(m\) and \(M\), so that
there is a fixed translated box \(\mathcal Q\) whose closure lies in the
original coordinate window and which contains all the reference points,
Harnack chains, and fixed geometric enlargements used below. The assumption
\[
 \operatorname{cyl}_{2,0,1}(\psi_j)\leq j^{-1}
\]
then implies that the cylindrical defects of the translated graphs tend to
zero uniformly on \(\mathcal Q\). Proposition~\ref{prop:graph-compactness}
therefore gives, after passing to a subsequence, convergence on
\(\mathcal Q\) to a graph domain \(\widetilde\Omega_\infty\) whose defining
function is independent of \(y_m\). Only this local cylindricality is
needed in the fixed-scale argument. Proposition~\ref{prop:solution-compactness}, together with
\eqref{eq:one-sided-Carleson} and the normalization at
\(A^+_{\rho_j,\Lambda}\), gives a non-negative limiting solution
\(\widetilde w_\infty\) in \(\widetilde\Omega_\infty\cap\mathcal Q\).
Since \(\rho_j\to\rho_\infty>0\), the reference points remain uniformly
separated from the graph, and local uniform convergence yields
\[
 \widetilde w_\infty(A^+_{\rho_\infty,\Lambda})=1,
 \qquad
 \widetilde w_\infty(A^-_{\rho_\infty,\Lambda})=0.
\]

It remains to verify the outer reference-balance hypothesis needed to apply
the cylindrical comparison theorem to \(\widetilde w_\infty\). Let
\[
 \widehat P_j=P_j^{-1}.
\]
This is the image of the original boundary point \(0\) under left
translation by \(P_j^{-1}\), and
\(\widehat P_j\in\partial\widetilde\Omega_j\). The original scale-one
reference points become
\[
 A^\pm_{1,\Lambda}(\widehat P_j)
 =
 P_j^{-1}\circ A^\pm_{1,\Lambda}(0).
\]
Multiplication of a solution by a positive constant does not change its
reference imbalance. Hence the hypothesis \(H_{w_j}(0,1)\leq H\) gives
\[
 \mathfrak h_{1,\Lambda}
 (\widetilde w_j;\widehat P_j)
 =
 \mathfrak h_{1,\Lambda}(w_j;0)
 \leq H.
\]
The points \(P_j\) remain in a fixed compact boundary portion, while
\(\rho_j\in[\rho_*,\kappa_*]\). Consequently, all the scale-\(\rho_j\) and
scale-one reference points occurring above lie in fixed compact
non-tangential regions. The relevant Harnack chains, taken in their
admissible directions, have uniformly bounded length and uniform clearance
from the graph. These chains, together with the one-sided Carleson estimate,
the outer imbalance bound, and the normalization
\[
 \widetilde w_j(A^+_{\rho_j,\Lambda})=1,
\]
give constants \(c,C>0\), independent of \(j\), such that
\[
 c
 \leq
 \widetilde w_j(A^\pm_{1,\Lambda}(\widehat P_j))
 \leq C.
\]
Here \(c\) and \(C\) may depend on \(m,M,H\), and \(\rho_*\), which is
harmless in the present fixed-scale contradiction. After passing to a further subsequence, we may assume that
\[
 \widehat P_j\longrightarrow\widehat P_\infty
 \in\partial\widetilde\Omega_\infty.
\]
Local uniform convergence at the scale-one reference points gives
\[
 0<c
 \leq
 \widetilde w_\infty
 \bigl(A^\pm_{1,\Lambda}(\widehat P_\infty)\bigr)
 \leq C.
\]
Moreover, passing to the limit in the outer imbalance bound yields
\[
 \mathfrak h_{1,\Lambda}
 (\widetilde w_\infty;\widehat P_\infty)
 \leq H.
\]

The limiting graph is cylindrical in \(y_m\). The localized form of the
one-function comparison theorem
\cite[Theorem~1.1 and Section~6]{NystromPolidoro2016} can therefore be
applied in \(\mathcal Q\); its proof uses the defining function only on a
fixed enlargement of the comparison region. It follows that, for some
finite constant \(C_\infty\),
\[
 \widetilde w_\infty(A^+_{\rho_\infty,\Lambda})
 \leq
 C_\infty
 \widetilde w_\infty(A^-_{\rho_\infty,\Lambda}).
\]
This contradicts
\[
 \widetilde w_\infty(A^+_{\rho_\infty,\Lambda})=1,
 \qquad
 \widetilde w_\infty(A^-_{\rho_\infty,\Lambda})=0.
\]
Consequently, the bad scales cannot remain bounded away from zero.

We may therefore assume, after passing to a subsequence, that
\(\rho_j\to0\). Perform the maximal-scale selection in
\eqref{eq:maximal-selected-scale}, starting from the bad scales \(\rho_j\),
and denote the selected scales by \(r_j\). The selection preserves the
divergence of the forward quotient and gives
\eqref{eq:maximal-reference-growth}. Moreover, \(r_j\to0\); otherwise a
subsequence bounded away from zero would be ruled out by the fixed-scale
argument just completed. Define \(W_j\) by \eqref{eq:secondary-blow-up}. Then
Lemmas~\ref{lem:expanding-growth} and
\ref{lem:secondary-blow-up-limit} apply. After passing to a subsequence, they
give a non-negative \(\mathcal K\)-harmonic function \(W_\infty\) of
polynomial growth in an unbounded cylindrical graph domain, vanishing
continuously on its graph boundary, such that
\[
 W_\infty(A^+_{1,\Lambda})=1,
 \qquad
 W_\infty(A^-_{1,\Lambda})=0.
\]
Because the limiting domain is cylindrical,
Proposition~\ref{prop:boundary-energy-zero-extension} supplies the global
subsolution property used in the rigidity argument; no corresponding
up-to-boundary energy estimate is required in the genuinely
\(y_m\)-dependent approximating domains. The displayed normalization
contradicts Theorem~\ref{thm:global-cylindrical-rigidity}. Thus the
forward quotient is uniformly bounded, while its reciprocal is bounded by
\eqref{eq:easy-reference-direction}. This proves
\eqref{eq:uniform-reference-comparison}.

The positivity of \(w\) in \(\mathscr D_{\kappa r}(P_0)\) follows from the
positivity of the backward reference value, the controlled admissible curves
in Lemma~\ref{lem:controlled-local-chains}, and the strong minimum principle.
After setting \(\kappa=\kappa_*\) and choosing \(\eta_0\) sufficiently small,
the proof is complete.
\end{proof}

\subsection{Localized weak comparison}
\label{subsec:localized-weak-comparison}

Having established uniform balance of the forward and backward reference
values, we now record the local comparison principle that converts this
balance into pointwise control of quotients near the graph. Unlike
Proposition~\ref{prop:uniform-reference-comparison}, this result requires
neither exact cylindricity nor smallness of the \(y_m\)-dependence.

\begin{lemma}
\label{lem:localized-weak-comparison}
Let \(M\geq1\). There exist constants
\[
 a=a(m,M)\geq2,
 \qquad
 b=b(m,M)\in(0,1),
 \qquad
 K=K(m,M)\geq1
\]
with the following property. Let \(P_0\in\Delta_\psi\), \(s>0\), and suppose
that \(U,V\) are non-negative solutions in
\(\Omega_\psi\cap Q_{M,as}(P_0)\) which vanish continuously on $\Delta_\psi\cap Q_{M,as}(P_0)$ and satisfy
\[
 U(A^\pm_{s,\Lambda}(P_0))>0,
 \qquad
 V(A^\pm_{s,\Lambda}(P_0))>0.
\]
Then \(U,V>0\) in \(\Omega_\psi\cap Q_{M,bs}(P_0)\), and, for every
\(P\) in this set,
\begin{equation}
\label{eq:localized-weak-comparison}
 K^{-1}
 \frac{U(A^-_{s,\Lambda}(P_0))}
      {V(A^+_{s,\Lambda}(P_0))}
 \leq
 \frac{U(P)}{V(P)}
 \leq
 K
 \frac{U(A^+_{s,\Lambda}(P_0))}
      {V(A^-_{s,\Lambda}(P_0))}.
\end{equation}
\end{lemma}

\begin{proof}
By translation and dilation, it is enough to consider \(P_0=0\) and \(s=1\).
The assertion is a box-local reformulation of
\cite[Lemma~9.1]{LitsgardNystrom2022}, applied there with \(v=U\) and
\(u=V\). That lemma is proved for the divergence-form operator
\[
 \nabla_x\!\cdot(A\nabla_x)
 +
 x\cdot\nabla_y-\partial_t
\]
in an arbitrary intrinsic Lipschitz graph domain. The operator considered
here corresponds to \(A=I_m\). In particular, the structural assumption
\cite[(3.4)]{LitsgardNystrom2022}, which includes independence of both \(A\)
and the defining function from \(y_m\), is not among the hypotheses of that
lemma; the authors explicitly note that this assumption is not used up to
and including Lemma~9.1.

The intrinsic balls used in the cited statement and the boxes \(Q_{M,r}\)
used here contain one another after multiplying their radii by constants
depending only on \(m\) and \(M\). Choose the outer radius in the cited lemma
to be a fixed multiple of \(s\), take its comparison scale to be \(s\), and
decrease the inner radius by the fixed geometric factor appearing there.
This produces constants \(a,b,K\), depending only on \(m,M\), and gives
\eqref{eq:localized-weak-comparison}. Positivity in the inner box follows
from the same admissible-curve and Harnack-chain argument used in the cited
lemma. No smallness assumption on the \(y_m\)-dependence is required.
\end{proof}

\subsection{One-step boundary oscillation contraction}
\label{subsec:one-step-contraction}

For \(P_0\in\Delta_\psi\), \(r>0\), and \(u,v\) as in
Theorem~\ref{thm:perturbative-BHI}, write
\[
 D_r(P_0)
 =
 \Omega_\psi\cap Q_{M,r}(P_0),
 \qquad
 \omega(P_0,r)
 =
 \osc_{D_r(P_0)}\frac vu.
\]

\begin{proposition}
\label{prop:one-step-contraction}
Assume the hypotheses of Theorem~\ref{thm:perturbative-BHI}.  Let
\(a,b,K\) be the constants in
Lemma~\ref{lem:localized-weak-comparison}, and let \(H_1\) be the constant
in Proposition~\ref{prop:uniform-reference-comparison}.  Set
\begin{equation}
\label{eq:theta-explicit}
 A=\frac ab,
 \qquad
 \rho=\frac ba,
 \qquad
 \theta=1-\frac{1}{2KH_1}.
\end{equation}
Then
\[
 \rho=\rho(m,M)\in(0,1),
 \qquad
 \theta=\theta(m,M,H)\in(0,1),
\]
and
\begin{equation}
\label{eq:one-step-contraction}
 \omega(P_0,\rho r)
 \leq
 \theta\,\omega(P_0,r)
\end{equation}
whenever
\[
 P_0\in\Delta_\psi\cap Q_{M,\kappa/2},
 \qquad
 0<r\leq\frac{\kappa}{2A},
 \qquad
 Q_{M,Ar}(P_0)\subset Q_{M,2}.
\]
\end{proposition}

\begin{proof}
We first verify that the quotient is well defined and bounded on
\(D_r(P_0)\). Apply Lemma~\ref{lem:localized-weak-comparison} to \(u\) and
\(v\) at the scale \(r/b\). Its outer box has radius $ar/b=Ar$, while its inner box has radius  $br/b=r$. The assumptions on \(P_0\) and \(r\) ensure that all the required boxes and
reference points lie in the region where
Proposition~\ref{prop:uniform-reference-comparison} applies. That proposition
gives positivity and uniform balance of the reference values of both \(u\)
and \(v\). Lemma~\ref{lem:localized-weak-comparison} therefore yields
\[
 u,v>0
 \qquad\text{in }D_r(P_0)
\]
and gives upper and lower bounds for \(v/u\) there. In particular,
\(\omega(P_0,r)<\infty\).

Set
\[
 m_r=\inf_{D_r(P_0)}\frac vu,
 \qquad
 M_r=\sup_{D_r(P_0)}\frac vu,
 \qquad
 \omega_r=M_r-m_r.
\]
If \(\omega_r=0\), then \eqref{eq:one-step-contraction} is immediate.
Suppose therefore that \(\omega_r>0\), and define
\begin{equation}
\label{eq:normalized-oscillation-solutions}
 W=\frac{v-m_ru}{\omega_r},
 \qquad
 \widetilde W=u-W
 =
 \frac{M_ru-v}{\omega_r}.
\end{equation}
Both \(W\) and \(\widetilde W\) are non-negative
\(\mathcal K\)-harmonic functions in \(D_r(P_0)\), and both vanish
continuously on the corresponding graph portion. Moreover,
\begin{equation}
\label{eq:normalized-quotient-partition}
 0\leq\frac Wu\leq1,
 \qquad
 0\leq\frac{\widetilde W}{u}\leq1,
 \qquad
 \frac Wu+\frac{\widetilde W}{u}=1
 \quad\text{in }D_r(P_0).
\end{equation}

Let \(s=r/a\). Since \(a\geq2\), the reference points at scale \(s\) and the controlled
curves joining them are contained in \(D_r(P_0)\). At
\(A^-_{s,\Lambda}(P_0)\), at least one of the two quotients in
\eqref{eq:normalized-quotient-partition} is at least \(1/2\).

Suppose first that
\[
 \frac{W(A^-_{s,\Lambda}(P_0))}
      {u(A^-_{s,\Lambda}(P_0))}
 \geq\frac12.
\]
In particular, \(W(A^-_{s,\Lambda}(P_0))>0\). The controlled curve joining
the reference points, together with the strong minimum principle, implies
that
\[
 W(A^+_{s,\Lambda}(P_0))>0.
\]
Thus \(W\) and \(u\) satisfy the positivity hypotheses of
Lemma~\ref{lem:localized-weak-comparison} at scale \(s\). Since
\[
 Q_{M,as}(P_0)=Q_{M,r}(P_0),
\]
the lower inequality in \eqref{eq:localized-weak-comparison}, applied with
\((U,V)=(W,u)\), gives, for every \(P\in D_{bs}(P_0)\),
\begin{align}
 \frac{W(P)}{u(P)}
 &\geq
 K^{-1}
 \frac{W(A^-_{s,\Lambda}(P_0))}
      {u(A^+_{s,\Lambda}(P_0))}=
 K^{-1}
 \frac{W(A^-_{s,\Lambda}(P_0))}
      {u(A^-_{s,\Lambda}(P_0))}
 \frac{u(A^-_{s,\Lambda}(P_0))}
      {u(A^+_{s,\Lambda}(P_0))}\geq
 \frac{1}{2KH_1}.
\label{eq:normalized-lower-alternative}
\end{align}
Here the last inequality follows from
Proposition~\ref{prop:uniform-reference-comparison} applied to \(u\) at the
point \(P_0\) and scale \(s\).

Since
\[
 \frac Wu
 =
 \frac{v/u-m_r}{\omega_r},
\]
inequality \eqref{eq:normalized-lower-alternative} gives
\[
 \frac{v(P)}{u(P)}
 \geq
 m_r+\frac{\omega_r}{2KH_1}
 \qquad\text{in }D_{bs}(P_0).
\]
The upper bound \(v/u\leq M_r\) continues to hold on this smaller set.
Consequently,
\begin{equation}
\label{eq:contraction-first-alternative}
 \osc_{D_{bs}(P_0)}\frac vu
 \leq
 \left(1-\frac{1}{2KH_1}\right)\omega_r.
\end{equation}

If the first alternative fails, then
\[
 \frac{\widetilde W(A^-_{s,\Lambda}(P_0))}
      {u(A^-_{s,\Lambda}(P_0))}
 \geq\frac12.
\]
As above, the controlled reference-point curve and the strong minimum
principle give
\[
 \widetilde W(A^+_{s,\Lambda}(P_0))>0.
\]
Applying the lower inequality in
\eqref{eq:localized-weak-comparison} to
\((U,V)=(\widetilde W,u)\) gives
\[
 \frac{\widetilde W(P)}{u(P)}
 =
 \frac{M_r-v(P)/u(P)}{\omega_r}
 \geq
 \frac{1}{2KH_1}
 \qquad\text{in }D_{bs}(P_0).
\]
It follows that
\begin{equation}
\label{eq:contraction-second-alternative}
 \osc_{D_{bs}(P_0)}\frac vu
 \leq
 \left(1-\frac{1}{2KH_1}\right)\omega_r.
\end{equation}

Finally, $bs=(b/a) r=\rho r$. Combining \eqref{eq:contraction-first-alternative} and
\eqref{eq:contraction-second-alternative} proves
\eqref{eq:one-step-contraction} with \(\rho,\theta\) as in
\eqref{eq:theta-explicit}.
\end{proof}

Iterating Proposition~\ref{prop:one-step-contraction} and using monotonicity
of the oscillation with respect to the radius gives
\begin{equation}
\label{eq:iterated-boundary-oscillation}
 \osc_{D_s(P_0)}\frac vu
 \leq
 C\left(\frac sr\right)^{\sigma_0}
 \osc_{D_r(P_0)}\frac vu,
 \qquad
 0<s\leq r,
 \qquad
 \sigma_0=\frac{\log\theta}{\log\rho}>0,
\end{equation}
for every admissible pair \((P_0,r)\), where
\(C=C(m,M,H)\).

\subsection{Interior oscillation and nearest-boundary patching}
\label{subsec:interior-oscillation}

We next give the interior part of the argument. This step is not a direct
consequence of the interior Harnack inequality alone. Indeed, the Harnack
inequality is time-directed, whereas the oscillation argument requires a
uniform comparison between a future and a past value of the denominator.
That reverse comparison is furnished by
Proposition~\ref{prop:uniform-reference-comparison}.

In the remainder of the proof, we use the symmetrized quasi-distance
\(d_{\mathcal K}^{\mathrm s}\) introduced in
\eqref{eq:symmetric-K-distance}. Put
\[
 \delta(P)
 =
 d_{\mathcal K}^{\mathrm s}(P,\Delta_\psi).
\]
For $P=(x',x_m,y',y_m,t)\in\Omega_\psi$, recall that
\[
 P^\partial
 =
 \bigl(x',\psi(x',y',y_m,t),y',y_m,t\bigr)
 \in\Delta_\psi
\]
is its vertical graph projection. The intrinsic graph condition gives
\begin{equation}
\label{eq:vertical-distance-comparison}
 C^{-1}d_{\mathcal K}^{\mathrm s}(P,P^\partial)
 \leq
 \delta(P)
 \leq
 d_{\mathcal K}^{\mathrm s}(P,P^\partial),
\end{equation}
where \(C=C(m,M)\). Indeed, the group law gives the exact identity
\[
 d_{\mathcal K}^{\mathrm s}(P,P^\partial)
 =
 x_m-\psi(x',y',y_m,t),
\]
while the lower bound in
\eqref{eq:vertical-distance-comparison} follows from the intrinsic Lipschitz
condition and the quasi-triangle inequality.

\begin{lemma}
\label{lem:interior-reverse-comparison}
There exist constants \(c_*=c_*(m,M)\in(0,1)\) and
\(C_*=C_*(m,M,H)\geq1\) with the following property. Let
\(P\in\Omega_\psi\cap Q_{M,\kappa/2}\), set \(P^0=P^\partial\), and suppose
that
\[
 0<s\leq c_*\delta(P).
\]
Define
\[
 P_s^-=P\circ(0,0,-\gamma s^2),
 \qquad
 P_s^+=P\circ(0,0,\gamma s^2),
\]
where \(\gamma=\gamma(m,M)>0\) is chosen sufficiently small. Then
\begin{equation}
\label{eq:interior-reverse-comparison}
 u(P_s^+)
 \leq
 C_*u(P_s^-).
\end{equation}
The same estimate holds with \(u\) replaced by \(v\), or more generally by
any non-negative solution satisfying the hypotheses of
Proposition~\ref{prop:uniform-reference-comparison} with scale-one reference
imbalance bounded by \(H\).
\end{lemma}

\begin{proof}
Let $d=d_{\mathcal K}^{\mathrm s}(P,P^0)$. By \eqref{eq:vertical-distance-comparison}, \(d\) is comparable to
\(\delta(P)\). After decreasing \(c_*\) and \(\kappa\), by factors depending
only on \(m\) and \(M\), all the reference points and admissible curves used
below lie in the region where the solution is defined.

Lemma~\ref{lem:controlled-local-chains} {\rm(iii)} gives an admissible curve
from \(A^+_{c_1d,\Lambda}(P^0)\) to \(P_s^+\), and another from \(P_s^-\) to
\(A^-_{c_1d,\Lambda}(P^0)\). Since admissible curves run toward decreasing
times, these are precisely the orientations required for the two Harnack
inequalities. The same lemma provides Harnack chains of uniformly bounded length, with
radii comparable to \(d\) and uniform clearance from the graph. The
uniformity in \(s\) follows from the restriction \(s\leq c_0d\). Indeed, the
time separation between the endpoints of either chain is  $c_1^2d^2-\gamma s^2$, and the constants are chosen so that
\[
 \frac12c_1^2d^2
 \leq
 c_1^2d^2-\gamma s^2
 \leq
 c_1^2d^2.
\]
The corresponding \(x\)- and \(y\)-displacements are bounded by \(Cd\) and
\(Cd^3\), respectively. After dilation by \(d^{-1}\), all the normalized
paths therefore remain in one fixed compact non-tangential region,
independently of \(s/d\). In particular, the endpoint in \(y_m\) is fixed by
the explicit curve construction and is not left uncontrolled. Covering the
paths by a uniformly bounded number of intrinsic cylinders of radius
comparable to \(d\) gives
\[
 u(P_s^+)
 \leq
 C u(A^+_{c_1d,\Lambda}(P^0)),
 \qquad
 u(A^-_{c_1d,\Lambda}(P^0))
 \leq
 C u(P_s^-),
\]
where \(C=C(m,M)\) is independent of \(s\). Proposition~\ref{prop:uniform-reference-comparison}, applied at the boundary
point \(P^0\) and scale \(c_1d\), gives
\[
 u(A^+_{c_1d,\Lambda}(P^0))
 \leq
 H_1u(A^-_{c_1d,\Lambda}(P^0)).
\]
Combining these three inequalities proves
\eqref{eq:interior-reverse-comparison}. The number and relative sizes of the
Harnack cylinders are independent of \(s/d\); this is why the constant
\(C_*\) remains uniform as \(s\downarrow0\).
\end{proof}

\begin{lemma}
\label{lem:interior-quotient-contraction}
There exist $\tau=\tau(m,M)\in(0,1)$, $\theta_{\mathrm i}=\theta_{\mathrm i}(m,M,H)\in(0,1)$, such that
\begin{equation}
\label{eq:interior-one-step-contraction}
 \osc_{\Omega_\psi\cap Q_{M,\tau s}(P)}\frac vu
 \leq
 \theta_{\mathrm i}
 \osc_{\Omega_\psi\cap Q_{M,s}(P)}\frac vu
\end{equation}
whenever \(P\in\Omega_\psi\cap Q_{M,\kappa/2}\),
\(0<s\leq c_*\delta(P)\), and a fixed geometric enlargement of
\(Q_{M,s}(P)\) is contained in \(\Omega_\psi\cap Q_{M,2}\). Consequently,
for \(0<h\leq s\),
\begin{equation}
\label{eq:iterated-interior-oscillation}
 \osc_{\Omega_\psi\cap Q_{M,h}(P)}\frac vu
 \leq
 C\left(\frac hs\right)^{\sigma_{\mathrm i}}
 \osc_{\Omega_\psi\cap Q_{M,s}(P)}\frac vu,
 \qquad
 \sigma_{\mathrm i}
 =
 \frac{\log\theta_{\mathrm i}}{\log\tau}>0,
\end{equation}
where \(C=C(m,M,H)\).
\end{lemma}

\begin{proof}
Write
\[
 q=\frac vu,
 \qquad
 m=\inf_{\Omega_\psi\cap Q_{M,s}(P)}q,
 \qquad
 M'=\sup_{\Omega_\psi\cap Q_{M,s}(P)}q,
 \qquad
 \omega=M'-m.
\]
The denominator \(u\) is positive in the region under consideration by
Proposition~\ref{prop:uniform-reference-comparison} and the local
Harnack-chain geometry. If \(\omega=0\), then
\eqref{eq:interior-one-step-contraction} is immediate. Suppose that
\(\omega>0\), and define
\[
 W=\frac{v-mu}{\omega},
 \qquad
 \widetilde W=u-W=\frac{M'u-v}{\omega}.
\]
Then \(W\) and \(\widetilde W\) are non-negative and
\(\mathcal K\)-harmonic in \(\Omega_\psi\cap Q_{M,s}(P)\), and
\[
 0\leq\frac Wu\leq1,
 \qquad
 0\leq\frac{\widetilde W}{u}\leq1,
 \qquad
 \frac Wu+\frac{\widetilde W}{u}=1.
\]

At the past point \(P_s^-\), at least one of the two quotients in the last
display is at least \(1/2\). Suppose first that
\[
 \frac{W(P_s^-)}{u(P_s^-)}
 \geq
 \frac12.
\]
Choose \(\tau=\tau(m,M)>0\) sufficiently small. The standard interior
Harnack-chain geometry then places every
\(Z\in Q_{M,\tau s}(P)\) causally between \(P_s^-\) and \(P_s^+\), with both
chains contained in \(Q_{M,s}(P)\) and consisting of a uniformly bounded
number of intrinsic cylinders. The forward Harnack inequality applied to
\(W\) gives
\[
 W(P_s^-)
 \leq
 C W(Z),
\]
while the Harnack inequality applied to \(u\), followed by
Lemma~\ref{lem:interior-reverse-comparison}, gives
\[
 u(Z)
 \leq
 C u(P_s^+)
 \leq
 CC_*u(P_s^-).
\]
Consequently,
\[
 \frac{W(Z)}{u(Z)}
 \geq
 \frac{1}{2C^2C_*}
 =:c_{\mathrm i}>0
 \qquad\text{for }Z\in Q_{M,\tau s}(P).
\]
Since
\[
 \frac Wu=\frac{q-m}{\omega},
\]
we obtain
\[
 q(Z)\geq m+c_{\mathrm i}\omega
 \qquad\text{in }Q_{M,\tau s}(P).
\]
The upper bound \(q\leq M'\) remains valid on the smaller box, and hence
\[
 \osc_{\Omega_\psi\cap Q_{M,\tau s}(P)}q
 \leq
 (1-c_{\mathrm i})\omega.
\]

If instead
\[
 \frac{\widetilde W(P_s^-)}{u(P_s^-)}
 \geq
 \frac12,
\]
the same argument applied to \(\widetilde W\) gives
\[
 \frac{\widetilde W(Z)}{u(Z)}
 =
 \frac{M'-q(Z)}{\omega}
 \geq
 c_{\mathrm i}
 \qquad\text{in }Q_{M,\tau s}(P).
\]
Thus
\[
 q(Z)\leq M'-c_{\mathrm i}\omega
\]
on the smaller box, and the same oscillation contraction follows. Taking $\theta_{\mathrm i}=1-c_{\mathrm i}\in(0,1)$
proves \eqref{eq:interior-one-step-contraction}. Iteration at the scales
\(s,\tau s,\tau^2s,\ldots\), followed by monotonicity of oscillation and
interpolation between consecutive powers of \(\tau\), gives
\eqref{eq:iterated-interior-oscillation}.
\end{proof}

We can now give the nearest-boundary argument.

\begin{proposition}
\label{prop:nearest-boundary-holder}
Under the hypotheses of Theorem~\ref{thm:perturbative-BHI}, there exist
there exist
\[
 r_0=r_0(m,M)\in(0,1),
 \qquad
 \kappa_0=\kappa_0(m,M)\in(0,1),
\]
and
\[
 \sigma
 =
 \min\{\sigma_0,\sigma_{\mathrm i}\}
 =
 \sigma(m,M,H)\in(0,1)
\]
such that
\begin{equation}
\label{eq:local-holder-from-oscillation}
 \left|
 \frac{v(P)}{u(P)}
 -
 \frac{v(\widetilde P)}{u(\widetilde P)}
 \right|
 \leq
 C
 \left(
 \frac{d_{\mathcal K}^{\mathrm s}(P,\widetilde P)}{r_0}
 \right)^\sigma
 \osc_{D_{r_0}(P_0)}\frac vu
\end{equation}
whenever $P,\widetilde P\in\mathscr D_{\kappa_0}(P_0)$. Here \(C=C(m,M,H)\).
\end{proposition}

\begin{proof}
Choose \(r_0=r_0(m,M)\) sufficiently small relative to the geometric
localization constant in Proposition~\ref{prop:uniform-reference-comparison},
and then choose \(\kappa_0=\kappa_0(m,M)\ll r_0\) so that every fixed
enlargement of a box used below is contained in \(Q_{M,2}(P_0)\). All
subsequent reductions of these localization constants are by geometric
factors depending only on \(m\) and \(M\). Put
\[
 q=\frac vu,
 \qquad
 \Omega_0=\osc_{D_{r_0}(P_0)}q,
 \qquad
 h=d_{\mathcal K}^{\mathrm s}(P,\widetilde P),
 \qquad
 d=\delta(P).
\]
After decreasing \(\kappa_0\), both \(P\) and \(\widetilde P\) belong to
\(D_{r_0}(P_0)\). If \(h=0\), there is nothing to prove. If
\(h\geq c_2r_0\), then
\eqref{eq:local-holder-from-oscillation} follows directly from
\[
 |q(P)-q(\widetilde P)|\leq\Omega_0
\]
after increasing \(C\). We may therefore assume that
\[
 0<h<c_2r_0,
\]
where \(c_2=c_2(m,M)>0\) is sufficiently small. Set \(P^*=P^\partial\). By
\eqref{eq:vertical-distance-comparison}, $d_{\mathcal K}^{\mathrm s}(P,P^*)
 \leq
 Cd$. We distinguish two cases.

Suppose first that $h\geq c_3d$, where \(c_3=c_3(m,M)\in(0,1)\). The quasi-triangle inequality and the
comparability between intrinsic balls and boxes give
\[
 P,\widetilde P\in D_{Ch}(P^*).
\]
After decreasing \(c_2\) and \(\kappa_0\), choose a fixed
\(\bar r\asymp r_0\) such that $D_{\bar r}(P^*)\subset D_{r_0}(P_0)$and all the boundary contractions between the scales \(Ch\) and \(\bar r\)
are admissible. Iterating
\eqref{eq:iterated-boundary-oscillation} gives
\[
 \begin{aligned}
 |q(P)-q(\widetilde P)|
 &\leq
 \osc_{D_{Ch}(P^*)}q\leq
 C\left(\frac h{r_0}\right)^{\sigma_0}\Omega_0.
 \end{aligned}
\]

Suppose now instead that  $h<c_3d$. By decreasing \(c_3\), the points \(P\) and \(\widetilde P\) lie in a common
intrinsic box centered at \(P\), and a fixed enlargement of this box remains
inside \(\Omega_\psi\). Choose \(c_4=c_4(m,M)>0\) sufficiently small that
\[
 s=c_4d\leq c_*\delta(P)
\]
and the hypotheses of Lemma~\ref{lem:interior-quotient-contraction} hold at
the scale \(s\). After decreasing \(c_3\) further, we also have
\[
 \widetilde P\in Q_{M,Ch}(P),
 \qquad
 Ch\leq s.
\]
Applying \eqref{eq:iterated-interior-oscillation} from the scale \(s=c_4d\)
down to the scale \(Ch\) gives
\begin{equation}
\label{eq:nearest-boundary-interior-step}
 |q(P)-q(\widetilde P)|
 \leq
 C\left(\frac hd\right)^{\sigma_{\mathrm i}}
 \osc_{\Omega_\psi\cap Q_{M,c_4d}(P)}q.
\end{equation}
The quasi-triangle inequality and the box geometry give
\[
 Q_{M,c_4d}(P)
 \subset
 Q_{M,Cd}(P^*).
\]
After decreasing \(\kappa_0\), the boundary oscillation estimate
\eqref{eq:iterated-boundary-oscillation}, applied at \(P^*\), yields
\[
 \osc_{\Omega_\psi\cap Q_{M,c_4d}(P)}q
 \leq
 C\left(\frac d{r_0}\right)^{\sigma_0}\Omega_0.
\]
Since \(h<d<r_0\) and
\(\sigma=\min\{\sigma_0,\sigma_{\mathrm i}\}\),
\[
 \left(\frac hd\right)^{\sigma_{\mathrm i}}
 \left(\frac d{r_0}\right)^{\sigma_0}
 \leq
 \left(\frac h{r_0}\right)^\sigma.
\]
Combining this estimate with
\eqref{eq:nearest-boundary-interior-step} completes the proof.
\end{proof}

\begin{proposition}
\label{prop:reduction-rigidity}
Theorem~\ref{thm:perturbative-BHI} holds.
\end{proposition}

\begin{proof}
By left translation, it is enough to prove the theorem with its distinguished
boundary point at the origin. All the estimates below are translation
invariant, and we retain the notation \(P_0\) in the formulas so that the
conclusion can be translated back directly. Apply
Proposition~\ref{prop:uniform-reference-comparison} to both \(u\) and
\(v\). Choose a fixed scale \(r_0=r_0(m,M)>0\) and decrease \(\kappa\), by a
factor depending only on \(m\) and \(M\), so that every box required by
Lemma~\ref{lem:localized-weak-comparison} is contained in
\(Q_{M,2}(P_0)\). That lemma gives
\[
 K^{-1}
 \frac{v(A^-_{r_0,\Lambda}(P_0))}
      {u(A^+_{r_0,\Lambda}(P_0))}
 \leq
 \frac{v(P)}{u(P)}
 \leq
 K
 \frac{v(A^+_{r_0,\Lambda}(P_0))}
      {u(A^-_{r_0,\Lambda}(P_0))}
\]
in \(\mathscr D_\kappa(P_0)\), after another decrease of \(\kappa\) by a
geometric factor depending only on \(m\) and \(M\). Uniform reference
balance at the scale \(r_0\), the controlled
Harnack chains joining the three reference points at that scale, and the
chains joining reference points at the fixed comparable scales \(r_0\) and
\(1\), together with \eqref{eq:reference-control-scale-one}, give
\[
 \frac{v(A^\pm_{r_0,\Lambda}(P_0))}
      {u(A^\mp_{r_0,\Lambda}(P_0))}
 \asymp
 \frac{v(A_{1,\Lambda}(P_0))}
      {u(A_{1,\Lambda}(P_0))},
\]
with constants depending only on \(m,M,H\). Consequently,
\[
 C^{-1}
 \frac{v(A_{1,\Lambda}(P_0))}
      {u(A_{1,\Lambda}(P_0))}
 \leq
 \frac{v(P)}{u(P)}
 \leq
 C
 \frac{v(A_{1,\Lambda}(P_0))}
      {u(A_{1,\Lambda}(P_0))}
\]
for \(P\in\mathscr D_\kappa(P_0)\). This is
\eqref{eq:quotient-comparability-scale-one}. The same estimate also gives
\[
 \osc_{D_{r_0}(P_0)}\frac vu
 \leq
 C
 \frac{v(A_{1,\Lambda}(P_0))}
      {u(A_{1,\Lambda}(P_0))}.
\]
Proposition~\ref{prop:nearest-boundary-holder} therefore yields
\[
 \left|
 \frac{v(P)}{u(P)}
 -
 \frac{v(\widetilde P)}{u(\widetilde P)}
 \right|
 \leq
 C
 \bigl(d_{\mathcal K}^{\mathrm s}(P,\widetilde P)\bigr)^\sigma
 \frac{v(A_{1,\Lambda}(P_0))}
      {u(A_{1,\Lambda}(P_0))}
\]
for \(P,\widetilde P\in\mathscr D_\kappa(P_0)\), after decreasing \(\kappa\)
once more by a factor depending only on \(m\) and \(M\). Since
\(d_{\mathcal K}^{\mathrm s}\) and \(d_{\mathcal K}\) are
equivalent, this is \eqref{eq:quotient-holder-scale-one}. The proof is
complete.
\end{proof}

The proof of Theorem~\ref{thm:perturbative-BHI} uses substantially more than ordinary local compactness. The
strong minimum principle alone is too weak for the required conclusion:
because propagation is causal, it forces the limiting function to vanish
only on a one-sided subset of each \(y_m\)-fibre, in fact on a half-ray, and
provides no information on the complementary part of the fibre. Polynomial
growth makes the translation orbit finite-dimensional, while exact
cylindricity converts this finite-dimensionality into real-analytic
dependence on \(y_m\). It is this additional rigidity, rather than the
strong minimum principle by itself, that extends the zero set across the
missing part of every invariant fibre.

\begin{remark}
\label{rem:not-ratio-compactness}
The preceding reduction avoids passing directly to the quotients
\(v_j/u_j\) near the moving boundary. Uniform convergence of \(u_j\) and
\(v_j\) does not, by itself, imply convergence of their quotient where both
functions vanish. The reference-balance estimate supplies precisely the
quantitative lower bound needed before the weak comparison principle and
oscillation reduction can be used.
\end{remark}

\section{Concluding remarks}
\label{sec:conclusion}

This paper shows that exact cylindricality in the distinguished higher-order
variable is not necessary for local boundary comparison. What is required is
that the critical \(C^{0,1/3}_{y_m}\) defect be sufficiently small at the
scale under consideration and remain small under subsequent blow-ups. For
graphs with \(C^{0,\alpha}_{y_m}\) regularity, \(\alpha>1/3\), this occurs
below the threshold in \eqref{eq:threshold-scale}; more generally, it follows
from the vanishing-modulus hypothesis in
Theorem~\ref{thm:general-modulus}. Boundary H\"older decay, the one-sided
Carleson estimate, and localized weak comparison require no cylindricality.
The asymptotic cylindricality assumption enters only in establishing uniform
comparison of the forward and backward reference values.

The global rigidity theorem isolates the mechanism behind this comparison.
A collapsing sequence of bad scales produces a non-negative solution of
polynomial growth in an unbounded cylindrical limit domain. The strong
minimum principle is causal and initially propagates the zero set through
only a half-ray in each invariant \(y_m\)-fibre. Polynomial growth makes the
translation orbit finite-dimensional, while exact cylindricality ensures
that all \(y_m\)-translates remain solutions in the same domain. The resulting
finite-dimensional translation representation gives real-analytic dependence
on \(y_m\), which extends the one-sided zero set to the entire fibre. Thus
exact symmetry is used only in the limiting rigidity problem and is not
imposed on the original domain.

The exponent \(1/3\) is dictated by the homogeneous degree three of \(y_m\).
At the critical exponent, smallness of the scale-invariant defect is
sufficient by Theorem~\ref{thm:perturbative-BHI}, but boundedness without
smallness does not force the blow-up limits to be cylindrical. Whether
boundary comparison holds for arbitrary critical intrinsic Lipschitz
dependence remains open. Resolving this problem would require either a
different rigidity principle or a structural description of the possible
non-cylindrical tangent domains.

The argument also suggests two extensions. First, for divergence-form
Kolmogorov operators with coefficients invariant in \(y_m\), the required
local boundary estimates are already available in rough coefficient form
\cite{LitsgardNystrom2022}; the remaining issue is to establish an appropriate
global rigidity theorem. Second, in higher-step Kolmogorov geometries, one
expects the critical regularity in an invariant variable to be the reciprocal
of its homogeneous degree. In both directions, the expanding-box growth
estimate and the finite-dimensional translation argument provide a framework
for separating local potential theory from the global rigidity of tangent
domains.

The results of this paper also provide the boundary-comparison input in the
regularity program for the Kolmogorov obstacle problem initiated in
\cite{FrentzNystromPascucciPolidoro2010} and continued in
\cite{Bowman2025}. The former work established optimal regularity of the
solution in the intrinsic H\"older scale. Under a quantitative thickness
condition, Bowman subsequently proved the decisive interior estimate $\nabla_yu\in\mathrm L^\infty_{\mathrm{loc}}$ and showed that the regular free boundary is a non-characteristic graph, is
differentiable in the Kolmogorov geometry, and has a unique half-space
blow-up. In the notation of the present paper, its defining function is
\(1/2\)-H\"older continuous in the distinguished variable \(y_m\). Its
cylindrical defect at scale \(r\) is therefore
\[
 O\bigl(r^{3/2-1}\bigr)=O(r^{1/2}),
\]
and the boundary comparison theorem proved here applies at all sufficiently
small scales.

The present theorem supplies the homogeneous boundary Harnack component of
the argument for the free boundary. An additional step is required because the
derivatives determining the normal are not \(\mathcal K\)-harmonic. Indeed,
in the positivity set, where \(\mathcal Ku=1\),
\[
 \mathcal K(\partial_{x_i}u)
 =
 \partial_{x_i}(\mathcal Ku)-\partial_{y_i}u
 =
 -\partial_{y_i}u.
\]
The full treatment of this inhomogeneous problem is carried out in
\cite{NystromHigherRegularity2026}; we describe the mechanism only briefly
here. Under the obstacle-problem blow-up at radius \(r\), the commutator
source in the derivative equation is bounded by \(Cr\). A
\(\mathcal K\)-harmonic replacement, together with the identity
\(\mathcal Ku=1\), transfers the homogeneous boundary oscillation contraction
established here to quotients of the relevant directional derivatives, with
an additive \(O(r)\) error. Iteration yields H\"older continuity of the
non-degenerate normal. A second decay argument for the transport derivative,
whose equation has a divergence-form source of the same scale-decaying order,
gives a power modulus in the intrinsic time direction. Combined with the
enhanced regularity in the higher-order variables obtained in
\cite{Bowman2025}, this yields a quantitative improvement of flatness and
shows that the regular free boundary is locally
\(C_{\mathcal K}^{1,\beta}\) for some \(\beta\in(0,1)\). Thus the present
boundary comparison theorem, the results of \cite{Bowman2025}, and the
inhomogeneous comparison argument in
\cite{NystromHigherRegularity2026} together complete the higher-regularity
program at regular free-boundary points.

\medskip
\noindent
\textbf{Declaration on the use of generative AI.}
The author used ChatGPT, developed by OpenAI, as an interactive tool
during the preparation and revision of this manuscript.  The tool was
used for language editing, organization, LaTeX formatting, and checks of
clarity and internal consistency.  The mathematical arguments, results,
and conclusions were developed and verified by the author, who also
checked the references and takes full responsibility for the content of
the manuscript.

\end{document}